\documentclass{gtpart}
\usepackage[all]{xy}
\usepackage{amssymb, amsfonts, latexsym, amsmath, verbatim, amsthm, amscd,
  xspace, enumerate, etoolbox}
\usepackage{graphicx}
\usepackage{pinlabel}
\usepackage{hyperref}
\usepackage{microtype}

\title{Splitting Line Patterns in Free Groups}

\author{Christopher H. Cashen} 
\givenname{Christopher H.}
\surname{Cashen}
\address{Fakult\"at f\"ur Mathematik\\Universit\"at Wien\\\newline
1090 Wien\\\"{O}sterreich}
\email{christopher.cashen@univie.ac.at}
\urladdr{http://www.mat.univie.ac.at/~cashen}


\arxivreference{1009.2492}  
\arxivpassword{}   

\hypersetup{
    pdftitle={Splitting Line Patterns in Free Groups},    
    pdfauthor={Christopher H. Cashen},     
    pdfkeywords={free group, group splitting, line
  pattern, Whitehead graph, JSJ-decomposition,
  geometric word, virtually geometric multiword, relatively hyperbolic group}, 
    colorlinks=true,       
    linkcolor=black,          
    citecolor=black,        
    filecolor=black,      
    urlcolor=black           
}

\theoremstyle{plain}
\newtheorem{theorem}{Theorem}[section]
\newtheorem{lemma}{Lemma}[section]
\newtheorem{proposition}{Proposition}[section]
\newtheorem{corollary}{Corollary}[section]

\newtheorem*{vgtheorem}{Characterization of Virtual Geometricity}

\theoremstyle{remark}
\newtheorem{claim}{Claim}[theorem]

\newtheorem*{remark}{Remark}

\theoremstyle{definition}
\newtheorem{definition}{Definition}[section]
\newtheorem{example}{Example}[section]

\makeatletter
\@addtoreset{claim}{proposition}
\makeatother
\makeatletter
\@addtoreset{claim}{lemma}
\makeatother

\newcommand{\decompthm}{\hyperlink{decompthmtarget}{Relative JSJ-Decomposition Theorem}\xspace}

\newcommand{\vgthm}{\hyperlink{vgthmtarget}{Characterization of Virtual Geometricity}\xspace}

\AtEndEnvironment{example}{\hfill\ensuremath{\lozenge}}

\newenvironment{claimproof}[1][Proof of Claim]{\vspace{1ex}\noindent{\it #1.}\hspace{0.5em}}
	{\hfill$\lozenge$\vspace{1ex}}

\def\makeautorefname#1#2{\expandafter\def\csname#1autorefname\endcsname{#2}}
\let\fullref\autoref

\makeautorefname{virtualsplittingtheorem}{Virtual Splitting Theorem} 
\makeautorefname{theorem}{Theorem} 
\makeautorefname{lemma}{Lemma} 
\makeautorefname{proposition}{Proposition} 
\makeautorefname{corollary}{Corollary} 
\makeautorefname{definition}{Definition}
\makeautorefname{example}{Example}
\makeautorefname{section}{Section}
\makeautorefname{subsection}{Section}
\makeautorefname{subsubsection}{Section}
\makeautorefname{claim}{Claim}

\makeatletter 
\let\c@lemma=\c@theorem 
\makeatother
\makeatletter 
\let\c@proposition=\c@theorem 
\makeatother
\makeatletter 
\let\c@corollary=\c@theorem 
\makeatother
\makeatletter 
\let\c@definition=\c@theorem 
\makeatother
\makeatletter 
\let\c@example=\c@theorem 
\makeatother

\DeclareMathOperator{\shadow}{Shadow}
\DeclareMathOperator{\Aut}{Aut} 
\DeclareMathOperator{\Wh}{\mathfrak{W}} 
\DeclareMathOperator{\Aug}{Aug} 
\DeclareMathOperator{\Ind}{Ind} 
\def\bdry{\partial}
\def\X{\mathcal{X}} 
\def\Y{\mathcal{Y}}

\def\hull{\mathcal{H}} 
\def\tree{\mathcal{T}} 
\def\lines{\textbf{L}} 
\def\line{\mathcal{L}} 
\def\l{\line}

\def\multimot{\underline{w}} 
\def\mot{w} 
\def\multiklass{[\multimot]}
\newcommand{\klass}{[\langle\mot\rangle]}
\def\indg{\Ind^G_\Gamma(\multimot)}
\def\indgklass{\Ind^G_\Gamma([\multimot])}
\def\we{\mathfrak{E}}
\def\wv{\mathfrak{V}}
\def\wc{\mathfrak{C}}
\newcommand{\gr}{\Theta}
\def\basis{\underline{b}}
\def\basic{b}
\def\F{F}
\def\maxoverlap{M}
\DeclareMathOperator{\nbhd}{N}
\def\augw{\Aug_\Gamma(\multimot)}

\def\decomp{\mathcal{D}} 
\def\decompw{\decomp_{\multimot}}

\def\qmap{\Delta}
\def\augmap{\alpha_\Gamma}
\def\e{\epsilon}
\newcommand{\fingen}[1]{\langle #1\rangle}

\newcommand{\bp}{{\bf 1}}
\newcommand{\wrel}{\mathbin{\circledcirc}}

\makeatletter
\newsavebox\myboxA
\newsavebox\myboxB
\newlength\mylenA

\newcommand*\xoverline[2][0.75]{%
    \sbox{\myboxA}{$\m@th#2$}%
    \setbox\myboxB\null
    \ht\myboxB=\ht\myboxA%
    \dp\myboxB=\dp\myboxA%
    \wd\myboxB=#1\wd\myboxA
    \sbox\myboxB{$\m@th\overline{\copy\myboxB}$}
    \setlength\mylenA{\the\wd\myboxA}
    \addtolength\mylenA{-\the\wd\myboxB}%
    \ifdim\wd\myboxB<\wd\myboxA%
       \rlap{\hskip 0.5\mylenA\usebox\myboxB}{\usebox\myboxA}%
    \else
        \hskip -0.5\mylenA\rlap{\usebox\myboxA}{\hskip 0.5\mylenA\usebox\myboxB}%
    \fi}
\makeatother

\newcommand{\inv}[1]{\xoverline{#1}}

\newcommand{\closure}[1]{\overline{#1}} 

\newlength{\figstandardheight}
\def\from{\colon\thinspace}
\def\onto{\twoheadrightarrow}
\def\into{\hookrightarrow}
\def\rjsj{rJSJ\xspace}

\begin{document}
\begin{abstract}
We construct a boundary of a finite rank free group relative to a
finite list of conjugacy classes of maximal cyclic subgroups.
From the cut points and uncrossed cut pairs of this boundary we
construct a simplicial tree on which the group acts cocompactly. We
show that the quotient graph of groups is the JSJ decomposition of the
group relative to the given collection of conjugacy classes.

This provides a characterization of virtually geometric
multiwords: they are the multiwords that are built from geometric pieces.
In particular, a multiword is virtually geometric if and only if the
relative boundary is planar. 
\end{abstract}

\maketitle

\section{Introduction}

Let $\F=\F_n$ be a free group of finite rank $n>1$.
Let $\multiklass=\{[\left<\mot_1\right>],\dots,[\left<\mot_k\right>]\}$ be
a \emph{multiclass}, a non-empty collection of distinct conjugacy classes of maximal cyclic subgroups.

The goal of this paper is to find splittings of $\F$ \emph{relative to}
  $\multiklass$ (rel $\multiklass$), that is, splittings of $\F$
as a free product or as an amalgam over cyclic subgroups in such a
way that each $[\langle\mot_i\rangle]\in\multiklass$ is elliptic.
We do this by analyzing the topology of a certain relative boundary
$\decomp$ of $\F$, defined as follows:

The free group $\F$ has a well-defined Gromov boundary $\bdry\F$ that
is homeomorphic to a Cantor set.
Left multiplication of $\F$ on itself extends continuously to an
action of $\F$ on $\bdry\F$ by homeomorphisms.
For each non-trivial element $f\in\F$, the $f$--action on $\bdry\F$
has an attracting fixed point $f^\infty$ and a repelling fixed point,
which is the attracting fixed point of $\inv{f}=f^{-1}$, and is
denoted $\inv{f}^\infty$.

\begin{definition}\label{def:boundarypattern}
  Define the \emph{boundary pattern associated to}
  $\multiklass$ to be:
 \[\bdry\multiklass=\{\{w^\infty,\inv{w}^\infty\}\mid
  \left<w\right>\text{ is a maximal cyclic subgroup with }[\left<w\right>]\in\multiklass\}\]
\end{definition}

\begin{definition}\label{def:decompositionspace}
 The \emph{decomposition space
$\decomp=\decomp_{\multiklass}$ of $\bdry\F$ \emph{associated to}
$\multiklass$} is the quotient of $\bdry\F$ obtained by
identifying the two points $\xi_0$ and $\xi_1$ for each
$\{\xi_0,\xi_1\}\in\bdry\multiklass$.
Let $\qmap\from \bdry\F \to \decomp$ be the quotient map.
\end{definition}

The action of $\F$ on $\bdry\F$ preserves the boundary pattern
$\bdry\multiklass$, so it induces an action of $\F$
on $\decomp$ by homeomorphisms.

It is not hard to see that if $\F$ has a free splitting rel
$\multiklass$ then $\decomp$ is not connected.
Similarly, if $\F$ splits over $\langle f\rangle$ rel $\multiklass$,
where $\langle f\rangle$ is a maximal cyclic
subgroup, then
$\decomp\setminus\qmap(\bdry\langle f\rangle)$ is not connected.
If $[\langle f\rangle]\in\multiklass$ then 
$\qmap(\bdry\langle f\rangle)$ is a single point in $\decomp$, so 
there is a point whose removal disconnects $\decomp$.
If $f$ is non-trivial and $[\langle f\rangle]\notin\multiklass$ then 
$\qmap(f^\infty)$ and $\qmap(f^{-\infty})$ are distinct points in $\decomp$, so 
there is a pair of points whose removal disconnects $\decomp$.

We will focus on the case that $\F$ does not split freely rel
$\multiklass$, in which case we will see that $\decomp$ is connected.
The previous paragraph then suggests that we analyze cut points and
cut pairs of $\decomp$.

We show that the cut points and cut pairs of $\decomp$ encode a
simplicial tree on which $\F$ acts cocompactly. 
The quotient graph of groups gives us a canonical decomposition of $\F$
rel $\multiklass$.
The main result, the \decompthm(\fullref{theorem:JSJ}), is that this canonical graph of groups
decomposition of $\F$ obtained from $\decomp$ is the \emph{JSJ decomposition of $\F$ relative
  to $\multiklass$} (the \rjsj).
That is, it is the decomposition that encodes all cyclic splittings of
$\F$ relative to $\multiklass$, and satisfies certain universality and
maximality properties, in the sense of
Guirardel and Levitt \cite{GuiLev09} (see also \cite{KhaMia05}).

In \fullref{sec:preliminaries} we introduce some preliminaries,
including various versions of Whitehead graphs.

In \fullref{sec:topologyofdecompositionspace} we use generalized
Whitehead graphs to investigate topological
features of $\decomp$.
We regard this section as semi-preliminary, as it is a development of
ideas that were present in \cite{CasMac11}.
However, some of the proofs are technical, so they are
presented here in detail in the interests of rigor and of making
this paper self-contained.

In \fullref{sec:splitting} we construct a simplicial tree from the cut
points and cut pairs of $\decomp$ and show that the quotient graph is
the \rjsj.

We apply these results in \fullref{sec:vg} to characterize virtually
geometric multiclasses, which will be introduced in \fullref{sec:vgintro}.

A benefit of our approach to relative splittings via the decomposition
space and generalized Whitehead graphs is that the arguments end up
being combinatorial.
It follows that not only is the \rjsj algorithmically constructible,
which was already
known by work of Kharlampovich and Miasnikov \cite{KhaMia05}, but
there is a combinatorial algorithm that is actually implementable.
In subsequent work with Manning, we have extended the methods of this
paper to get an implementable algorithm to construct the \rjsj, and we
have written a computer program \cite{CasMan14vg} (see also \cite{CasMan15})
that will compute the \rjsj and decide whether or not a given
multiclass is virtually geometric.

\subsubsection{First Examples}
We give two examples to give an idea of what decomposition spaces and
relative JSJ decompositions can look like. 
These examples are of a very special type: the free group $\F$
can be viewed as the fundamental group of a compact, connected, orientable surface
with boundary, and the multiclass includes the conjugacy class of each
of the boundary curves of the surface. 
In this case there are no relative free splittings, and we can `see'
the relative cyclic splittings---they correspond to essential, non-peripheral simple
closed curves\footnote{A curve is \emph{essential} if it not homotopic
  to a curve that bounds a disc and \emph{non-peripheral} if it not homotopic to a
  curve that bounds a once-punctured disc or an annulus.} in the surface that can be homotoped to be disjoint from
a multicurve representing the multiclass.

Much of the work in
\fullref{sec:topologyofdecompositionspace} is about how to `see' relative cyclic
splittings when no ambient surface topology is available.

\begin{example}\label{ex:first}
Consider $\F=\fingen{a,b}$
and $\multiklass=\{[\langle ab\inv{a}\inv{b}\rangle]\}$.
The decomposition space is homeomorphic to the circle. 
To see this, view $\F$ as the fundamental group of a complete, finite volume hyperbolic
punctured torus $\Sigma$, and represent
$ab\inv{a}\inv{b}$ as a simple closed curve running around the puncture. 
The universal cover of $\Sigma$ is the hyperbolic plane
$\mathbb{H}^2$, see \fullref{fig:firstexample}. 
\begin{figure}[h]
  \centering
  \includegraphics[width=1.5in]{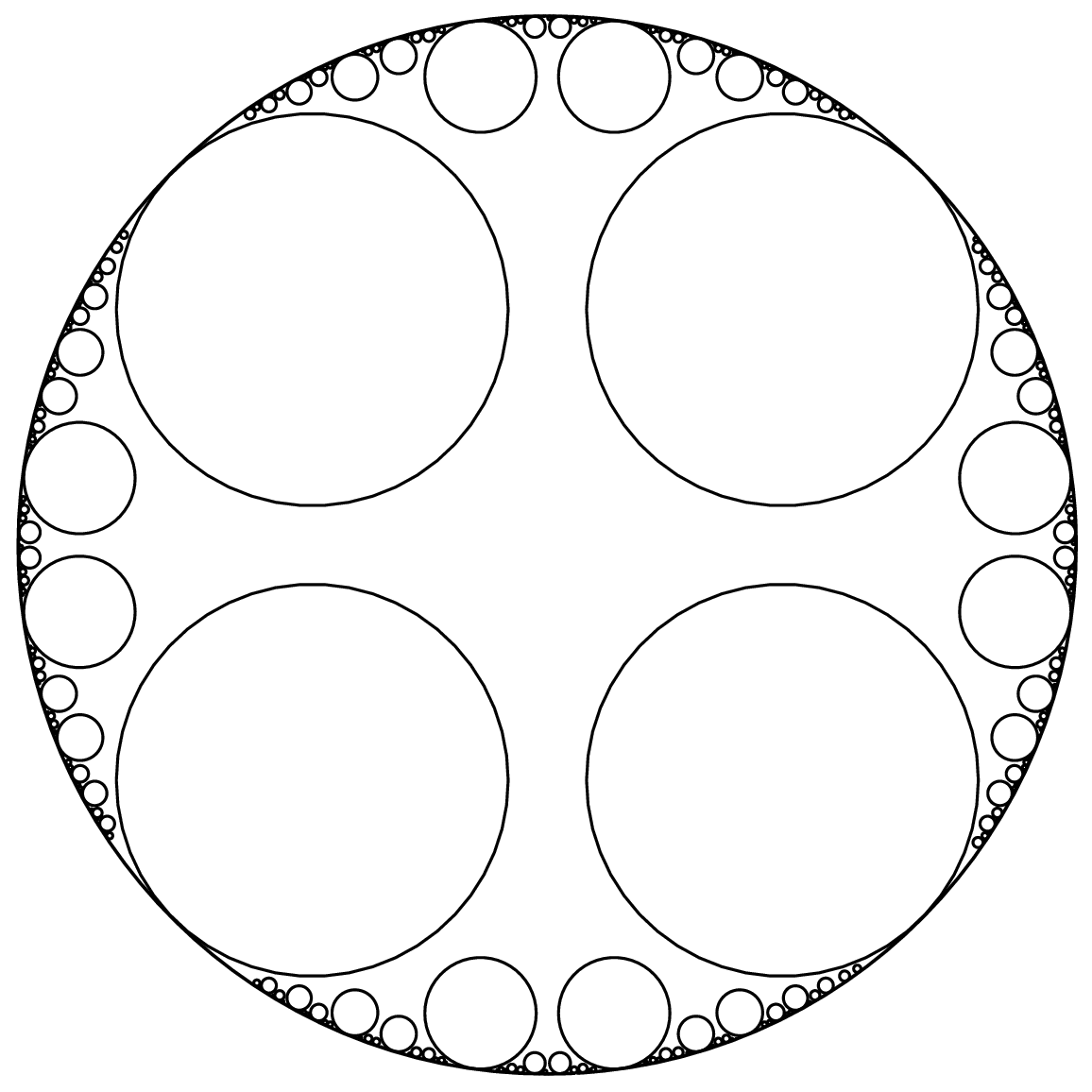}
  \caption{Horocycles in $\mathbb{H}^2$ demonstrating $\decomp=\bdry\mathbb{H}^2=S^1$.}
  \label{fig:firstexample}
\end{figure}
The group $\F$ acts by deck transformations on $\mathbb{H}^2$, and the
action extends to a continuous surjection
$\bdry\F\to\bdry\mathbb{H}^2$ that is 
 2 to 1 on the parabolic points
and 1 to 1 off them.
The curve representing $ab\inv{a}\inv{b}$ is freely
homotopic to the quotient of a horocycle.
The element $ab\inv{a}\inv{b}$ acts parabolically,
fixing a point $\xi\in\bdry\mathbb{H}^2$, and the preimage of $\xi$ in
$\bdry\F$ is exactly the two points $(ab\inv{a}\inv{b})^\infty$ and
$(ab\inv{a}\inv{b})^{-\infty}$ of $\bdry\F$.
Since there is only one orbit of parabolic points, we conclude $\decomp=\bdry\mathbb{H}^2=\mathbb{S}^1$.

Alternatively, we could take $\Sigma'$ to be a hyperbolic one-holed torus with
geodesic boundary component representing $ab\inv{a}\inv{b}$.
The universal cover sits inside of $\mathbb{H}^2$ as a thickened tree,
and $\bdry\F$ embeds  into $\bdry\mathbb{H}^2$.
The points $(ab\inv{a}\inv{b})^\infty$ and
$(ab\inv{a}\inv{b})^{-\infty}$ are sent to the endpoints of a interval
in $\bdry\mathbb{H}^2$ not containing any other points of $\bdry\F$, so the
quotient map $\qmap\from\bdry\F\to\decomp=\mathbb{S}^1$ can be viewed
as a circular analogue of the map collapsing missing intervals of the
ternary Cantor set to get the unit interval.

In this example the \rjsj is trivial --- a single vertex stabilized by
$\F$. 
This is due to the universality requirement, see
\fullref{theorem:JSJ}.
The reason is that cyclic splittings of $\F$ rel $\{[\langle
ab\inv{a}\inv{b}\rangle]\}$ correspond to essential, non-peripheral simple closed curves in
$\Sigma'$, but for every such curve there is another intersecting it,
so none of these splittings are universal.
\end{example}

\begin{example}\label{ex:second}
Consider the marked
surface $\Sigma$ in \fullref{fig:surface}. 
The labeled curves are generators of
the fundamental group $\F_5=\left<a,b,c,d,e\right>$.

\begin{figure}[h]
\begin{minipage}[b]{0.48\textwidth}
  \centering
\labellist
\small
\pinlabel $a$ at 66 15
\pinlabel $c$ at 53 15
\pinlabel $b$ at 75 56
\pinlabel $d$ at 44 56
\pinlabel $e$ at 60 63
\endlabellist
  \includegraphics[width=.95\textwidth]{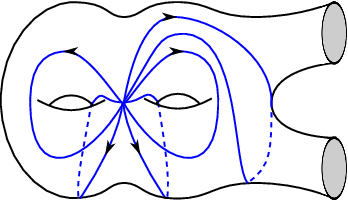}
  \caption{A marked surface}
  \label{fig:surface}
\vspace{1em}
\end{minipage}
\hfill
\begin{minipage}[b]{0.48\textwidth}
\labellist
\small
\pinlabel $a$ at 97 15
\pinlabel $c$ at 9 58
\pinlabel $d$ at 42 82
\pinlabel $e$ at 148 15
\pinlabel $\inv{a}c$ at 78 58
\pinlabel $dc\inv{d}\inv{c}$ at 77 80
\pinlabel $dc\inv{d}\inv{c}ab\inv{a}\inv{b}\inv{e}$ at 131 80
\endlabellist
  \centering
  \includegraphics[width=.95\textwidth]{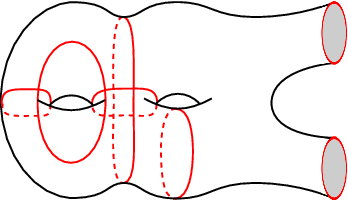}
  \caption{A multi-curve representing $\multiklass$}
  \label{fig:surfacedecomp}
\end{minipage}
\end{figure}

Consider the multiclass: \[\multiklass=\{[\langle a\rangle], [\langle
c\rangle], [\langle d\rangle], [\langle e\rangle], [\langle\inv{a}c\rangle],
[\langle dc\inv{d}\inv{c}\rangle], [\langle dc\inv{d}\inv{c}ab\inv{a}\inv{b}\inv{e}\rangle]\}\]
\fullref{fig:surfacedecomp} shows a \emph{multicurve} representing
$\multiklass$.

In this example, all boundary curves of the surface belong to the
multicurve, so $\decomp$ is connected as in \fullref{ex:first}.

\begin{figure}[h]
\begin{minipage}[b]{0.48\textwidth}
\labellist
\small
\pinlabel $\Sigma_1$ at 15 70
\pinlabel $\Sigma_2$ at 160 70
\pinlabel $A_3$ at 105 -6
\pinlabel $A_4$ at 105 101
\endlabellist
 \centering
  \includegraphics[width=.95\textwidth]{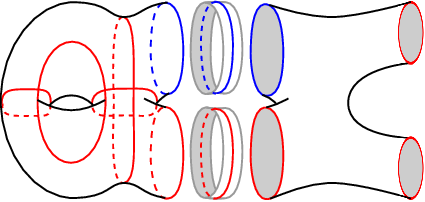}
  \caption{Decomposition of a surface.}
  \label{fig:surfacedecomp2}
\end{minipage}
\hfill
\begin{minipage}[b]{0.48\textwidth}
  \centering
\labellist
\pinlabel $\left<a\right>$ at 33 -1
\pinlabel $\left<dc\inv{d}\inv{c}a\right>$ at 33 40
\pinlabel $\left<dc\inv{d}\inv{c}a,\,ba\inv{b},\,e\right>$ at 60 20
\pinlabel $\left<a,c,d\right>$ at 0 20
\pinlabel $dc\inv{d}\inv{c}a$ at 0 34
\pinlabel $a$ at 1 10
\pinlabel $a$ at 22 -1
\pinlabel $a$ at 44 -1
\pinlabel $ba\inv{b}$ at 68 10
\pinlabel $dc\inv{d}\inv{c}a$ at 68 34
\pinlabel $dc\inv{d}\inv{c}a$ at 11 41
\pinlabel $dc\inv{d}\inv{c}a$ at 55 41
\endlabellist
\includegraphics[width=.75\textwidth]{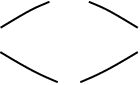}
\caption{Corresponding graph of groups}
\label{fig:surfacegog}
\end{minipage}
\end{figure}

Consider the decomposition of $\Sigma$ into subsurfaces given in
\fullref{fig:surfacedecomp2}.
The corresponding graph of groups in \fullref{fig:surfacegog} is the
rJSJ for this example.
In \fullref{fig:surfacegog} the vertices are labelled with their
stabilizer subgroups. The edge stabilizers are all infinite cyclic,
and the label at each end of each edge indicates the image of a fixed
generator of the edge stabilizer in the vertex group.

This is the rJSJ because every essential, non-peripheral simple closed curve in $\Sigma$ that
does not intersect a curve of the multicurve is
either homotopic to an essential, non-peripheral simple closed curve in the subsurface
$\Sigma_2$ or to the core curve of annulus $A_3$
or annulus $A_4$. 

Let us examine the four subsurfaces:
\begin{enumerate}
\item Subsurface $\Sigma_1$ is `filled' by the
multicurve, in the sense that every essential, non-peripheral simple closed curve in $\Sigma_1$ intersects
one of the curves of the multicurve.
This means that $\Sigma_1$ does not contribute any relative
splittings of $\F$.
The fundamental group of $\Sigma_1$ is an example of what will be
called a \emph{rigid}
vertex group.
\item The subsurface $\Sigma_2$ is a sphere with four holes. 
It is `empty', in the sense that there are no curves
of the multicurve in its interior.
Thus, any essential, non-peripheral simple closed curve in $\Sigma_2$ yields a cyclic
splitting of $\F$ rel $\multiklass$.
However, as in the previous example,
these splittings are not universal.
The fundamental group of $\Sigma_2$ is an example of what will be
called a \emph{QH-surface}
vertex group\footnote{These are the torsion-free examples of the
  `quadratically hanging' vertex groups of Rips and Sela
  \cite{RipSel94}.}.
\item The annulus $A_3$ has a core curve representing $[\langle
  a\rangle]\in\multiklass$, so $\qmap(a^\infty)=\qmap(\inv{a}^{\infty})$ is
  a cut point of $\decomp$.
\item The core curve of annulus $A_4$ is not homotopic to a curve of
  the multicurve.
The points $\qmap((dc\inv{d}\inv{c}a)^\infty)$ and
$\qmap((dc\inv{d}\inv{c}a)^{-\infty})$ are distinct points in
$\decomp$.
They form a cut pair in $\decomp$.
\end{enumerate}

Observe the following features of this decomposition, and compare \fullref{theorem:JSJ}.
The decomposition is bipartite: there is a collection of
annuli and a collection of more complicated subsurfaces, and each
subsurface is adjacent only to members of the opposite collection. 
Among the more complicated subsurfaces there are those that are filled
(the rigid vertices)
and those that are empty (the QH-surface vertices). 
Every splitting of $\F$ rel $\multiklass$ comes from
either the core curve of one of the annuli or an essential, non-peripheral simple closed curve in
one of the empty subsurface pieces.
\end{example}

\subsection{The Decomposition Space and the \rjsj}
We are interested in cut sets of $\decomp$.
If $\decomp$ is connected and $\decomp\setminus\{x\}$ is not connected, then $x$ is called a
\emph{cut point}.
Similarly, if $\decomp$ is connected and $\{x_0,x_1\}$ is a pair of
points, neither of which is a cut point, such that
$\decomp\setminus\{x_0,x_1\}$ is not connected, then $\{x_0,x_1\}$ is
called a \emph{cut pair}.
We call $(\F,\multiklass)$ \emph{rigid} if $\decomp$ is connected with no
cut points and no cut pairs\footnote{The fact that $(\F,\multiklass)$ is
quasi-isometrically rigid if and only if $\decomp$ is connected with no
cut points and no cut pairs is the main result of \cite{CasMac11}. 
Since we will not be concerned with quasi-isometric rigidity in this
paper, we take these condition on $\decomp$ to be the definition of
rigidity.}.
We call $(\F,\multiklass)$ a \emph{QH-surface} if there exists a compact
surface
with boundary, $\Sigma$, such that $\F=\pi_1(\Sigma)$
and $\multiklass=[\bdry\Sigma]$.
In this case we write $(\F,\multiklass)\sim (\Sigma,[\bdry\Sigma])$.

\subsubsection{Induced Multiclasses}
\begin{definition}\label{def:inducedmultiword}
  Let $G$ be a non-cyclic vertex group of a graph of groups
  decomposition $\Gamma$ of $\F$ rel $\multiklass$ with cyclic edge
  stabilizers.
Define the \emph{induced multiclass} in $G$,
denoted $\indgklass$, to be the set of distinct $G$--conjugacy classes
of $G$--maximal cyclic subgroups that either contain the image of an
edge injection into $G$ or are contained in an
$\F$--maximal cyclic subgroup whose conjugacy class is in $\multiklass$.
\end{definition}

In \fullref{ex:second}, the induced multiclass in $\pi_1(\Sigma_1)$
is:
\[\{[\fingen{a}],[\fingen{c}],[\fingen{d}],[\fingen{\inv{a}c}],[\fingen{dc\inv{d}\inv{c}}],[\fingen{dc\inv{d}\inv{c}a}]\}\]
These classes come from the four curves of the multicurve in the
interior of $\Sigma_1$, plus the two boundary curves, one of which was
a member of the multicurve, and one of which was not.

Similarly, $\Ind_\Gamma^{\pi_1(\Sigma_2)}=\{[\fingen{ba\inv{b}}], [\fingen{e}],
[\fingen{dc\inv{d}\inv{c}ab\inv{a}\inv{b}\inv{e}}],[\fingen{dc\inv{d}\inv{c}a}]\}$.
All of these classes correspond to boundary curves of $\Sigma_2$, three of which were
members of the multicurve, and one of which was not.

The vertex corresponding to $G$ is said to be \emph{rigid} or 
 \emph{QH} if $(G,\indgklass)$ is rigid or is a
QH--surface, respectively.
It will turn out that all
 non-cyclic vertices of the \rjsj are either rigid or QH.

\subsubsection{Outline of the Construction of the \rjsj}
Suppose that $\F$ does not split freely relative to $\multiklass$,
which is equivalent to supposing
that $\decomp$ is connected, by \fullref{thm:freesplitting}.
Suppose that $\{x_0,x_1\}$ and $\{y_0,y_1\}$ are cut pairs such that $y_0$ and $y_1$
lie in different complementary components of $\{x_0,x_1\}$.
In this case, we say
$\{y_0,y_1\}$ \emph{crosses} $\{x_0,x_1\}$.
If there does not exist a cut pair crossing
$\{x_0,x_1\}$ then we say $\{x_0,x_1\}$ is \emph{uncrossed}.

\fullref{proposition:circle} generalizes a construction of Bowditch to
show that if $\F$ does not split freely rel $\multiklass$, and if
$(\F,\multiklass)$ is neither rigid nor a QH-surface, then $\decomp$ contains
a cut point or an uncrossed cut pair. 

\fullref{proposition:periodic} says that an uncrossed cut pair is
rational\footnote{$\bdry\F$ can be thought of as the set of infinite, freely
  reduced words in the generators of $\F$ and their inverses. The
  points whose expressions as such are eventually periodic are
  commonly called `rational', in analogy to decimal representations of
rational numbers. Being a rational point of $\bdry\F$ is equivalent to
being fixed by an infinite cyclic subgroup. We extend the terminology
to call a pair of points \emph{rational} if they are fixed by an
infinite cyclic subgroup.}, that is, it is
stabilized by an infinite cyclic subgroup of $\F$. 
Moreover, there are only finitely many conjugacy classes of
stabilizers of uncrossed cut pairs.

We show in \fullref{proposition:multipleotal} that the collection of cut points and uncrossed cut pairs in
$\decomp$ has the structure of a simplicial tree.
Otal proved this for the cut points. We generalize his proof to work
simultaneously with the cut points and uncrossed cut pairs. 
Since $\F$ acts by homeomorphism on $\decomp$, and since there are
finitely many conjugacy classes of stabilizers of cut points and
uncrossed cut pairs, we get a cocompact $\F$--action on this cut
point/uncrossed cut pair tree.
\fullref{theorem:JSJ} says the \rjsj is the quotient graph of groups of this action.

A consequence of \fullref{proposition:multipleotal} is that every
uncrossed cut pair of $\decomp$ corresponds to a cyclic splitting of
$\F$ rel $\multiklass$.

\bigskip

The main work is proving \fullref{proposition:circle},
\fullref{proposition:periodic}, and \fullref{proposition:multipleotal}. 
Verifying that the resulting graph of groups satisfies the desired
properties of the \rjsj is routine.

\subsubsection{Otal, Bowditch, and Cut Points}
Otal \cite{Ota92} showed that the cut points of $\decomp$ have a simplicial tree
structure.
Bowditch \cite{Bow01} proved that the cut points of the boundary of
a relatively hyperbolic group have a simplicial tree structure.

It can be shown, see Manning \cite{Man15}, that $\decomp$ is equivariantly homeomorphic to the
Bowditch boundary of $\F$ relative to $\multiklass$, so Otal's result
is an early special case of Bowditch's result. 
For our purposes these results are not sufficient: 
The cut point tree does not see all the universal relative cyclic splittings, because it
misses the ones coming from uncrossed cut pairs.
Our methods treat cut points and uncrossed cut pairs in a unified way.

\subsection{Virtual Geometricity}\label{sec:vgintro}
In \fullref{sec:vg} we apply the \decompthm to characterize
virtual geometricity.

$\multiklass\subset \F$ is \emph{geometric} if it can be represented by
an embedded multicurve in the boundary of a handlebody with fundamental
group $\F$.

Otal's main result in \cite{Ota92}, suitably reinterpreted,
is that in the case that the \rjsj is trivial, $\multiklass$
is geometric if and only if the corresponding decomposition space is planar.
Furthermore, planarity of the decomposition space can be deduced from the
Whitehead graph of $\multiklass$.

$\multiklass$ is \emph{virtually
  geometric} if there is a finite index subgroup $G$ of $\F$ such that the
`lift' of $\multiklass$ to $G$ is geometric.
The \emph{lift of $\multiklass$ to $G$}, for $\multiklass=\{[\left<\mot_1\right>],\dots,[\left<\mot_k\right>]\}$, is the multiclass of $G$ that
contains every conjugacy class of maximal cyclic subgroup of $G$ that
is conjugate in $\F$ into one of the $\fingen{\mot_i}$, see \fullref{section:lift}.

We use the \rjsj to reduce virtual geometricity to geometricity of the
induced multiclasses in the vertex groups:

\begin{vgtheorem}[\fullref{theorem:planarequalsvg}]
\hypertarget{vgthmtarget}{
  For a multiclass in a free group, the following are equivalent:
\begin{enumerate}
\item The multiclass is virtually geometric.
\item The decomposition space is planar.
\item For every non-cyclic vertex group of the \rjsj, the induced multiclass is geometric.
\end{enumerate}
Thus, virtually geometric multiclasses are those that are built
from geometric pieces.}
\end{vgtheorem}

When there are no uncrossed cut pairs we use the \rjsj to
explicitly construct a finite index subgroup and handlebody that
demonstrate virtual geometricity.

When there are uncrossed cut pairs we first pinch them to cut points
and then apply the previous construction. 
The pinching is done in such a way as to preserve planarity of the
decomposition space, using a technical fact,  \fullref{corollary:arcconnectedcomponents}, that the closure of a
complementary component of a cut pair in $\decomp$ is arc-connected.

\subsection{Acknowledgements}

I thank Jason Manning, who noticed that the decomposition space of
\cite{CasMac11} was the same as the space considered by Otal, and
brought the question of virtual geometricity to my attention.
I also thank him for helpful comments on an earlier version of this paper.

I thank the anonymous referee for their diligence.

This work was partially supported by the Agence Nationale de la Recherche
  (ANR) grant ANR-2010-BLAN-116-01 GGAA, the European Research Council
  (ERC) grant of Goulnara ARZHANTSEVA, grant agreement \#259527, and
  the Austrian Science Fund (FWF):M1717-N25

\section{Preliminaries}\label{sec:preliminaries}

\subsection{Definitions and Notation}

The \emph{degree} of a homomorphism from the integers into a free
group is the index of its image in the maximal cyclic subgroup containing
the image.

A nontrivial element $g\in \F$ is \emph{indivisible} if is not a proper
power. 

Let $\basis$ be a basis of $\F$.
The Cayley graph of $\F$ with respect to $\basis$ is a tree $\tree$.
Assign each edge length one; $\F$
acts isometrically on $\tree$ by left multiplication.
We will use $\bp$ to denote the vertex corresponding to the identity
element of $\F$.

The tree $\tree$ has
a Gromov boundary at infinity $\bdry\tree$ that is identified with $\bdry\F$.
This boundary compactifies the tree: $\closure{\tree}=\tree\cup \bdry\tree$ is
a compact topological space whose topology on $\tree$ agrees with the
metric topology. 
For $x,y\in\closure{\tree}$, there exists a unique geodesic $[x,y]$
connecting them.
By a \emph{simplicial geodesic} we shall mean an isometric
embedding $\phi\from [a,b]\into \closure{\tree}$ taking integers to
vertices, with $a,b\in\mathbb{Z}\cup\{\pm\infty\}$.
We use the notation $\phi\from[a,b]\onto [x,y]$ to indicate a
simplicial geodesic with $\phi(a)=x$ and $\phi(b)=y$.
There is unique such simplicial geodesic if $a$ or $b$ is finite.

For a fixed vertex $u$, a basis for the topology of $\bdry\tree$ is
given by the sets
\[\shadow^u(v)=\{\xi\in\bdry\tree\mid v\in[u,\xi]\}\quad\text{for } v\in\tree\setminus\{u\}.\]
 The resulting topology does not depend on the choice of $u$.

\subsection{Cut Pairs}
Recall that a minimal cut set is a subset $Y\subset X$ such that
$X\setminus Y$ is not connected but $X\setminus Z$ is connected for
every proper subset $Z$ of $Y$.

\begin{lemma}\label{lemma:limitpoints}
  Let $Y$ be a closed minimal cut set of connected, locally connected
  space $X$. Every complementary
  component limits to every point of $Y$.
\end{lemma}
\begin{proof}
$X\setminus Y$ is locally connected, so components are proper,
non-empty clopens.  
If $y\in Y$ is not a limit point of a component $C$ then $C$ is still a
  proper, non-empty clopen in $X\setminus (Y\setminus y)$,
  contradicting minimality of $Y$.
\end{proof}

\begin{definition}
 If $Y$ is a cut set of $X$ then $Z$ \emph{crosses} $Y$ if there are
 points $z_0,\,z_1\in Z$ in different components of $X\setminus Y$.
\end{definition}

Recall that a cut pair is a minimal cut set of size two, and a cut
pair is said to be uncrossed if no cut pair crosses it.
The following lemma is easily verified:
\begin{lemma}\label{lemma:precisely2components}
Let $X$ be a connected, locally connected space in which cut points and cut pairs have
finitely many complementary components.
  \begin{enumerate}
  \item Crossing is a symmetric relation among cut pairs.
\item A cut pair cannot cross a cut point.
\item A cut pair with at least three complementary components is uncrossed.
  \end{enumerate}
\end{lemma}

\subsection{A Multiclass Lifted to a Finite Index Subgroup}\label{section:lift}
\begin{definition}\label{definition:lift}
   The \emph{lift} of $\multiklass$ to a finite index subgroup $G$ is the multiclass 
  $\multiklass_G$ consisting of the distinct 
  $G$--conjugacy classes of $G$--maximal cyclic subgroups that are
  contained in an $\F$--maximal cyclic subgroup whose conjugacy class
  is in $\multiklass$.
\end{definition}

The inclusion $\iota\from G\into \F$ extends to a
 homeomorphism $\bdry\iota\from\bdry G\to\bdry\F$ that takes
 $\bdry\multiklass_G$ to $\bdry\multiklass$, inducing
 a $G$--equivariant homeomorphism $\decomp_{\multiklass_G}\to\decomp_{\multiklass}$.

\begin{example}
Let $\F=\fingen{a,b}$.
Let $\multiklass=\{[\fingen{a}],[\fingen{b}]\}$.
Let $G$ be the index 2 subgroup $G=\fingen{a^2,b,ab\inv{a}}<\F$.
The lift of $\multiklass$ to $G$ is
$\multiklass_G=\{[\fingen{a^2}],[\fingen{b}],[\fingen{ab\inv{a}}]\}$.

\begin{figure}[h]
\hfill
  \begin{minipage}[h]{.5\linewidth}
    \labellist
\small
\pinlabel $\F=\pi_1(H)=\fingen{a,b}$ [br] at -50 10
\pinlabel $\multiklass=\{[\fingen{a}],[\fingen{b}]\}$ [br] at -50 40
\pinlabel $G=\pi_1(\widetilde H)=\fingen{a^2,b,ab\inv{a}}$ [br] at -50 150
\pinlabel $\multiklass_G=\{[\fingen{a^2}],[\fingen{b}],[\fingen{ab\inv{a}}]\}$ [br] at -50
180 
\endlabellist
    \includegraphics[scale=.5]{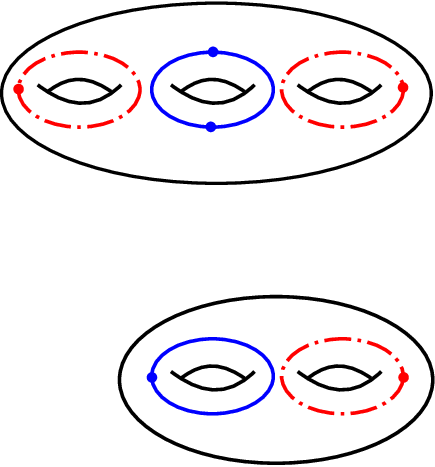}
  \end{minipage}
  \caption{Lifting to a finite index subgroup}
  \label{fig:lift}
\end{figure}
We can visualize the situation by taking $F$ to be the fundamental group of a
handlebody $H$ and picking curves representing $[\fingen{a}]$ and
$[\fingen{b}]$.
Then $G$  corresponds to a 2--fold cover $\widetilde H$ of $H$, and
$\multiklass_G$ is represented by the curves in $\widetilde H$
covering the chosen curves in $H$, as in 
\fullref{fig:lift}.
\end{example}

\subsection{Normalization} \label{sec:normalization}
See Serre \cite{Ser03} for an introduction to graphs of groups and
Bass-Serre theory.

Suppose $\Gamma$ is a graph of groups decomposition of $\F$ with cyclic
edge groups. 
It will be convenient to normalize $\Gamma$.
We describe a sequence of moves that change the graph of groups description without
changing the group itself or the conjugacy classes of non-cyclic vertex groups.
If $e$ is an edge of $\Gamma$ let $\eta(e,0)$ and $\eta(e,1)$ denote
the initial and terminal vertices of $e$, respectively.
Let $\phi_{e,i}\from G_e\into G_{\eta(e,i)}$ be the edge injection of
an edge group into a vertex group.

First, if there is an edge incident to two non-cyclic vertex groups,
subdivide it by adding a vertex with stabilizer equal to
the stabilizer of the edge group.

Second, for each edge $e$ let $G_e=\left<z_e\right>$ and let $\left<z_{e,i}\right>$ be the maximal
cyclic subgroup of $G_{\eta(e,i)}$ containing $\phi_{e,i}(G_e)$.
Since no nontrivial element is conjugate to a power of itself in the free
group, it is possible to choose the generators $z_e$ and $z_{e,i}$ so that
for all $e$ and $i$ the map $\phi_{e,i}$ takes $z_e$ to a positive
power of $z_{e,i}$ and so that if
$\inv{g}\left<z_{e,i}\right>g=\left<z_{e',i'}\right>$ for some $g\in G_{\eta(e,i)}=G_{\eta(e',i')}$ then $\inv{g}z_{e,i}g=z_{e',i'}$.

Third, if an edge group maps into a non-maximal cyclic subgroup of a
non-cyclic vertex group, we may un-collapse an edge as in \fullref{fig:uncollapse}.

\begin{figure}[h!]
\labellist
\small
\pinlabel $G$ [r] at 0 6
\pinlabel $G$ [r] at 135 6
\pinlabel $\left<z\right>$ [c] at 205 6
\tiny
\pinlabel $g^p$ [bl] at 0 6
\pinlabel $g$ [bl] at 135 6
\pinlabel $z^p$ [bl] at 220 6
\pinlabel $z$ [br] at 190 6
\endlabellist
  \centering
\includegraphics[height=.05in]{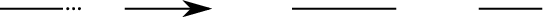}
\caption{Un-collapse an edge.}\label{fig:uncollapse}
\end{figure}

Fourth, consider two edges $e$ and $e'$ incident to a non-cyclic
vertex $v=\eta(e,i)=\eta(e',i')$.
Suppose $\phi_{e,i}(G_{e})$ and $\phi_{e',i'}(G_{e'})$ are distinct
and conjugate in $G_v$. Choose $g\in G_v$ such that
$\inv{g}\phi_{e,i}(G_{e})g=\phi_{e',i'}(G_{e'})$.
Replace the edge map $\phi_{e,i}$ with $\inv{g}\phi_{e,i}g$.

Fifth, fold all the edges together that map into a common maximal
cyclic subgroup in a given non-cyclic vertex, as in
\fullref{fig:fold}.

\begin{figure}[h!]
  \centering
\labellist
\small
\pinlabel $\left<y\right>$ [c] at 53 5
\pinlabel $G$ [c] at 130 5
\pinlabel $\left<z\right>$ [l] at 190 5
\pinlabel $G$ [r] at 308 5
\pinlabel $\left<z\right>$ [c] at 369 5
\tiny
\pinlabel $y^p$ [br] at 51 5
\pinlabel $y$ [bl] at 60 5
\pinlabel $g$ [br] at 123 5
\pinlabel $g$ [bl] at 138 5
\pinlabel $z^q$ [br] at 202 5
\pinlabel $g$ [bl] at 301 5
\pinlabel $z^q$ [br] at 365 5
\pinlabel $z^{pq}$ [bl] at 378 5
\endlabellist
\includegraphics[height=.05in]{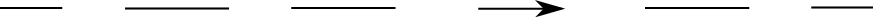}
\caption{Fold two edges.}\label{fig:fold}
\end{figure}

Folding is always possible when two edges map into a common
maximal cyclic subgroup of a non-cyclic vertex because we are in the
free group.
Consider the possible obstructions:
\begin{itemize}
\item If $\fingen{y}=\fingen{z}$ then the two edges of \fullref{fig:fold} form a loop
  corresponding to a stable letter $t$ conjugating $y$ to $y^{\pm q}$. This
  would mean $\fingen{t,y}$ is a Baumslag-Solitar subgroup, but free groups
  do not contain such subgroups.
\item If $\fingen{y}\neq\fingen{z}$ we could imagine the situation
depicted in \fullref{fig:nonfold} with $r>1$ and $q>1$.
In this case, $\fingen{y,z}$ is a virtually free-by-cyclic subgroup, but free groups
  do not contain such subgroups.
\end{itemize}

\begin{figure}[h!]
  \centering
\labellist
\small
\pinlabel $\left<y\right>$ [r] at 4 2
\pinlabel $G$ [c] at 66 2
\pinlabel $\left<z\right>$ [l] at 125 2
\tiny
\pinlabel $y^r$ [bl] at -5 2
\pinlabel $g$ [br] at 60 2
\pinlabel $g$ [bl] at 75 2
\pinlabel $z^q$ [br] at 140 2
\endlabellist
\includegraphics[height=.015in]{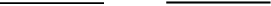}
\caption{Cannot occur in a free group.}\label{fig:nonfold}
\end{figure}

Similarly, since $\F$ is free, every edge group maps onto a maximal
cyclic subgroup in one of its two vertex groups.
Otherwise we would find a Baumslag-Solitar or virtually free-by-cyclic
subgroup.
Therefore, for any edge that is incident to two cyclic vertices, one
of the inclusions of the edge group into the vertex groups is an
isomorphism, and we can collapse the edge, as in
\fullref{fig:collapse}.

\begin{figure}[h!]
  \centering
\labellist
\small
\pinlabel $\left<y\right>$ [c] at 52 4
\pinlabel $\left<z\right>$ [l] at 115 4
\pinlabel $\left<z\right>$ [l] at 250 4
\tiny
\pinlabel $y^p$ [br] at 47 4
\pinlabel $y$ [bl] at 62 4
\pinlabel $z^q$ [br] at 123 4
\pinlabel $z^{pq}$ [br] at 260 4
\endlabellist
\includegraphics[height=.05in]{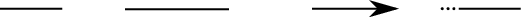}
\caption{Collapse an edge.}\label{fig:collapse}
\end{figure}

We are left with a new graph of groups decomposition of $\F$ that is bipartite: vertex groups are either maximal cyclic subgroups or are non-cyclic. 
Cyclic vertex groups are adjacent only to non-cyclic vertex groups,
and vice versa. 
For each edge, the edge map to the incident non-cyclic
vertex group maps the edge group onto a maximal cyclic subgroup of the
vertex group.
Furthermore, for each non-cyclic vertex group $G$ and each maximal
cyclic subgroup $C$ of $G$ there is at most one incident edge whose
edge group maps into a conjugate of $C$.

\begin{remark}
Another normalization that if often applied to a graph of groups is to 
make them \emph{reduced}. 
This means that if there is a non-loop edge $e$ incident to a vertex 
$\eta(e,i)$ such that the edge inclusion $\phi_{e,i}$ is an 
isomorphism, then the edge $e$ should be collapsed.
We do not assume that $\Gamma$ is reduced, because in some cases doing
so would ruin the
`bipartite' condition.   
\end{remark}

\subsection{Whitehead Graphs}
Our tool for understanding the topology of the decomposition
space associated to a multiclass is the generalized Whitehead graph of
the multiclass.
This machinery was developed in \cite{CasMac11}.

\subsubsection{Classical Whitehead Graph}
Let $\klass$ be a conjugacy class of maximal cyclic
subgroups of $\F$.
The (classical) Whitehead graph $\Wh_{\basis}(\bp)\{\klass\}$ of $\klass$ with respect to a basis
$\basis$ of $\F$ is a graph with $2n$ vertices labeled with the
elements of $\basis$ and their inverses. 
Let $\mot$ be a freely and cyclically reduced word in $\basis^\pm$ that generates a
representative of $\klass$.
One edge of $\Wh_{\basis}(\bp)\{\klass\}$ joins vertex $x$ to vertex $y$ for each occurrence of
$\inv{x}y$ in $\mot$, thought of
as a cyclic word. 
This definition extends to a multiclass by adding edges for each
class of the multiclass.
We will see a geometric interpretation and examples in \fullref{sec:generalizedwhiteheadgraph}.

Let $|\klass|_{\basis}$ be the minimal $\basis$--length of a generator
of a representative of $\klass$.
The \emph{complexity} of the Whitehead graph is the number of edges,
which is equal to $\sum_{\klass\in\multiklass}|\klass|_{\basis}$. 
A Whitehead graph $\Wh_{\basis}(\bp)\{\multiklass\}$ is \emph{minimal}
if its complexity is minimal among the complexities of
$\Wh_{\phi^{-1}(\basis)}(\bp)\{\multiklass\}$ for $\phi\in\Aut(\F)$.

Whitehead's Algorithm \cite{Whi36} picks a basis $\basis$ for which
$\Wh_{\basis}(\bp)\{\multiklass\}$ is minimal.
The proof shows that there is a finite set of `Whitehead
automorphisms' so that if $\Wh_{\basis}(\bp)\{\multimot\}$ is not
minimal then there exists a Whitehead automorphism $\phi$ that
strictly reduces the complexity.
The algorithm checks if any Whitehead automorphism reduces the
complexity, and repeats this process until no reducing Whitehead
automorphism exists.

A important observation in the proof is that if a Whitehead graph
$\Wh_{\basis}(\bp)\{\multiklass\}$ is connected and has a cut vertex, then it is not
minimal.

An easy extension of Whitehead's methods yields the following:

\begin{proposition} \label{proposition:Whitehead}
The following are equivalent:
  \begin{enumerate}
  \item Some Whitehead graph for $\multiklass$ is not connected.
\item Every minimal Whitehead graph for $\multiklass$ is not connected.
\item $\F$ splits freely rel $\multiklass$.
  \end{enumerate}
\end{proposition}

It is easy to see that if $\F$ splits freely rel $\multiklass$ then
$\decomp$ is not connected:
\begin{corollary}\label{corollary:disconnected}
 If there is a basis $\basis$ such that $\Wh_{\basis}(\bp)\{\multiklass\}$ is not connected,
 then $\decomp$ is not connected.
\end{corollary}
The converse is also true, see \fullref{thm:freesplitting}.

\subsubsection{Standing Assumption}\label{standingassumption}
From now on, unless otherwise noted, we assume that $\multiklass=\{[\langle
\mot_1\rangle],\dots, [\langle\mot_k\rangle]\}$ is fixed and $\basis$ is a basis of $\F$ such that
  $\Wh_{\basis}(\bp)\{\multiklass\}$ is connected without cut vertices.
Let $\tree$ denote the Cayley tree of $\F$ with respect to $\basis$.

Having fixed a reference basis, we simplify notation by
considering the \emph{multiword} $\multimot=\{\mot_1,\dots,\mot_k\}$,
where the $\mot_i$ are cyclically reduced and generate
non-conjugate maximal cyclic subgroups.
Similarly, by choosing representatives we pass from the induced
multiclass in a vertex group of a splitting to an induced multiword,
and from a lifted multiclass in a finite index subgroup to a lifted
multiword.

We drop $\multiklass$ and $\basis$ from the notation
unless they are necessary for clarity.

\begin{remark}
There is no loss of generality from these assumptions. 
If $\F$ splits freely relative to $\multiklass$ then first pass to a
maximal relative free splitting and then deal with the factors separately.
If there is no such free splitting then
\fullref{proposition:Whitehead} says that to ensure the no-cut-vertex assumption it suffices to choose the basis
that gives the minimal complexity Whitehead graph.
However, minimality is not necessary.
\end{remark}

\subsubsection{Generalized Whitehead Graph and Friends}\label{sec:generalizedwhiteheadgraph}
For each
$\{(\inv{f}\inv{\mot}_if)^\infty,(\inv{f}\mot_if)^\infty\}\in\bdry\multimot$
there is a unique bi-infinite geodesic with endpoints
$(\inv{f}\inv{\mot}_if)^\infty$ and $(\inv{f}\mot_if)^\infty$\!.
The stabilizer $\langle\inv{f}\mot_if\rangle$ acts cocompactly with
translation length $|\mot_i|$.

\begin{definition}
  The \emph{line pattern generated by $\multimot$} is the set $\lines$ (=$\lines_{\multimot}$)
  of `lines', bi-infinite geodesics $\line$ in $\closure{\tree}$, with endpoints $\line^-\!,\,\line^+\in\bdry\tree$ such that $\{\line^-,\line^+\}\in\bdry\multimot$.
\end{definition}

\begin{definition}
Let $\X$ be a connected subset of $\closure{\tree}$ with $\X\cap\tree\neq\emptyset$.
The \emph{Whitehead graph over $\X$}, denoted
$\Wh(\X)$ (= $\Wh_{\basis}(\X)\{\multimot\}$) is a graph with one vertex
for each component of $\closure{\tree}\setminus \closure{\X}$ and one
edge $\we$ joining vertices $\wv$ and $\wv'$ for
each $\line\in\lines$ with one endpoint in $\wv$ and the other in $\wv'$.

Additionally, each vertex and edge carries a piece of data that
records whence it came:
Let $\tree_\wv$ denote the component of
$\closure{\tree}\setminus\closure{\X}$ corresponding to vertex $\wv$, and let
$\line_\we$ be the line in $\lines$ corresponding to edge $\we$.
\end{definition}

Recall that $\bp$ denotes the vertex of $\tree$ corresponding to the
identity element of $\F$, so the notation
$\Wh_{\basis}(\bp)\{\multiklass\}$ of the previous section refers to the
Whitehead graph over the vertex $\bp$.

We have partially defined functions from $\closure{\tree}$ and $\lines$ to
$\Wh(\X)$:
\begin{definition}
  If $x\in\closure{\tree}\setminus\closure{\X}$ define $\gr_\X(x)$ to
  be the vertex of $\Wh(\X)$ corresponding to the component
  of $\closure{\tree}\setminus\closure{\X}$ containing $x$.

If $\line\in\lines$ such that $\line\cap\X\neq\emptyset$ and
$\bdry\line\cap\bdry\closure{\X}=\emptyset$, define $\gr_\X(\line)$ to be the
edge of $\Wh(\X)$ contributed by $\line$.

We shorten the notation to $\gr$ when $\X$ is apparent. 
\end{definition}

The point of this extra data is to consider Whitehead graphs not just
as abstract graphs, but as pictures of $\lines$ in $\tree$.

\begin{figure}[h]
  \centering
\labellist
\small
\pinlabel \mbox{Closeup at $b$} at 480 250
\pinlabel \mbox{Closeup over $[\bp,a]$} [t] at 480 0
\endlabellist
  \includegraphics[width=\linewidth]{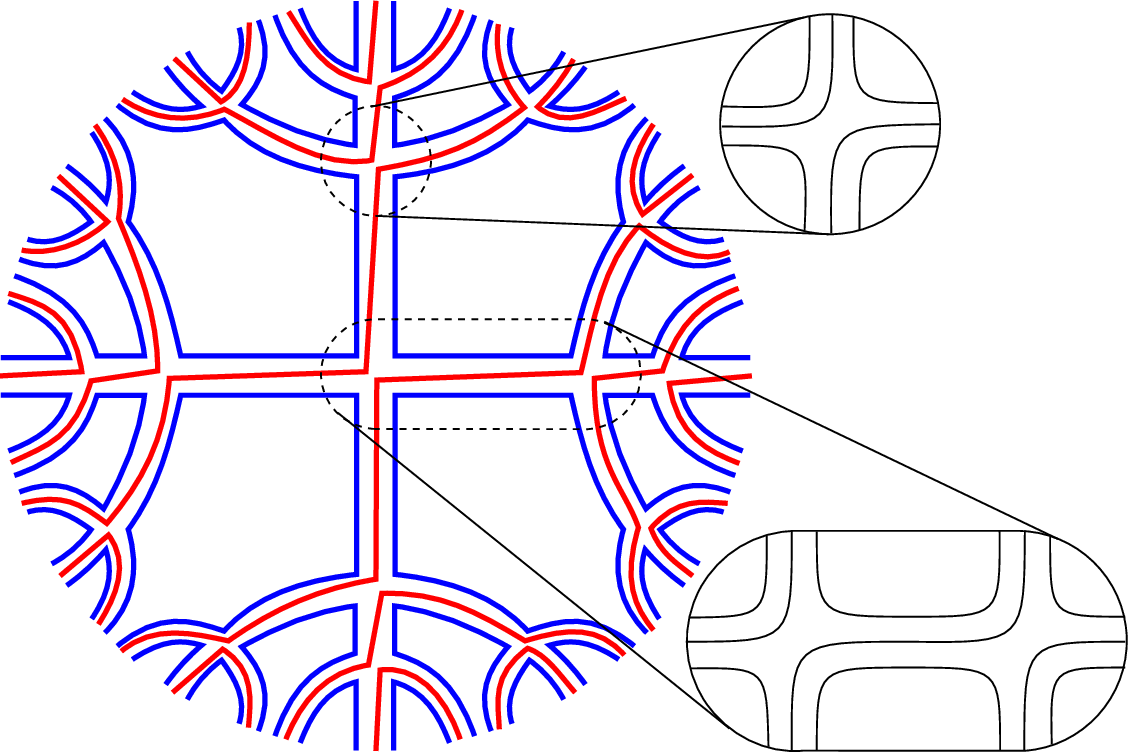}
  \caption{Line pattern $\lines_{\{ab,ab\bar{a}\bar{b}\}}$ with closeups.}
  \label{fig:linepattern}
\end{figure}

\begin{figure}[h]
  \centering
\labellist
\small
\pinlabel $\Wh(b)\{ab,ab\bar{a}\bar{b}\}$ at -30 0 
\pinlabel $\Wh([\bp,a])\{ab,ab\bar{a}\bar{b}\}$ at 550 0
\endlabellist
\includegraphics[height=1in]{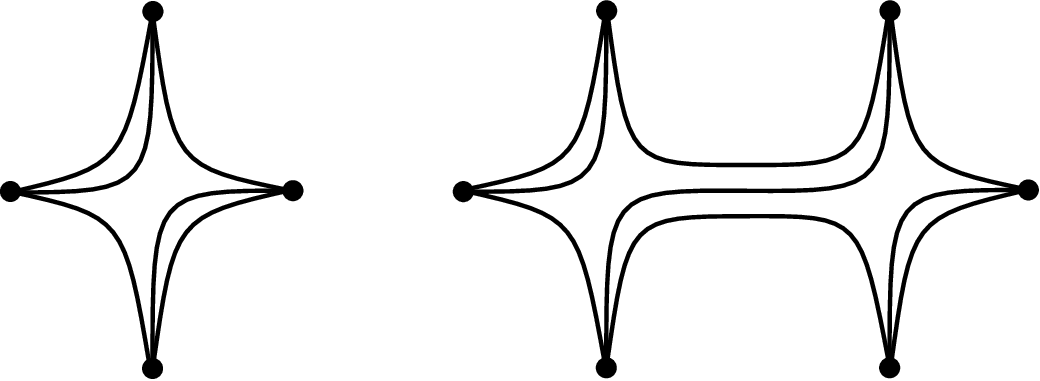}
  \caption{Whitehead graphs}
  \label{fig:whex}
\end{figure}

\fullref{fig:linepattern} depicts the line pattern
$\lines_{\{ab,ab\bar{a}\bar{b}\}}$ in $\F=\fingen{a,b}$.
It also shows closeups of the lines that pass through the vertex $b$
and through the edge $[\bp,a]$.
The Whitehead graphs $\Wh(b)\{ab, ab\inv{a}\inv{b}\}$ and
$\Wh([\bp,a])\{ab,ab\inv{a}\inv{b}\}$ in \fullref{fig:whex} show
the corresponding Whitehead graphs.

Notice that the Whitehead graph over $\X$ looks like the closeup of
the line pattern passing through $\X$, with lines exiting $\X$ through
a common edge of $\tree$ pinched to a vertex.
It will sometimes be convenient \emph{not} to do this pinching, but at
the same time remember the incidence of the edges.
For this we introduce the following formalism:

\begin{definition}
  A \emph{graph with loose ends} at $v_1,\dots,v_k$ is a
  graph $\Gamma$ with a specified subset of vertices
  $\{v_1,\dots,v_k\}$ that have been marked `deleted'. 
An edge $e$ of $\Gamma$ incident to a deleted $v_i$ is said to
\emph{have a loose end at $v_i$.}

A \emph{component} of a graph with loose ends is an equivalence class of edges and
undeleted vertices given by the incidence relation.
\end{definition}

\begin{definition}
For connected sets $\X\subset\Y\subset\closure{\tree}$ with
$\X\cap\tree\neq\emptyset$, let $\Wh(\X)\wrel\Y$ denote $\Wh(\X)$
with loose ends at each vertex $\wv\in\Wh(\X)$ such that
$\tree_\wv\cap\Y\neq\emptyset$.
\end{definition}
In other words, to construct $\Wh(\X)$ we look at the lines of
$\lines$ passing through $\X$ and pinch off a vertex for each edge $e$
of $\tree\setminus\X$ incident to $\X$.
For $\Wh(\X)\wrel \Y$, we do the same, except that we do not pinch a
vertex if $e\in\Y$.
Instead we will have loose ends at a deleted vertex corresponding to
such an $e$.
See \fullref{fig:whloose} for an example, and compare to \fullref{fig:whex}.
The utility of this definition is that it will let us build up large
Whitehead graphs from smaller pieces.
The idea, which will be made precise in \fullref{sec:cutting}, is that
if $\Y=\coprod_i\X_i$ is a connected subset of $\tree$ written as a
disjoint union of connected pieces $\X_i$ then we can build up $\Wh(\Y)$
from the various `graphs with loose ends' $\Wh(\X_i)\wrel\Y$ by
`splicing loose ends'.

\begin{figure}[h]
  \centering
  \includegraphics[height=1in]{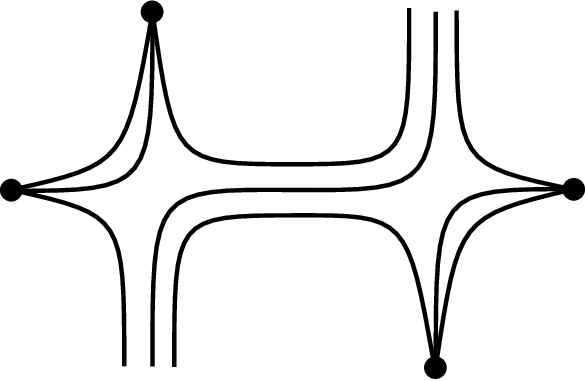}
  \caption{Whitehead graph with loose ends $\Wh([\bp,a])\{ab\bar{a}\bar{b},ab\}\wrel [\bar{b},ab]$}
  \label{fig:whloose}
\end{figure}

\begin{definition}
For connected sets $\X\subset\Y\subset\closure{\tree}$ with
$\X\cap\tree\neq\emptyset$, let $\wc$ be a component of
$\Wh(\X)\wrel\Y$.

$\wc$ \emph{has a loose end at} $\wv$ if it contains an edge with a loose end
at $\wv$.
$\wc$ \emph{has an end at}
$\xi\in\bdry\tree$ if $\xi$ is a limit point of
$\bigcup_{\we\in\wc}\line_\we$.  

Define $\bdry\tree_\wc=\coprod_{\wv\in\wc}\bdry\tree_\wv$.
\end{definition}

\subsubsection{Splicing}\label{sec:splicing}
Manning \cite{Man10} gave a construction for combining two graphs called
\emph{splicing}.
Let $\Gamma_0$ and $\Gamma_1$ be graphs with vertices
$\gamma_i\in\Gamma_i$ of the same valence.
Let a \emph{splice map} $\sigma$ be a bijection between edges incident to $\gamma_0$
and edges incident to $\gamma_1$.
The result of splicing $\Gamma_0$ and $\Gamma_1$ at $\gamma_0$ and
$\gamma_1$ by $\sigma$ is defined to be a graph whose vertices are the union of
vertices of $\Gamma_0$ and $\Gamma_1$, minus $\gamma_0$ and
$\gamma_1$.
Edges not incident to $\gamma_0$ or $\gamma_1$ are retained.
If $e_0=[u,\gamma_0]$ and $e_1=\sigma(e_0)=[\gamma_1,v]$, then add an
edge $[u,v]$ in the new graph.
In the above terminology, take the graph $\Gamma_0$ with
loose ends at $\gamma_0$ and the graph $\Gamma_1$
with loose ends at $\gamma_1$ and `tie up' the loose ends by matching
them using the given splice map $\sigma$.

\subsubsection{Cutting Whitehead Graphs into Pieces and Splicing them Together}\label{sec:cutting}
The following lemma is the motivation for the splicing construction:
\begin{lemma}\label{lemma:cutting}
For connected sets $\X\subset\Y\subset\closure{\tree}$ with
$\X\cap\tree\neq\emptyset$, suppose
$\X_0\amalg\X_1=\X\setminus e$ for some edge $e$ of $\X$ in $\tree$.
Then $\Wh(\X)\wrel\Y$ is obtained from $\Wh(\X_0)\wrel\Y$ and
$\Wh(\X_1)\wrel\Y$ by discarding any edges $\we$ with
$\bdry\l_\we\cap\closure{\X}\neq\emptyset$ and splicing remaining
loose ends at $\gr_{\X_0}(e)$ in $\Wh(\X_0)\wrel\Y$ to loose ends
at $\gr_{\X_1}(e)$ in $\Wh(\X_1)\wrel\Y$.
\end{lemma}

In this situation we say \emph{$\Wh(\X_i)\wrel\Y$ includes into
$\Wh(\X)\wrel\Y$ via splicing}. See \fullref{fig:inclusionviasplicing}.

\begin{figure}[h]
  \centering
\labellist
\pinlabel $\Wh(\bar{b})\wrel \Y$ at 70 0
\pinlabel $\Wh([\bp,a])\wrel \Y$ at 100 225
\pinlabel $\Wh(\Y)$ at 625 150
\pinlabel $\Wh(ab)\wrel \Y$ at 775 250
\endlabellist
  \includegraphics[width=\linewidth]{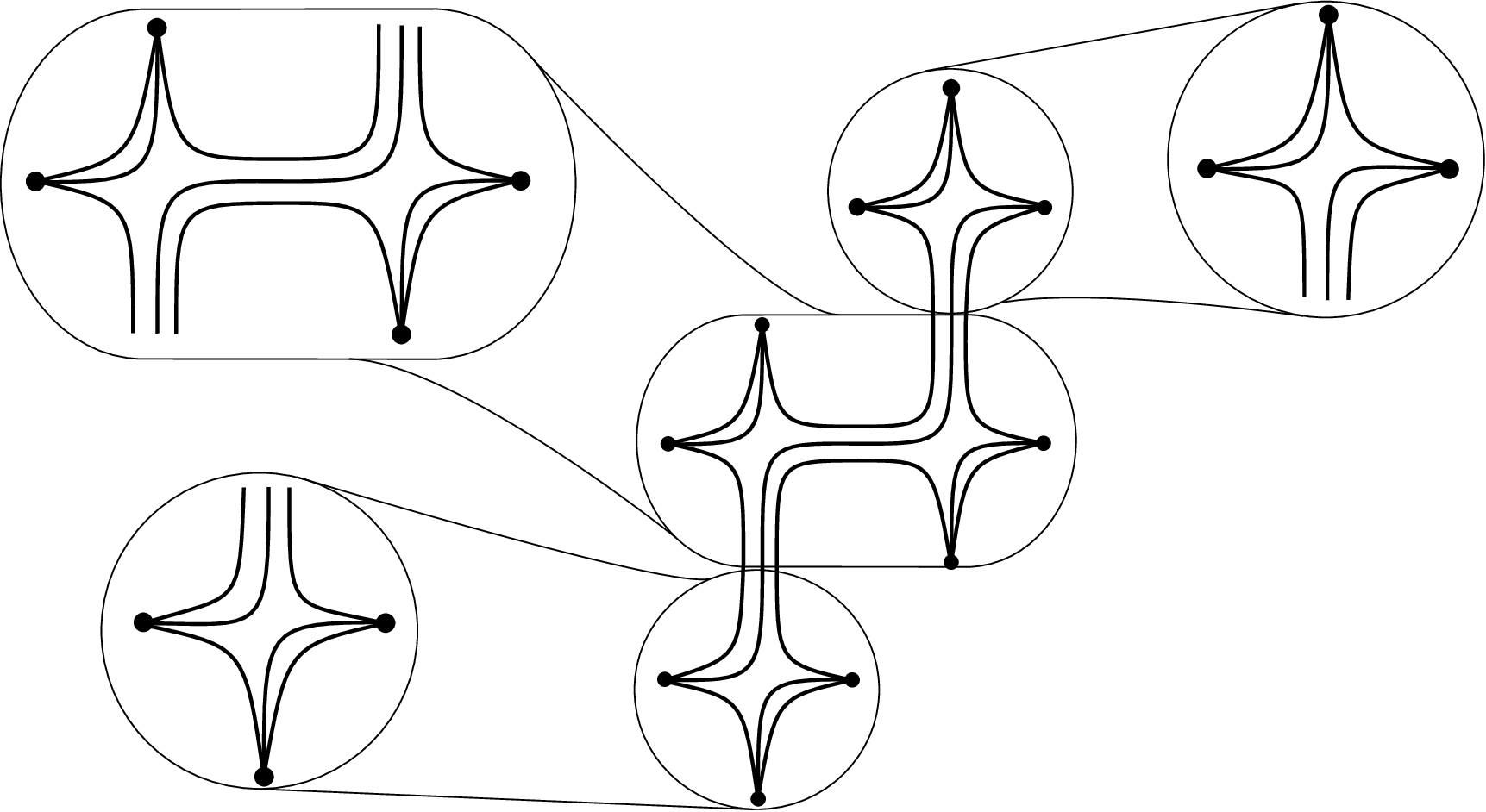}
  \caption{Inclusion via splicing for $\F=\fingen{a,b}$,
    $\multimot=\{ab\bar{a}\bar{b},ab\}$, $\Y=[\bar{b},ab]$}
  \label{fig:inclusionviasplicing}
\end{figure}

Manning states a similar result in the case that $\X=\Y$ is bounded.
The proof is immediate from the definitions.
The essential observation is that there is a natural splicing map defined by
sending an edge $\we\in\Wh(\X_0)\wrel\Y$ with a loose end at
$\gr_{\X_0}(e)$ to the unique edge $\we'\in\Wh(\X_1)\wrel\Y$ with
a loose end at $\gr_{\X_1}(e)$ such that $\line_\we=\line_{\we'}$.

Manning notes that splicing two graphs that are connected without cut
vertices yields a graph that is connected without cut vertices. 
$\Wh(\bp)$ is connected without cut vertices, by our standing assumption, so, by induction on the
number of vertices of $\X$:
\begin{lemma}[\cite{Man10}]\label{lem:boundedX}
  If $\X$ is a bounded connected subset of $\tree$ then $\Wh(\X)$ is
  connected without cut vertices.
\end{lemma}

\subsection{A First Connectivity Lemma}\label{sec:wgtd}
Recall that $\qmap\from\bdry\F\to\decomp$ is the quotient map.
\begin{lemma}[{\cite[Prop 2.1]{Ota92}\cite[Theorem~49]{Mar95}}]\label{lemma:basicconnectivity}
Let $e$ be an edge in $\tree$ with endpoints $u$ and $v$.
Let $\mathcal{A}=\shadow^u(v)$.
The set $A=\qmap(\mathcal{A})$ is connected in $\decomp$.
\end{lemma}
We include the proof for completeness:
\begin{proof}
Suppose $B,\,C\subset\decomp$ are open with
$A\subset B\cup C$ and $A\cap B\cap C=\emptyset$.
Let $\mathcal{B}=\mathcal{A}\cap \qmap^{-1}(B)$ and
$\mathcal{C}=\mathcal{A}\cap \qmap^{-1}(C)$.
Compactness of $\bdry \tree$ implies $\mathcal{B}$ and $\mathcal{C}$ are compact clopens. 
Since $\mathcal{B}$ is compact and open, there are finitely many vertices
$x_1,\dots ,x_a$ so that $\mathcal{B}=\cup_{i=1}^a\shadow^u(x_i)$. 
Assume $x_1,\dots ,x_a$ contains as few points as possible, ie, the
shadows of a
proper subcollection of the $x_i$ have union a proper subset of $\mathcal{B}$.
There is a similar finite collection $y_1,\dots , y_b$ that determines $\mathcal{C}$. 

Assume $\mathcal{B}$ is non-empty.
If $x_1=v$ then $\mathcal{B}=\mathcal{A}$, so $\mathcal{C}=\emptyset$,
and we are done.
Otherwise, consider the convex hull $\hull$ of $\{x_i\}_{i=1}^a\cup\{y_j\}_{j=1}^b\cup
\{u\}$; it is a finite tree with leaves $\{x_i\}_{i=1}^a\cup\{y_j\}_{j=1}^b\cup
\{u\}$.
Let $\X=\hull\setminus (\{x_i\}_{i=1}^a\cup\{y_j\}_{j=1}^b\cup
\{u\})$.
By \fullref{lem:boundedX}, 
$\Wh(\X)$ is connected without cut vertices.

Consider the vertex $\gr(x_1)\in\Wh(\X)$.
Since $\gr(u)$ is not a cut vertex, there are edges of $\Wh(\X)$ incident
to $\gr(x_1)$ and not $\gr(u)$.
Such an edge corresponds to a line $\line\in\lines$ with one
endpoint in the shadow of $x_1$ and the other endpoint in the shadow
of $z$ for some $z\in \{x_i\}_{i=2}^a\cup\{y_j\}_{j=1}^b$.
In the decomposition space these two endpoints are identified, so they
must both be in $\mathcal{B}$, hence $z\in \{x_i\}_{i=2}^a$.
Since $\Wh(\X)$ is connected and $\gr(u)$ is not a cut vertex, we conclude
that all the vertices of $\Wh(\X)$ except $\gr(u)$ belong to $\{\gr(x_1),\dots
,\gr(x_n)\}$, so $\mathcal{C}=\emptyset$. 
Thus, $A$ is connected.
\end{proof}

Together with \fullref{proposition:Whitehead}, this implies:
\begin{theorem}\label{thm:freesplitting}
Let $\multiklass$ be an arbitrary multiclass (not necessarily
satisfying the standing assumption of \fullref{standingassumption}).
  $\F$ admits a free splitting relative to $\multiklass$ if and only if
  $\decomp_{\multiklass}$ is not connected.
\end{theorem}
\begin{proof}
Choose a basis $\basis$ for $\F$ such that $\multiklass$ is
Whitehead minimal.
By \fullref{proposition:Whitehead}, $\F$ splits freely rel
$\multiklass$ if and only if $\Wh_{\basis}(\bp)\{\multiklass\}$ is not
connected.
If $\Wh_{\basis}(\bp)\{\multiklass\}$ is not connected then $\decomp$ is
not connected by \fullref{corollary:disconnected}. 

Suppose $\Wh_{\basis}(\bp)\{\multiklass\}$ is connected.
Pick an edge $e\in\tree$ with vertices $u$ and $v$.
Let $A=\shadow^u(v)\subset \bdry\tree$ and
$A'=\shadow^v(u)\subset\bdry\tree$.
Note $\bdry\tree=A\coprod A'$.
Whitehead minimality implies that $\Wh_{\basis}(\bp)\{\multiklass\}$ has no
cut vertex, so \fullref{lemma:basicconnectivity} applies.
Moreover, $\qmap(A)$ and $\qmap(A')$ have a point in common, since
$\Wh_{\basis}(\bp)\{\multiklass\}$ connected implies there is some line
in $\lines$ crossing $e$. 
Thus, $\decomp=\qmap(A)\cup\qmap(A')$ is a union of connected sets
with non-empty intersection, so it is connected.
\end{proof}


\section{The Topology of the Decomposition
  Space}\label{sec:topologyofdecompositionspace}
In this section we prove some results for later use.
We begin, in \fullref{sec:decompositionspaces} by recalling Moore's
Decomposition Theorem, which will be needed in \fullref{sec:vg}.
This theorem requires a technical hypothesis called `upper
semi-continuity'.

In \fullref{sec:saturated} we explicitly construct a neighborhood
basis for $\decomp$. 
It follows quickly that the decomposition we consider is upper
semi-continuous and that $\decomp$ is metrizable.
Furthermore, when $\decomp$ is connected it is locally connected and
arc connected.

In the two remaining subsections we prove technical results that will be
used in \fullref{sec:splitting} to show that when the \rjsj is
non-trivial there do exist cut points or uncrossed cut pairs of
$\decomp$, and that uncrossed cut pairs are rational.

The purpose of \fullref{sec:morewgtd} is to make rigorous the idea
that we can `see' a neighborhood of $x\in\decomp$ by looking at
components of a Whitehead graph over a ray in $\tree$ tending towards a point of
$\qmap^{-1}(x)$ in $\bdry\F$.
Precisely, given a point
$x\in\decomp$ and a neighborhood $N'$ of $x$ there exists a connected
open neighborhood $N\subset N'$ of $x$ such that
for any simplicial geodesic $\phi\from
  [0,\infty]\to\closure{\tree}$ with $\phi(\infty)\in\qmap^{-1}(x)$
 one of the following
holds:
\begin{itemize}
\item $\qmap^{-1}(x)\notin \bdry\multimot$ and there is a bijection
  between the components of $N\setminus \{x\}$ and components of
  $\Wh(\phi([1,\infty]))\wrel \phi([0,\infty])$.
\item $\qmap^{-1}(x)\in \bdry\multimot$ and there is a bijection
  between the half of the components of $N\setminus \{x\}$ and components of
  $\Wh(\phi([1,\infty]))\wrel \phi([0,\infty])$.
\end{itemize}
In the second case we only see half the components of
$N\setminus\{x\}$ because $|\qmap^{-1}(x)|=2$. 
The other
half of the components correspond to components of a Whitehead graph over a ray tending
towards the other point in $\qmap^{-1}(x)\in\bdry\F$.

There are two main results of \fullref{sec:detect}.
One, \fullref{lemma:cutpointsareinmultiword}, is a refinement of a rational approximation result from
\cite{CasMac11}.
The other, \fullref{lemma:partialcutpairhull}, says that if
$\mathcal{H}$ is the convex hull in $\closure{\tree}$ of the preimage
of a cut pair of $\decomp$ then at every point $x$ of $\hull$ the
Whitehead graph $\Wh(x)\wrel\hull$ has at least two components.

\subsection{Aside on General Decomposition
  Spaces}\label{sec:decompositionspaces}
We recall some results about more general
decomposition spaces.

Let $X$ be a topological space. Let $X=\coprod_{i\in I}X_i$
be a decomposition of $X$ as a disjoint union.
The \emph{non-degenerate} elements are the non-singleton $X_i$.
The \emph{decomposition space} is the index set $I$ with the
quotient topology induced by $X_i\mapsto i$.

\begin{definition}\label{def:usc}
  A decomposition $X=\coprod_{i\in I}X_i$ is \emph{upper
    semi-continuous} if for each $i\in I$ the set $X_i$ is
  compact and for each open set $U\supset X_i$ there
  is an open set $V\supset X_i$ such that for all
  $j\in I$ if $X_j\cap V$ is non-empty then
  $X_j\subset U$.
\end{definition}

\begin{lemma}[{\cite[Theorem 3-33]{HocYou88}}]\label{lemma:Hausdorff}
  The decomposition space of an upper semi-continuous decomposition of
  a compact Hausdorff space is Hausdorff.
\end{lemma}

The following theorem will be used in \fullref{sec:vg}.
\begin{theorem}[Moore's Decomposition Theorem \cite{Moo25}]\label{theorem:Moore}
  An upper semi-continuous decomposition of the 2--sphere into
  connected, non-separating sets has decomposition space homeomorphic
  to the 2--sphere.
\end{theorem}

\subsection{Basic Topology of the Decomposition
  Space}\label{sec:saturated}
We now return our attention to the decomposition space $\decomp$ of $\bdry\F$
associated to $\multimot$.

Let $\lines|_\X=\{\l\in\lines\mid
\l\cap\X\neq\emptyset\}$.
If $\X$ is bounded, $\lines|_\X$ is finite.

\subsubsection{Neighborhood Basis}\label{sec:nbhdbasis}
The goal of this section is to describe a neighborhood basis of
$\decomp$.
A set $\mathcal{A}$ in $\bdry\tree$ is \emph{saturated} if
$\mathcal{A}=\qmap^{-1}(\qmap(\mathcal{A}))$.
The $\qmap$--image of an open saturated set is open.
First we build open saturated neighborhoods of points $\xi\in\bdry\tree$.
The construction proceeds in stages. The idea is to start with a basic neighborhood
of $\xi$ in $\bdry\tree$. This neighborhood might not be saturated, so we add some points
to make it saturated. The resulting set might not be open, so add
basic neighborhoods of all of the newly added points, and repeat.

Let $\maxoverlap$ be the `maximum overlap' in $\lines$: the maximum length of $\l\cap \l'$
for distinct lines $\l$ and $\l'$ in $\lines$. This is finite, since $\multimot$ consists of finitely many words. 

Let $\xi_0$ be a point in $\bdry\tree$, and let $r\in\mathbb{N}$.

Let $\phi_0\from [0,\infty]\onto [\bp,\xi_0]$ be the simplicial geodesic.
If $\xi_0$ is the endpoint of a line $[\xi_0,\xi_1]\in\lines$ let
$r_0=\max\{r,1+d_\tree(\bp,[\xi_0,\xi_1])\}$.
Otherwise, there is at most one line of $\lines$ with one endpoint in $\shadow^\bp(\phi_0(r+M+1))$
and the other in the complement of $\shadow^\bp(\phi_0(r))$.
If there are none, let $r_0=r$.
If there is one, $\line$, let $r_0>r$ be the minimal number such that $\line\cap [\bp,\xi_0]\subset\phi_0([0,r_0+M])$.
Then no line of $\lines$ has one endpoint in $\shadow^\bp(\phi_0(r_0+M+1))$
and the other in the complement of $\shadow^\bp(\phi_0(r_0))$, since
such a line would overlap $\line$ by more than $\maxoverlap$.

Consider two neighborhoods of $\xi_0$, an `inner neighborhood'
$\mathcal{S}_0=\shadow^\bp(\phi_0(r_0+\maxoverlap+1))$ contained in an `outer neighborhood'
$\mathcal{O}_0=\shadow^\bp(\phi_0(r_0))$.
Let $\lines_0'$ be the lines of $\lines$ with exactly one
endpoint in $\mathcal{S}_0$ and both
endpoints in $\mathcal{O}_0$.

If $\xi_0$ does not belong to an element of $\bdry\multimot$ then
$\lines_0'$ accounts for all lines with exactly one endpoint in
$\mathcal{S}_0$ and we define `stage 0' to be $\{\mathcal{S}_0\}$.
Otherwise $\line=[\xi_0,\xi_1]$ is the unique line with one endpoint in
$\mathcal{S}_0$ and one outside of $\mathcal{O}_0$.
Let $\phi_1\from [0,\infty]\onto[\bp,\xi_1]$
be the simplicial geodesic.
Let $\mathcal{S}_1=\shadow^\bp(\phi_1(r_0+\maxoverlap+1))$ and $\mathcal{O}_1=\shadow^\bp(\phi_1(r_0))$.
Let
$\lines_1'=\left(\lines|_{\phi_1(r_1+\maxoverlap)}\cap\lines|_{\phi(r_1+\maxoverlap+1)}\right)\setminus
\{\l\}$.
Define `stage 0' to be $\{\mathcal{S}_0, \mathcal{S}_1\}$.
Note in this case $\qmap(\line^-)=\qmap(\line^+)\in\qmap(\mathcal{S}_0)\cap\qmap(\mathcal{S}_1)$.

Suppose now we have constructed stage $k$.
Let $r_{k+1}=r_k+\maxoverlap+1$.
Let $I$ be the indices of the stage $k$ sets.
For $i\in I$ and $\l_{i,j}\in\lines_i'$, 
let $\phi_{i,j}\from[0,\infty]\onto [\bp,\xi_{i,j}]$ be the simplicial geodesic, where
$\xi_{i,j}$ is the endpoint of $\l_{i,j}$ that is not in $\mathcal{S}_i$.
Define an inner neighborhood
$\mathcal{S}_{i,j}=\shadow^\bp(\phi_{i,j}(r_{k+1}+\maxoverlap+1))$ and outer
neighborhood $\mathcal{O}_{i,j}=\shadow^\bp(\phi_{i,j}(r_{k+1}))$ of $\xi_{i,j}$.

For fixed $i$ and all $j$ we have
$\mathcal{O}_{i,j}\subset\mathcal{O}_i$ and
$\mathcal{O}_{i,j}$ is disjoint from $\mathcal{S}_i$ and from every
$\mathcal{O}_{i,j'}$ for $j'\neq j$.

Let $\lines_{i,j}'$ be the lines of $\lines$ with exactly one
endpoint in $\mathcal{S}_{i,j}$ and with both endpoints in $\mathcal{O}_{i,j}$.
We know that there is one line of $\lines$ with an endpoint in
$\mathcal{S}_{i,j}$ and the other not in $\mathcal{O}_{i,j}$ --- it is
$\l_{i,j}$, and its other endpoint is in $\mathcal{S}_i$.
There can be no other such line, for it would overlap $\line_{i,j}$ by
more than $\maxoverlap$.

Call $\mathcal{S}_i$ the \emph{predecessor} of $\mathcal{S}_{i,j}$. 
Let `stage $k+1$' be the sets $\mathcal{S}_{i,j}$ produced from the sets of stage $k$.
By construction, the stage $k+1$ sets are disjoint from each other and
from all sets in the previous stages.
This is true for their $\qmap$ images in $\decomp$ as well, except
that $\qmap(\mathcal{S}_{i,j})\cap\qmap(\mathcal{S}_i)=\qmap(\line_{i,j}^-)=\qmap(\line_{i,j}^+)$.

Repeat this construction for as many stages as possible, potentially
infinitely many, to produce a
collection $\{\mathcal{S}_\tau\}_\tau$ of disjoint open sets.
Let $\mathcal{S}(r)=\cup_\tau\mathcal{S}_\tau$.
By construction, $\mathcal{S}(r)$ is an open saturated neighborhood of
$\xi_0$, so its image $S_{\xi_0}(r)=\qmap(\mathcal{S}(r))$ is an open neighborhood
of $\qmap(\xi_0)$.
\begin{definition}
For $x\in\decomp$ pick $\xi_0\in\qmap^{-1}(x)$ and define $\nbhd(x,r)=S_{\xi_0}(r)$. 
\end{definition}
By construction, $\nbhd(x,r)$ does not depend on the 
choice of $\xi_0\in\qmap^{-1}(x)$. 

It is immediate from the construction that: 
\begin{proposition}\label{proposition:neighborhoodbasis}
 $\{\nbhd(x,r)\mid r\in\mathbb{N}\}$ is a neighborhood basis for $x$.
\end{proposition}

\subsubsection{Upper Semi-continuity}
\begin{proposition}\label{proposition:usc}
  The decomposition of $\bdry\tree$ whose non-degenerate elements are $\bdry\multimot$ is an
  upper semi-continuous decomposition.
\end{proposition}
\begin{proof}
This  follows directly from
\fullref{sec:nbhdbasis} since the sets $S_{\xi_0}$ are saturated.
\end{proof}

\begin{corollary}\label{proposition:hausdorff}
  $\decomp$ is Hausdorff.
\end{corollary}

\begin{proposition}\label{corollary:metrizable}
  $\decomp$ is metrizable.
\end{proposition}
\begin{proof}
 $\qmap\from \bdry\tree\to\decomp$ is a continuous map from a compact
  space to a Hausdorff space, by \fullref{proposition:hausdorff}, so it
  is a closed map.
The codomain is first-countable by \fullref{proposition:neighborhoodbasis}.
The domain is metrizable, so a theorem of Stone \cite[Theorem 1]{Sto56} says $\decomp$ is metrizable.
\end{proof}

\subsubsection{Connectivity}

\begin{lemma}\label{lemma:foreshadowconnected}
For every $x\in\decomp$ and $r\in\mathbb{N}$, the set $\nbhd(x,r)$ is connected.
\end{lemma}
\begin{proof}
  \fullref{lemma:basicconnectivity} says that for each of the sets $\mathcal{S}_\tau$ in the
  construction of $\nbhd(x,r)$, the set $\qmap(\mathcal{S}_\tau)$
  is connected in $\decomp$.
If there are two stage 0 sets their $\qmap$--images
have a point in common, and the $\qmap$--image of every set from a
higher stage has a point in common with the $\qmap$--image of its predecessor.
Thus $\nbhd(x,r)$ is connected.
\end{proof}

\begin{proposition}[{cf \cite[Proposition~2.1]{Ota92}}]\label{corollary:connected}
$\decomp$ is connected and locally connected.
\end{proposition}
\begin{proof}
It is locally connected by \fullref{lemma:foreshadowconnected}.

Let $u$ and $v$ be neighboring vertices in $\tree$.
Let $\mathcal{A}=\shadow^u(v)$. Then
$\mathcal{A}^c=\bdry\tree\setminus\mathcal{A}=\shadow^v(u)$.
By \fullref{lemma:basicconnectivity}, $\qmap(\mathcal{A})$ and
$\qmap(\mathcal{A}^c)$ are each connected.
$\Wh(\bp)$ is connected, so there exists
$\line\in\lines|_u\cap\lines|_v$, and
$\qmap(\line^-)=\qmap(\line^+)\in\qmap(\mathcal{A})\cap\qmap(\mathcal{A}^c)$.
\end{proof}

An \emph{arc} is an embedded path.
A space is \emph{arc-connected} if any two points can be joined by an arc.
A space is \emph{Peano} if it is compact, Hausdorff, connected,
locally connected, and metrizable.\footnote{Alternatively, a space is Peano if it is a continuous image of the
unit interval. These definitions are equivalent for Hausdorff spaces by the Hahn-Mazurkiewicz Theorem.}
\begin{proposition}\label{lem:peano}
  $\decomp$ is Peano.
\end{proposition}
\begin{proof}
$\decomp$ is compact. It is Hausdorff by
\fullref{proposition:hausdorff}, connected and locally connected by
\fullref{corollary:connected}, and metrizable by \fullref{corollary:metrizable}.
\end{proof}

\begin{theorem}[{\cite[Theorem~31.2]{Wil04}}]
  A Peano space is arc-connected.
\end{theorem}
\begin{corollary}\label{corollary:arcconnected}
    $\decomp$ is arc-connected.
\end{corollary}

\subsection{Whitehead Graphs and Cut Sets of the
  Decomposition Space}\label{sec:morewgtd}

\begin{lemma}\label{lemma:rationalcutpoints}
  If $x\in\decomp$ such that $|\qmap^{-1}(x)|=1$ then $x$ is not a
  cut point.
\end{lemma}
\begin{proof}
  Let $\xi=\qmap^{-1}(x)$. 
Let $\phi\from [0,\infty] \onto [\bp,\xi]$ be the simplicial geodesic.
For all $i\in\mathbb{N}$ the set $\qmap(\shadow^{\phi(i)}(\phi(i-1)))$ is
connected, by \fullref{lemma:basicconnectivity}.
 $\decomp\setminus \{x\}$ is an increasing union of such sets,
so it is connected.
\end{proof}
\begin{proposition}[{cf \cite[Lemma 4.9]{CasMac11}}]\label{lemma:connectedhull}
  Let $S\subset\decomp$ be closed with $|\qmap^{-1}(S)|>1$, and let $\hull$ be the convex hull of $\qmap^{-1}(S)$. 
Then $\wc\mapsto \qmap(\bdry\tree_\wc)$ is a bijection between components of $\Wh(\hull)$ and components of $\decomp\setminus S$.
\end{proposition}
\begin{proof}
Let $\wc$ be a component of $\Wh(\hull)$.
Let $\wv$ be a vertex of $\wc$.
By \fullref{lemma:basicconnectivity}, $\qmap(\bdry\tree_\wv)$ is connected in
$\decomp$.
If vertices $\wv$ and $\wv'$ are joined by an edge $\we$ in $\wc$
then, by definition, 
$\l_\we\in\lines$ is a line with one endpoint in $\bdry\tree_\wv$ and
the other in $\bdry\tree_{\wv'}$, so $\qmap(\bdry\tree_\wv)$ and
$\qmap(\bdry\tree_{\wv'})$ have a point in common.
This implies $\qmap(\bdry\tree_\wc)$ is
connected.
It is also open, since $\bdry\tree_\wc$ is open and saturated.

Since $S$ is closed, for every
$\xi\in\bdry\tree\setminus\qmap^{-1}(S)$ there is a vertex $v\in\hull$
such that $[v,\xi]\cap\hull=\{v\}$.
Therefore, $\gr_\hull(\xi)$ is a vertex of $\Wh(\hull)$
with $\xi\in\bdry\tree_{\gr_\hull(\xi)}$.

Letting $\wc$ range over all components of $\Wh(\hull)$, we
get disjoint, connected, open subsets of the form
$\qmap(\bdry\tree_\wc)$, whose union is all of $\decomp\setminus S$.
\end{proof}

\begin{lemma}\label{lemma:ends}
  Let $S\subset\decomp$ be a closed minimal cut set, and let $\hull$ be the convex hull of $\qmap^{-1}(S)$.
For all $x\in S$, each component $\wc$ of $\Wh(\hull)$ has an end at
a point in $\qmap^{-1}(x)$.
\end{lemma}
\begin{proof}
  $\wc$ has an end at a point in $\qmap^{-1}(x)$ if and only if $\qmap(\bdry\tree_\wc)$ has $x$ as
  a limit point, which it must, by \fullref{lemma:limitpoints}.
\end{proof}

\begin{definition}
Let $\phi\from [0,l]\to\closure{\tree}$ be a simplicial geodesic
for some $l\in\mathbb{N}\setminus \{1\}$.
  An edge path $\we_0,\dots, \we_k$ in $\Wh(\phi([1,l-1]))\wrel \phi([0,l])$ is
  \emph{non-backtracking} if the path in $\tree$ obtained by
  concatenating the segments $\line_{\we_i}\cap\phi([0,l])$ does not backtrack.
\end{definition}

\begin{lemma}\label{lemma:nonbacktracking}
For $l\in\mathbb{N}\setminus \{1\}$, let $\phi\from [0,l]\to\tree$ be
a simplicial geodesic.
Let $\we$ be an edge of $\Wh(\phi([1,l-1]))\wrel \phi([0,l])$.
There exists a non-backtracking edge path $\we=\we_0,\dots, \we_k$ with 
$\l_{\we_k}\in\lines|_{\phi(0)}$.
\end{lemma}
\begin{proof}
Let $r$ be the smallest integer such that $\phi(r)\in\line_{\we_0}$.
If $r\leq 0$ we are done.
Suppose not.
By minimality of $r$, the edge $\gr_{\phi(r)}(\l_{\we_0})$ is incident to an undeleted vertex $\wv\in\Wh(\phi(r))\wrel\phi([0,l])$.

$\Wh(\phi(r))$ is connected without cut vertices, so the vertex $\gr_{\phi(r)}(\phi(r+1))$ is not a cut vertex.
Thus, there is an edge path $\we'_1,\dots, \we'_j$ in
$\Wh(\phi(r))$ connecting $\wv$ to
$\gr_{\phi(r)}(\phi(r-1))$ that does not go through the vertex
$\gr_{\phi(r)}(\phi(r+1))$.
Choosing the shortest such path guarantees that
$\l_{\we_i'}\cap\phi([0,l])=\phi(r)$ for all $i<j$, and
$\l_{\we_j'}\cap\phi([0,l])=\phi([r',r])$ for some $r'<r$.
Extend the existing edge path by edges
$\we_i=\gr_{\phi([1,l-1])}(\l_{\we_i'})$.
This gives a non-backtracking edge path beginning with $\we$ that
reaches closer to $\phi(0)$. Proceed by induction.
\end{proof}

The number of lines in $\lines|_e$ for an edge $e\in\tree$
corresponding to basis element $\basic$ is equal to the valence of the
$\basic$--vertex in $\Wh(\bp)$.

\begin{corollary}\label{corollary:boundedraycomponents}
  The
maximum valence of a vertex in $\Wh(\bp)$ is an upper bound for the
number of components of $\Wh(\phi([1,\infty]))\wrel\phi([0,\infty])$.
\end{corollary}

\begin{proposition}\label{proposition:finitelymanycomponents}
  $\decomp\setminus S$ has finitely many
  components for every finite $S\subset\decomp$.
\end{proposition}
\begin{proof}
Let $\{\xi_1,\dots, \xi_k\}=\qmap^{-1}(S)$, and suppose
$\decomp\setminus S$ is not connected.
Then $k\geq 2$, by \fullref{lemma:rationalcutpoints}.
  Let $\hull$ be the convex hull of
  $\qmap^{-1}(S)$.
If $k=2$, let $\X$ be an arbitrary vertex in $\hull$; otherwise, let $\X$ be the convex hull of the branch points of
$\hull$.
For $1\leq i\leq k$, let 
$\phi_i\from[0,\infty]\into\closure{\tree}$ be the simplicial geodesic
ray with $\phi_i([0,\infty])\cap\X=\phi(0)$ and $\phi(\infty)=\xi_i$.
By \fullref{lemma:cutting}, $\Wh(\hull)$ is obtained by splicing the
$\Wh(\phi_i([1,\infty]))\wrel\hull$ to $\Wh(\X)\wrel\hull$.
The former have finitely many components by
\fullref{corollary:boundedraycomponents}, and the later is finite, so
$\Wh(\hull)$ has finitely many components.
By \fullref{lemma:connectedhull}, 
$\decomp\setminus S$ has finitely many components. 
\end{proof}

\begin{lemma}\label{lemma:connectedray}
Let $\phi\from [-\infty,\infty]\onto [\inv{g}^\infty,g^\infty]$ be a
simplicial geodesic for some cyclically reduced $g\in\F$.
Let $\line$, $\line'\in\lines|_{\phi((0,1))}$.
Then $\gr_{\phi([-\infty,\infty])}(\line)$ and
$\gr_{\phi([-\infty,\infty])}(\line')$ are in the same component of
$\Wh(\phi([-\infty,\infty]))$ if and only if
$\gr_{\phi([1,\infty])}(\line)$ and
$\gr_{\phi([1,\infty])}(\line')$ are in the same component of $\Wh(\phi([1,\infty]))\wrel\phi([0,\infty])$.
\end{lemma}
\begin{proof}
$\Wh(\phi([1,\infty]))\wrel\phi([0,\infty])$ includes into $\Wh(\phi([-\infty,\infty]))$ via splicing, so the `if' direction is clear.

For the converse, consider the finitely many lines in
$\lines|_{\phi((0,1))}$.
For every pair $\line_i$ and $\line_j$ that contribute edges to a common component of
$\Wh(\phi([-\infty,\infty]))$, choose an edge path $\mathfrak{P}_{i,j}$ connecting them. 

Define a $g$--action on $\Wh(\phi([-\infty,\infty]))$ by sending edge
$\we$ to edge $\gr_{\phi([-\infty,\infty])}(g\line_\we)$.
The $g$--action permutes components, and by replacing $g$ with a
suitable power we may assume the permutation is trivial.
Let $m$ be sufficiently large so that for all $i$ and $j$ and all
$\we\in\mathfrak{P}_{i,j}$, we have
$g^m\line_\we\cap\phi([-\infty,\infty])\subset\phi([1,\infty])$.

By \fullref{lemma:nonbacktracking}, there is a non-backtracking edge
path connecting $\gr_{\phi([1,\infty])}(\line)$ to an edge $\we$ such
that $\line_\we\in \lines|_{g^m\phi(1)}$, and a similarly defined edge
$\we'$ for $\line'$.
Since the $g$--action is trivial on components, $\gr_{\phi([-\infty,\infty])}(\line_\we)$ and $\gr_{\phi([-\infty,\infty])}(\line_{\we'})$ are
in the same component of $\Wh(\phi([-\infty,\infty]))$.
Therefore, they are connected in
$\Wh(\phi([1,\infty]))\wrel\phi([0,\infty])$ by one of the $g^m\mathfrak{P}_{i,j}$.
Concatenating these three paths connects $\gr_{\phi([1,\infty])}(\line)$ and
$\gr_{\phi([1,\infty])}(\line')$ in $\Wh(\phi([1,\infty]))\wrel\phi([0,\infty])$.
\end{proof}

\begin{lemma}\label{lemma:niceneighborhood}
For every $x\in\decomp$ and every neighborhood $N'$ of $x$ and every
simplicial geodesic $\phi\from [0,\infty]\to \closure{\tree}$ with $\qmap(\phi(\infty))=x$
there exists a connected open neighborhood $N\subset N'$ of $x$
such that:
  \[\# (N\setminus \{x\}) = |\qmap^{-1}(x)|\lim_{r\to\infty}\#\left
    (\Wh(\phi([r,\infty]))\wrel\phi([0,\infty])\right)\]
Here, $|\cdot|$ denotes cardinality and $\#(\cdot)$ denotes number of
components.
\end{lemma}
Note that $x$
  is a cut point in $\decomp$ if and
  only if: \[\lim_{r\to\infty}\#\left
    (\Wh(\phi([r,\infty]))\wrel\phi([0,\infty])\right)>1\text{ and
    }|\qmap^{-1}(x)|=2\]
\begin{proof}
As $r$ increases,
$\#\left(\Wh(\phi([r,\infty]))\wrel\phi([0,\infty])\right)$ is a
non-decreasing sequence of integers that by
\fullref{corollary:boundedraycomponents} is bounded above.
Thus, for any sufficiently large $r_1$ we may assume that the limit has
been achieved and that 
$\qmap(\shadow^{\phi(0)}(\phi(r_1)))\subset N'$.
Let $\wc_1,\dots,\wc_k$ be the components.
For each $i$, let $\line_{i,1},\dots \line_{i,m_i}$ be the lines of
$\lines$ that cross the edge $\phi((r_1-1,r_1))$ and satisfy
$\gr_{\phi([r_1,\infty])}(\line_{i,j})\in\wc_i$.

For each $\line_{i,j}$ take a connected open neighborhood $N_{i,j}$ of
$\qmap(\line_{i,j}^+)$ as in \fullref{sec:nbhdbasis}.
We can take these neighborhoods small enough so that they are disjoint
and contained in $N'$.
For each $i$ define $N_i=
\qmap(\bdry\tree_{\wc_i})\cup\bigcup_jN_{i,j}$. 
The $N_i$ are disjoint, connected, open sets that all have $x$ as a limit point, and
$x$ is the only limit point that any two have in common. 

If $\qmap^{-1}(x)=\phi(\infty)$ then let $N=\{x\}\cup \bigcup_i N_i$
and we are done.
Otherwise, let $\phi'$ be a simplicial geodesic ray converging to the
other point in $\qmap^{-1}(x)$, and
repeat the construction to produce
connected open sets $N'_1,\dots, N'_k$.
The number $k$ of components is the same since there is a group
element acting cocompactly on $[\phi'(\infty),\phi(\infty)]$.
Let $N=\{x\}\cup \bigcup N_{i}\cup\bigcup N'_{i}$.
\end{proof}

\begin{proposition}\label{lemma:niceneighborhood2}
Let $\phi\from [-\infty,\infty]\onto [h\inv{\mot}^\infty,h\mot^\infty]$ be a
  simplicial geodesic for some $\mot\in\multimot$. 
Let $x_0=\qmap(h\mot^\infty)$, and suppose
  $x_0$ is not a cut point.
\[\lim_{r\to\infty}\#(\Wh(\phi([r,\infty]))\wrel\phi([0,\infty]))=\#(\Wh(\phi([1,\infty]))\wrel\phi([0,\infty]))=1\]
Furthermore, for every neighborhood $N'$ of $x_0$
  there exists an open connected neighborhood $N\subset N'$ of $x_0$
  such that $N\setminus\{x_0\}$ has precisely two components.
\end{proposition}
\begin{proof}
The first statement is true because $\left< w\right>$ acts cocompactly on $\phi((-\infty,\infty))$
and $x_0$ is not a cut point.
The second statement is an application of \fullref{lemma:niceneighborhood}.
\end{proof}

The next fact will be used in \fullref{sec:vg}, but it is
convenient to prove it now:
\begin{proposition}\label{corollary:arcconnectedcomponents}
  If $A$ is the closure of a complementary component of a cut point
  or cut pair in $\decomp$, then $A$ is arc-connected.
\end{proposition}
\begin{proof}
Let $x$ be the cut point or one of the points of the cut pair.
Let $N'$ be a neighborhood of $x$ in $\decomp$.
Construct open connected sets $N_i\subset N'$ as in 
\fullref{lemma:niceneighborhood}.
Let $N=\{x\}\cup \bigcup_{N_i\subset A} N_i$.
Then $N\subset A\cap N'$ is a connected set containing $x$ that is
open in $A$.
Therefore, $A$ is locally connected at $x$.

$A$ is arc-connected just as in \fullref{corollary:arcconnected},
since all of the other necessary properties are inherited from $\decomp$.
\end{proof}

\subsection{Cut Point and Cut Pair Detection and Approximation}\label{sec:detect}
\begin{proposition}\label{lemma:cutpointsareinmultiword}
If $x\in\decomp$ is a cut point then the stabilizer of $x$ is
conjugate to $\left< \mot\right>$ for some $\mot\in\multimot$.
In particular, there are finitely many orbits of cut points.
\end{proposition}
\begin{proof}
By \fullref{lemma:rationalcutpoints}, $\qmap^{-1}(x)$ consists of two
points, so $\qmap^{-1}(x)=\{\inv{h}\inv{\mot}^\infty,\inv{h}\mot^\infty\}$ for some $h\in \F$ and $\mot\in\multimot$.
The stabilizer of $x$ is therefore $\inv{h}\langle\mot\rangle h$.
\end{proof}

\begin{lemma}[Rational Approximation cf {\cite[Lemma 4.12]{CasMac11}}]\label{lemma:periodic}
Let $\phi\from [0,\infty]\to\closure{\tree}$ be a simplicial geodesic.
There exist elements $g$, $h\in \F$ and
$a\in\basis\cup\inv{\basis}$ such that:
\begin{enumerate}
\item the oriented edges $[h,ha]$ and $[gh,gha]$ belong to $\phi([0
,\infty])\cap[\inv{g}^\infty,g^\infty]$,
\item components of $\Wh([ha,gh])\wrel [h,gha]$ that
 are in different components of $\Wh(\phi([1,\infty]))\wrel\phi([0,\infty])$ are in different
 components of $\Wh([\inv{g}^\infty,g^\infty])$, and\label{item:consistent}
\item for each line $\l\in\lines|_{h}\cap\lines|_{ha}$, the lines
  $\l$ and $g\l\in\lines|_{gh}\cap\lines|_{gha}$ contribute
  edges to the same
  component of $\Wh(\phi([1,\infty]))\wrel\phi([0,\infty])$.\label{item:preserves}
\end{enumerate}  
Moreover, $g$ is conjugate to an element whose word length with respect to $\basis$
is bounded in terms of the rank of $\F$ and the maximum
valence among vertices in $\Wh(\bp)$, independent of $\phi$.
\end{lemma}
\begin{proof}
$n=\mathrm{rank}(\F)$. Let $x$ be the maximum valence of
$\Wh(\bp)$. Let $y$ be the $x$--th
Bell number, the number of distinct partitions of $x$ items into
nonempty subsets. Let
$z=1+(2n)^{y+2}$.
Along any directed segment $\X$ of $\phi([1,\infty])$ of length $z$ there is some
$a\in\basis\cup\inv{\basis}$ such that there are
  at least $y+2$ many directed $a$--edges in the segment.
Fix the first of these, $e=[g_0,g_0a]$.

Fix a numbering of the components of $\Wh(\phi([1,\infty]))\wrel \phi([0,\infty])$.
The set $\lines|_e$ of lines of $\lines$ that contain $e$ is finite.
Fix a numbering of them $1,\dots, k$.
Of course, $k\leq x$, by \fullref{corollary:boundedraycomponents}.
Partition them into subsets according to which component of
$\Wh(\phi([1,\infty]))\wrel \phi([0,\infty])$ the corresponding edge belongs.

Consider an element $g'\in \F$ such that the oriented edge $g'e$ is in
$\X$.
There is a bijection $\lines|_e\to\lines|_{g'e}:\l\mapsto g'\l$.
Push forward the numbering of $\lines|_e$ to $\lines|_{g'e}$, and
consider the partition of $1,\dots, k$ according to
which component of
$\Wh(\phi([1,\infty]))\wrel \phi([0,\infty])$ 
each line of $\lines|_{g'e}$ belongs to.
So $g'$ gives a new partition of $1,\dots,k$.

There are at least $y+1$ such elements $g'$, but at most $y$ distinct
partitions of the numbers $1,\dots,k$, so there exist $g_1$ and $g_2$ such that the oriented
edges $g_1e$ and $g_2e$ are edges of $\phi([1,\infty])$ (say, with $g_2e$
between $g_1e$ and $\phi(\infty)$) and for each line
$\l\in\lines$ containing $g_1e$, the corresponding line $g_2\inv{g}_1\l$
containing $g_2e$ is in the same component of $\Wh(\phi([1,\infty]))\wrel \phi([0,\infty])$.

The desired elements are $h=g_1g_0$ and $g=g_2\inv{g_1}$.
The word length of $\inv{h}gh$ is the distance from $g_1e$ to $g_2e$,
which is at most $z$.
\end{proof}

\begin{corollary}\label{corollary:approximatinggiscut}
  With notation as above,
  $\decomp\setminus\qmap(\{\inv{g}^\infty,g^\infty\})$ has
  at least as many components as
  $\Wh(\phi([1,\infty]))\wrel \phi([0,\infty])$.
\end{corollary}

\begin{corollary}\label{corollary:no3componentcutpairsimpliesno3componentrays}
  If $\decomp$ does not have any cut points or cut pairs with more
  than two complementary components, then for every simplicial
  geodesic $\phi\from [0,\infty]\to\closure{\tree}$ there are at
  most two components of $\Wh(\phi([1,\infty]))\wrel \phi([0,\infty])$.
\end{corollary}
 
\begin{lemma}\label{lemma:lookslikeacutpair}
Let $\phi\from[-\infty,\infty]\to\closure{\tree}$ be a simplicial
geodesic such that $\qmap(\phi(\infty))$ and
$\qmap(\phi(-\infty))$ are distinct points and neither is a cut point.
If, for some increasing sequence $(r_i)$ of non-negative integers, there
are at least two components of every $\Wh(\phi([-r_i,r_i]))\wrel\phi([-r_i-1,r_i+1])$,
then $\qmap(\{\phi(-\infty),\phi(\infty)\})$ is a cut
pair.
\end{lemma}
\begin{proof}
$\Wh(\phi([-\infty,\infty]))$ has at least two components.
Let $\hull$ be the convex hull of
$\qmap^{-1}(\qmap(\{\phi(-\infty),\phi(\infty)\}))$.
By \fullref{lemma:connectedhull} we must show $\Wh(\hull)$ has at
least two components.
We are done if $\hull=\phi([-\infty,\infty])$.
Otherwise, we obtain $\Wh(\hull)$ from $\Wh(\phi([-\infty,\infty]))$
by deleting at most two vertices and splicing on connected graphs at
the deleted vertices.
Thus, $\Wh(\hull)$ has the same number of components as $\Wh(\phi([-\infty,\infty]))$.
\end{proof}

\begin{lemma}\label{lemma:leavesofcomponents}
Let $\hull$ be the convex hull of a subset of $\bdry\tree$, and let $\Y$ be a connected subset of $\hull$. 
Let $\wc$ be a component of $\Wh(\Y)\wrel\hull$,
and let $\X=\Y\cap \bigcup_{\we\in\wc}\line_\we$.
\begin{enumerate}
\item For every leaf $v$ of $\X$ there exists an edge $\we\in\wc$ with
  $v\in\line_\we$ and an edge $e\in\hull\setminus\Y$ incident to $v$ such that $\we$ has a loose end at $\gr_\Y(e)$.\label{item:leafhaslooseend}
\item The sum of the number of distinct ends and loose ends of $\wc$
  is at least two. \label{item:twodeletedvertices}
\end{enumerate}
\end{lemma}
\begin{proof}

  Let $v$ be a leaf of $\X$.
There is at most one edge $e_0$ of $\X$ incident
  to $v$.

Let $\we$ be an edge of $\wc$ such that $v\in\line_\we$.
Let $e_1\neq e_0$ be an edge of $\line_\we$ incident to $v$.
Since $v$ is a leaf of $\X$, $e_1$ is
not an edge of $\Y$, so there exists a vertex $\gr_\Y(e_1)\in\Wh(\Y)$.
If $e_1\in\hull\setminus\Y$ then we are done: $\we$ has a loose end at $\gr_\Y(e_1)$.
If there are no such edges then, since $\Wh(v)$ is connected without
cut vertices, it would mean that $v$ is a leaf of $\hull$, but $\hull$
has no leaves.
This proves
(\ref{item:leafhaslooseend}).

(\ref{item:twodeletedvertices}) follows directly from
(\ref{item:leafhaslooseend}) unless $\X$ is a single vertex.
If $\X=v$ then repeat the above argument with $e_0=e_1$ to find a
second edge $e_2\neq e_1$ such that $\wc$ also has an edge with a loose end
at $\gr_\Y(e_2)$.
\end{proof}

\begin{lemma}\label{lemma:twocomponentsplit}
  Let $S\subset\decomp$ be a closed minimal cut set that is not a cut point and contains a point $x$
  such that $|\qmap^{-1}(x)|=2$. For $\epsilon\in\pm$, let
  $\psi_x^\epsilon\from [0,\infty]\to \closure{\tree}$ be a simplicial
  geodesic ray such that $\psi_x^\epsilon([1,\infty])$ contains no branch
  point of the convex hull $\hull$ of $\qmap^{-1}(S)$ and such that
  $\{\psi_x^-(\infty),\psi_x^+(\infty)\}=\qmap^{-1}(x)$. 
Then for $\epsilon\in\pm$ there is a unique component
$\wc_\epsilon\subset\Wh(\hull)$ such that
$\psi_x^\epsilon([1,\infty])\cap
\bigcup_{\we\in\wc_\epsilon}\line_\we\neq\emptyset$, and these are the
only two components of $\Wh(\hull)$. Hence, $S$ has exactly two
complementary components.
\end{lemma}
\begin{proof}
By \fullref{lemma:niceneighborhood2}, each 
$\Wh(\psi_x^\epsilon([1,\infty]))\wrel\phi([0,\infty])$ is connected.
Since $\psi_x^\epsilon([1,\infty])$ contains no branch
  point of $\hull$, it follows that
  $\Wh(\psi_x^\epsilon([1,\infty]))\wrel\phi([0,\infty])=\Wh(\psi_x^\epsilon([1,\infty]))\wrel\hull$ includes into $\Wh(\hull)$ via splicing. 
Therefore, there is a single component $\wc_\epsilon\subset\Wh(\hull)$
containing the image of
$\Wh(\psi_x^\epsilon([1,\infty]))\wrel\phi([0,\infty])$.

If $C$ is a component of $\decomp\setminus S$ then
by \fullref{lemma:ends} the corresponding 
component $\wc$ of $\Wh(\hull)$ has an end at
$\psi_x^\epsilon(\infty)$ for at least one $\epsilon$, so $\wc$ is either $\wc_-$ or
$\wc_+$.
\end{proof}

\begin{lemma}\label{lemma:partialcutpairhull}
Let $\{x,y\}$ be a cut pair in $\decomp$.
Let $\hull$ be the convex hull of
$\qmap^{-1}(\{x,y\})$.
Let $\phi\from[0,l]\to\hull$ be a simplicial geodesic of finite
length $l\geq 2$.
Then $\Wh(\phi([1,l-1]))\wrel\phi([0,l])$ has at least two components.
\end{lemma}
\begin{proof}
If $\qmap^{-1}(\{x,y\})$ is two points then $\phi$ can be extended to
be a simplicial geodesic $\phi\from [-\infty,\infty]\onto \hull$.
It is an easy consequence of \fullref{lemma:nonbacktracking} that if 
$\Wh(\phi([1,l-1]))\wrel\phi([0,l])$ is connected then so is
$\Wh(\hull)$, so suppose $|\qmap^{-1}(\{x,y\})|>2$.

Let $\X=\phi([1,l-1])$ and $\Y=\phi([0,l])$.

For $z\in\{x,y\}$, if $|\qmap^{-1}(z)|=2$ then for $\epsilon\in\pm$ let $\psi_z^\epsilon\from [0,\infty]\to
\closure{\tree}$ be the simplicial geodesic
rays such that $\{\psi_z^-(\infty),\psi_z^+(\infty)\}=\qmap^{-1}(z)$
and such that $\psi_z^\epsilon(0)$ is the only branch point of
$\hull$ in the image of $\psi_z^\epsilon$.
Let $\line_z\in\lines$ be the line such that $\{\line_z^-,\line_z^+\}=\qmap^{-1}(z)$.

We make three preliminary claims:
\begin{claim}\label{claim:nodoublelooseend}
Let $\mathcal{Z}\subset\hull$ be connected. Let $z\in\{x,y\}$ with $|\qmap^{-1}(z)|=2$.
Then $\gr_{\mathcal{Z}}(\line_z)$ is the only possible component of
$\Wh(\mathcal{Z})\wrel\hull$ that, for
both $\epsilon\in\pm$,
contains edges $\we_\epsilon$ with
$\line_{\we_\epsilon}\cap\psi_z^\epsilon([1,\infty])\neq\emptyset$.
\end{claim}
\begin{claimproof}
Every component of $\Wh(\mathcal{Z})\wrel\hull$ that is not
equal to $\gr_{\mathcal{Z}}(\line_x)$ or $\gr_{\mathcal{Z}}(\line_y)$
includes into a component of $\Wh(\hull)$ via splicing. 
The claim follows from \fullref{lemma:twocomponentsplit}.
\end{claimproof}
\begin{claim}\label{claim:notribplelooseend}
  Let $\mathcal{Z}\subset\hull$ be connected. 
No component of $\Wh(\mathcal{Z})\wrel\hull$ has loose ends
  at three distinct deleted vertices. 
\end{claim}
\begin{claimproof}
If $\mathcal{Z}$ does not contain a branch point of $\hull$ then the claim is
trivial because  $\Wh(\mathcal{Z})\wrel\hull$ has at most two deleted
vertices.

If there is only one branch point, or if there are two and
$\mathcal{Z}$ contains both, or if
$\mathcal{Z}$ is disjoint from one of $\line_x$ or $\line_y$, then the claim follows from \fullref{claim:nodoublelooseend}.

The remaining possibility is
that $\line_x\cap\line_y$ contains an edge and $\mathcal{Z}$ contains
exactly one branch point of $\hull$.
Then $\Wh(\mathcal{Z})\wrel\hull$ has at most three deleted vertices,
which, without loss of generality, we may assume are
$\gr_{\mathcal{Z}}(\psi_x^+(\infty))$,
$\gr_{\mathcal{Z}}(\psi_y^+(\infty))$, and $\gr_{\mathcal{Z}}(\psi_x^-(\infty))=\gr_{\mathcal{Z}}(\psi_y^-(\infty))$.

Suppose $\wc$ is a component with loose
ends at all three deleted vertices. 
Let $\mathcal{Z}'=\mathcal{Z}\cup(\line_x\cap\line_y)$.
Consider the component $\wc'$ of
$\Wh(\mathcal{Z}')\wrel\hull$ containing
the image of $\wc$ under splicing $\Wh(\mathcal{Z})\wrel\hull$ to
$\Wh((\line_x\cap\line_y)\setminus\mathcal{Z})\wrel\hull$.
Since $\wc$ had loose ends at $\gr_{\mathcal{Z}}(\psi_x^+(\infty))$ and
$\gr_{\mathcal{Z}}(\psi_y^+(\infty))$, so does $\wc'$.
Since $\wc$ had a loose end at
$\gr_{\mathcal{Z}}(\psi_x^-(\infty))=\gr_{\mathcal{Z}}(\psi_y^-(\infty))$,
the set $\bigcup_{\we\in\wc'}\line_\we$ contains an edge in $\line_x\cap\line_y$.
Thus, $\mathcal{Z}'\cap \bigcup_{\we\in\wc'}\line_\we$ contains a leaf
$v\in\line_x\cap\line_y$ with $v\neq\psi_x^+(0)=\psi_y^+(0)$.
By \fullref{lemma:leavesofcomponents}, we conclude that
$v=\psi_x^-(0)=\psi_y^-(0)$ and that $\wc'$ has a loose end at either $\gr_{\mathcal{Z}'}(\psi_x^-(1))$ or $\gr_{\mathcal{Z}'}(\psi_y^-(1))$.
In either case this contradicts \fullref{claim:nodoublelooseend}.
\end{claimproof}

\begin{claim}\label{claim:componentpartition}
  Let $\mathcal{Z}\subset\hull$ be a bounded, connected set such that for
  $z\in\{x,y\}$ and $\epsilon\in\pm$ we have four distinct vertices
  $\gr_{\mathcal{Z}}(\psi_z^\epsilon(\infty))$.
For $\epsilon,\epsilon'\in\pm$, let $P^{\epsilon,\epsilon'}$ be the
set of components of $\Wh(\mathcal{Z})\wrel\hull$ that contain an
edge with a loose end at $\gr_{\mathcal{Z}}(\psi_x^\epsilon(\infty))$
and an edge with a loose end at
$\gr_{\mathcal{Z}}(\psi_y^{\epsilon'}(\infty))$.
Then one of the following is true:
\begin{itemize}
\item $P^{+,+}$ and $P^{-,-}$ are non-empty and $P^{+,-}$ and
  $P^{-,+}$ are empty.
\item $P^{+,-}$ and $P^{-,+}$ are non-empty and $P^{+,+}$ and
  $P^{-,-}$ are empty.
\end{itemize}
\end{claim}
\begin{claimproof}
\fullref{lemma:leavesofcomponents} and \fullref{claim:nodoublelooseend} imply that every component
except $\gr_{\mathcal{Z}}(\line_x)$ and $\gr_{\mathcal{Z}}(\line_y)$ belongs to one of the $P^{\epsilon,\epsilon'}$\!.
\fullref{claim:notribplelooseend} implies that no component belongs to
more than one of the $P^{\epsilon,\epsilon'}$\!.
If fewer than two of the $P^{\epsilon,\epsilon'}$ are non-empty then
$\Wh(\mathcal{Z})$ is not connected without cut vertices, which we
know it is.

\fullref{lemma:niceneighborhood2} implies that each
  $\Wh(\psi_z^\epsilon([1,\infty]))\wrel\psi_z^\epsilon([0,\infty])$
  is connected.
Therefore, by splicing, we see that all of the components in $P^{\epsilon_1,\epsilon_2}$ and
$P^{\epsilon_3,\epsilon_4}$ include into a common component of
$\Wh(\hull)$ if $\epsilon_1=\epsilon_3$ or $\epsilon_2=\epsilon_4$.
$\Wh(\hull)$ has at least two components, so the claim follows.
\end{claimproof}

We now proceed with the proof of the lemma.

If $\Y\subset\line_z$ for $z\in\{x,y\}$ then
$\Wh(\X)\wrel\Y$ has one component that
is the single edge $\gr_{\X}(\line_z)$ with two loose ends, and at
least one other component containing vertices, so at least two
components. 

Otherwise, if $\X$ does not contain a branch point of
$\hull$ then $\X$ separates $\qmap^{-1}(x)$ from
$\qmap^{-1}(y)$ and 
$\Wh(\X)\wrel\hull=\Wh(\X)\wrel\Y$.
If this is connected then
at most one component of $\Wh(\hull)$ has ends in $\qmap^{-1}(x)$ and $\qmap^{-1}(y)$, contradicting \fullref{lemma:ends}.

If $\Wh(\X)\wrel\hull$ has four deleted vertices then
partition the components into parts
$\gr_{\X}(\line_x)$, $\gr_{\X}(\line_y)$, and
the non-empty $P^{\epsilon,\epsilon'}$ of
\fullref{claim:componentpartition}. 
Assume, without loss of generality, that $P^{+,+}$ and $P^{-,-}$ are
non-empty. 
$\Wh(\X)\wrel\Y$ is obtained from
$\Wh(\X)\wrel\hull$ by un-deleting two vertices
$\gr_{\X}(\psi_x^\epsilon(\infty))$ and
$\gr_{\X}(\psi_y^{\epsilon'}(\infty))$.

If $\epsilon'=-\epsilon$ then the two components of
$\Wh(\X)\wrel\Y$ are:
\[\gr_{\X}(\psi_x^\epsilon(\infty))\cup\gr_{\X}(\line_x)\cup
\bigcup_{\wc\in P^{\epsilon,\epsilon}}\wc\]
and
\[\gr_{\X}(\psi_y^{-\epsilon}(\infty))\cup\gr_{\X}(\line_y)\cup
\bigcup_{\wc\in P^{-\epsilon,-\epsilon}}\wc\]

If $\epsilon'=\epsilon$ then the components of
$P^{-\epsilon,-\epsilon}$ remain separate components in
$\Wh(\X)\wrel\Y$, distinct from the component
containing the un-deleted vertices.

If $\Wh(\X)\wrel\hull$ has three deleted
vertices then partition the components of
$\Wh(\X)\wrel\hull$ into three parts according to their two loose ends.
To get $\Wh(\X)\wrel\Y$ we un-delete one
vertex, which combines two of the parts into a single 
component but leaves the other part alone.
\end{proof}

\section{Splittings}\label{sec:splitting}
Armed with the machinery of
\fullref{sec:topologyofdecompositionspace}, we are now prepared to
construct the relative JSJ decomposition.

Otal \cite{Ota92} makes the following observation:
\begin{lemma}\label{lemma:splitimpliescut}
If $\F$ splits over $\left<g\right>$ relative to $\multimot$ then
$\qmap(\{\inv{g}^\infty,g^\infty\})$ is a cut point or cut pair in $\decomp$.
\end{lemma}

Since a rigid decomposition space has no cut points or cut pairs:
\begin{corollary}\label{corollary:rigiddontsplit}
If $(\F,\multimot)$ is rigid there are no cyclic splittings of $\F$
relative to $\multimot$.
\end{corollary}

We will prove a converse in \fullref{theorem:splitting}.
This takes care of one case for which the \rjsj is trivial.
We saw another trivial case in \fullref{ex:first}, in which the
decomposition space is a circle. We explore this case in
\fullref{sec:crossed}.
In particular, a circle has no cut points or uncrossed cut pairs.

In the language of Guirardel and Levitt \cite{GuiLev09}, a 
subgroup of $\F$ is \emph{universally elliptic rel $\multimot$} if it is elliptic in
every cyclic splitting of $\F$ rel $\multimot$.
A graph of groups decomposition is a JSJ decomposition if all the
splittings are over universally elliptic subgroups and the
decomposition is maximal with respect to this property.
Our goal is to show that cut points and uncrossed cut
pairs correspond to universally elliptic relative cyclic splittings.
The first step is to show that the stabilizer of a cut point or
uncrossed cut pair is an infinite cyclic group over which $\F$ splits
rel $\multimot$.
For cut points this was already noted by Otal.
We show that uncrossed cut pairs have infinite cyclic stabilizers in 
\fullref{sec:uncrossed}, and in \fullref{sec:otalsplit} we construct a simplicial tree with a
cocompact $\F$--action whose edge stabilizers are the stabilizers of
cut points and uncrossed cut pairs of $\decomp$.

In \fullref{sec:refine} we prove the stabilizers of cut points and
uncrossed cut pairs of $\decomp$ are exactly the maximal cyclic universally
elliptic subgroups of $\F$ rel $\multimot$ over which $\F$ splits rel $\multimot$.
In \fullref{sec:proofofdecomp} we conclude
that the splitting we have constructed is the \rjsj.

\subsection{Crossing Pairs and the Circle}\label{sec:crossed}
In this subsection we give criteria for the
decomposition space to be a circle.

\begin{lemma}[{\cite[Theorem 2]{Ota92}, \cite[Theorem
    6.1]{CasMac11}}]\label{lemma:circles} 
  The following are equivalent:
\begin{enumerate}
\item $(\F,\multimot)$ is a QH--surface.
\item Some Whitehead graph for $\multimot$ is a circle.
\item Every Whitehead graph for $\multimot$ with no cut vertex is a circle.
\item $\decomp$ is a circle. 
\item Every minimal cut set of $\decomp$ is a cut pair.
\end{enumerate}
\end{lemma}

Proofs of the following two lemmas are elementary and are left to the reader.
 \begin{lemma}[Cut Pair Exchange]\label{lemma:3of4}
   Suppose $\{x_0,x_1\}$ and $\{y_0,y_1\}$ are crossing cut pairs in
   $\decomp$. 
Then $\{x_0,y_0\}$ is a cut pair.
 \end{lemma}

 \begin{lemma}\label{lemma:transitivecutpairs}
   Suppose $\{x,y\}$ and $\{y,z\}$ are cut pairs of $\decomp$ and for
   every neighborhood $N'$ of $y$ there exists a connected
   neighborhood $N\subset N'$ of $y$ such that $N\setminus\{y\}$ has
   exactly two components. Then $\{x,z\}$ is a cut pair.
 \end{lemma}

The following proposition is a generalization of a construction of Bowditch for boundaries
of hyperbolic groups \cite{Bow98b}.

 \begin{proposition}\label{proposition:circle}
$\decomp$ is a circle if and only if all of the
 following are satisfied:
 \begin{enumerate}
 \item $\decomp$ is connected.
\item $\decomp$ has no cut points.
\item $\decomp$ has cut pairs.
\item  Every cut pair in $\decomp$ is crossed by a cut pair.
\end{enumerate}
\end{proposition}
\begin{proof}
A circle satisfies these conditions.
We prove the converse.

Define an equivalence relation on $\decomp$ by $x\sim y$ if $x=y$
  or if $\{x,y\}$ is a cut pair.
Transitivity follows from \fullref{claim:2componentneighborhood} and \fullref{lemma:transitivecutpairs}.
\begin{claim}\label{claim:2componentneighborhood}
For every cut pair
$\{x_0,x_1\}$ and every neighborhood $N'$ of $x_0$ there exists a
connected neighborhood $N\subset N'$ of $x_0$ such that $N\setminus \{x_0\}$
has precisely two components.
\end{claim}
\begin{claimproof}
  Every cut pair is crossed, so by \fullref{lemma:precisely2components}
every cut pair has precisely two complementary components.
By
\fullref{corollary:no3componentcutpairsimpliesno3componentrays},
for every simplicial geodesic $\phi\from [0,\infty]\to\closure{\tree}$
there are at most two components of
$\Wh(\phi([1,\infty]))\wrel\phi([0,\infty])$.
From this and the fact that there are no cut points, 
\fullref{lemma:niceneighborhood} gives the desired neighborhood.
\end{claimproof}

\begin{claim}\label{claim:equivalenceclassesclosed}
 Equivalence classes are closed. 
\end{claim}
\begin{claimproof}
If $[x]$ is a single point we are done. Otherwise, suppose
$(y_i)\to y$ for $x\neq y_i\in [x]$. For some
$\eta\in\qmap^{-1}(y)$ there exists a subsequence of $(y_i)$ and a
choice of $\eta_i\in \qmap^{-1}(y_i)$ so that $(\eta_i)\to\eta$ in
$\bdry\tree$.
Choose $\xi\in\qmap^{-1}(x)$ and
let $\phi\from [-\infty,\infty]\onto[\xi,\eta]$ be a simplicial
geodesic.
Passing to a further subsequence of $(\eta_i)$, there are
positive integers $r_i$ such that $r_{i+1}>r_i+1$ and
$[\xi,\eta_i]\cap[\xi,\eta]=\phi([-\infty,r_i+1])$.
Since $x\neq y_i\in[x]$, each $\{x,y_i\}$ is a cut pair, so there are at least two
components of $\Wh(\phi([-r_i,r_i]))\wrel\phi([-r_i-1,r_i+1])$, by
\fullref{lemma:partialcutpairhull}.
By \fullref{lemma:lookslikeacutpair},
$\qmap(\{\phi(-\infty),\phi(\infty)\})=\{x,y\}$ is a cut pair, so
$y\in [x]$.
\end{claimproof}

\begin{claim}
All of $\decomp$ is in one equivalence class.
\end{claim}
Given the claim, every point of $\decomp$ is a member of a cut
pair, and it follows from
\fullref{lemma:circles} that $\decomp$ is a circle.

\begin{claimproof}
We have assumed that a cut pair exists, so there is an equivalence
class $[x]$ consisting of more than one point.
Suppose that $[x]$ is not all of $\decomp$. 

Let $U$ be a component of $\decomp\setminus[x]$.
Since $\decomp$ is locally connected by \fullref{corollary:connected},
and since $[x]$ is closed by \fullref{claim:equivalenceclassesclosed}, $U$ is open
in $\decomp$. 
Since $\decomp$ is connected without cut points, $U$ has at least two
limit points in $[x]$.
Pick distinct points $y_0$ and $y_1$ in $\closure{U}\cap [x]$.
Since they are in $[x]$, these points are  a cut pair, and $\decomp\setminus\{y_0,y_1\}$ has exactly two
components, $A_0$ and $A_1$.
Assume $U\subset A_0$.

Let $\{z_0,z_1\}$ be a cut pair crossing $\{y_0,y_1\}$ with
complementary components $B_0$ and $B_1$. 
Assume $z_0\in A_0$, $z_1\in A_1$, $y_0\in B_0$ and $y_1\in B_1$.

By \fullref{lemma:3of4}, $z_0$ and $z_1$ are in $[x]\subset
\decomp\setminus U$.
Thus, $U$ is contained in $B_\e$,
where $\e$ is either 0 or 1.
Since $U\subset A_0$, we have $U\subset A_0\cap B_\e$.

Now, $\{y_\e,z_0\}$ is a cut pair whose 
components are $C_0=A_0\cap B_\e$ and $C_1=A_1\cup
B_{1-\e}\cup\{y_{1-\e}\}\cup\{z_1\}$.
However, $U$, and hence $C_0$, has $y_{1-\e}\in C_1$ as a limit point,
which is a contradiction. Thus, $[x]=\decomp$.
\end{claimproof}
\end{proof}

\begin{corollary}\label{corollary:uncrossedexist}
  If $\decomp$ is not rigid and not a circle then
  $\decomp$ contains cut points or uncrossed cut pairs.
\end{corollary}

\subsection{Uncrossed Cut Pairs}\label{sec:uncrossed}
In this section we show that uncrossed cut pairs have infinite cyclic
stabilizers. 
A priori, the preimage in $\bdry\F$ of a pair of points in $\decomp$
could be as many as four points. This first step is to rule out that possibility.
\begin{lemma}\label{lemma:uncrossedaxis}
  The preimage of an uncrossed cut pair is two points.
\end{lemma}
\begin{proof}
Let $\{x_0,x_1\}\subset\decomp$ be an uncrossed cut pair, and
suppose $|\qmap^{-1}(x_0)|=2$. Then there is an $h\in\F$ and a
$\mot\in\multimot$ such that
$\qmap^{-1}(x_0)=\{h\inv{w}^\infty,hw^\infty\}$.
Replacing $\{x_0,x_1\}$ by $\{\inv{h}x_0,\inv{h}x_1\}$, we may assume
$h$ is trivial.
Let $\hull$ be the convex hull of $\qmap^{-1}(\{x_0,x_1\})$.
Let $\phi\from[-\infty,\infty]\onto[\inv{w}^\infty,w^\infty]$ be the simplicial
geodesic with $\phi(0)=\bp$.
Let $p$ be large enough so that
$\phi([-\infty,-p-1])\cup\phi([p+1,\infty])$ contains no branch point of $\hull$.

By \fullref{lemma:twocomponentsplit}, for each $\epsilon\in\pm$ there
is a unique component $\wc_\epsilon\subset\Wh(\hull)$ such that
$\cup_{\we\in\wc_\epsilon}\line_\we$ meets
$\phi([\epsilon\cdot(p+1),\epsilon\cdot\infty])$, and these are the
only two components. 

We reach a contradiction by exhibiting a cut pair crossing
$\{x_0,x_1\}$.
For each $\epsilon\in\pm$, we have
$\mot^{\epsilon\cdot(2p+1)}x_1\subset\qmap(\bdry\tree_{\wc_\epsilon})$, so
$\{\mot^{2p+1}x_1,\mot^{-2p-1}x_1\}$ crosses $\{x_0,x_1\}$.

$\{\mot^{2p+1}x_1,\mot^{-2p-1}x_1\}$ is a cut pair by
\fullref{lemma:transitivecutpairs}:
$\mot^{2p+1}\{x_1,x_0\}$=$\{\mot^{2p+1}x_1,x_0\}$ and
$\mot^{-2p-1}\{x_0,x_1\}$=$\{x_0,\mot^{-2p-1}x_1\}$ are both cut pairs,
and,
by \fullref{lemma:niceneighborhood}, for every neighborhood $N'$ of
$x_0$ there exists a connected neighborhood $N\subset N'$ such that
$N\setminus\{x_0\}$ has exactly two components.
\end{proof}

\begin{lemma}\label{lemma:crossedcutpair}
  If 
  $\qmap(\{\xi,g^\infty\})$ is a cut pair, for some $g\in\F$ and
  $\xi\in\bdry\tree$ with $\xi\neq\inv{g}^\infty$\!, then it is a crossed cut pair.
\end{lemma}
\begin{proof}
We assume, without loss of generality, that $g$ is cyclically reduced.
  Let $\hull$ be the convex hull of $\{\inv{g}^\infty,g^\infty,\xi\}$.

There is a $g$--action on $\Wh([\inv{g}^\infty,g^\infty])$ given by
$\we\mapsto \gr_{[\bar{g}^\infty,g^\infty]}(g\line_\we)$ for an edge $\we$. 
This action
permutes the components of $\Wh([\inv{g}^\infty,g^\infty])$.
Replacing $g$ by a suitable power, we may assume the components are
fixed.

Let $\psi\from [0,\infty]$ be the simplicial geodesic ray with
$\psi(\infty)=\xi$ and
$\psi([0,\infty])\cap[\inv{g}^\infty,g^\infty]=\psi(0)$.
Let $\phi\from [-\infty,\infty]\onto [\inv{g}^\infty,g^\infty]$ be the
simplicial geodesic with $\phi(0)=\psi(0)$.

Let $\line\in\lines|_{\phi((0,1))}$ be a line such that
$\gr_\hull(\line)$ belongs to a component of $\Wh(\hull)$ with an end
at $\psi(\infty)$.
Then there exists an edge path $\mathfrak{P}:\gr_{[\bar{g}^\infty,g^\infty]}(\line)=\we_0,\dots,\we_k$ in
$\Wh([\inv{g}^\infty,g^\infty])$ with $\we_k$ incident to
$\gr_{[\bar{g}^\infty,g^\infty]}(\psi(\infty))$.
For all sufficiently large $m$, we have
$\phi([-\infty,\infty])\cap\bigcup_{\we\in
  g^m\mathfrak{P}}\line_\we\subset \phi([1,\infty])$.
Since the $g$--action preserves components, $\we_0$ and $g^m\we_0$ are in
the same component.
By \fullref{lemma:nonbacktracking} and \fullref{lemma:connectedray},
there is an edge path in
$\Wh(\phi([1,\infty]))\wrel\phi([0,\infty])$ that connects
$\gr_{\phi([1,\infty])}(\line_{\we_0})$ to
$\gr_{\phi([1,\infty])}(\line_{g^m\we_0})$.
By concatenating $g^m\mathfrak{P}$, we see that
$\gr_{\phi([1,\infty])}(\line_{\we_0})$ is in the
$\gr_{\phi([1,\infty])}(g^m\psi(\infty))$ component of
$\Wh(\phi([1,\infty]))\wrel\phi([0,\infty])$.
Since this is true for all sufficiently large $m$, every line in
$\lines|_{\phi((0,1))}$ that contributes an edge to a component of $\Wh(\hull)$ with an end
at $\psi(\infty)$ contributes an edge to the same component of
$\Wh(\phi([1,\infty]))\wrel\phi([1,\infty])$, so there is only one
component, $\wc_+$, of $\Wh(\hull)$ with ends at $\phi(\infty)$ and $\psi(\infty)$.
The same argument in the $\phi(-\infty)$ direction shows there is only
one component, $\wc_-$, of $\Wh(\hull)$ with ends at $\phi(-\infty)$ and
$\psi(\infty)$.
Since $\qmap(\{\psi(\infty),\phi(\infty)\})$ is a cut pair, $\wc_+\neq\wc_-$.

It follows from \fullref{lemma:leavesofcomponents} that for
$\epsilon\in\pm$, 
$\hull\cap\bigcup_{\we\in\wc_\epsilon}\line_\we=[\psi(\infty),\phi(\epsilon\cdot\infty)]$, so
$\gr_\hull(g^{\epsilon\cdot 1}\psi(\infty))\in\wc_\epsilon$.
The map $\we\mapsto \gr_{[\psi(\infty),\phi(\infty)]}(\line_\we)$
sends $\wc_+$ onto a component of $\Wh([\psi(\infty),\phi(\infty)])$
not containing the vertex
$\gr_{[\psi(\infty),\phi(\infty)]}(\phi(-\infty))$, and sends $\wc_-$
into the component of $\Wh([\psi(\infty),\phi(\infty)])$
containing
$\gr_{[\psi(\infty),\phi(\infty)]}(\phi(-\infty))$, so
$\gr_{[\psi(\infty),\phi(\infty)]}(\inv{g}\psi(\infty))$ and
$\gr_{[\psi(\infty),\phi(\infty)]}({g}\psi(\infty))$ are in
different components of $\Wh([\psi(\infty),\phi(\infty)])$.
Moreover, for $\line\in\lines|_{[\bar{g}\phi(0),g\phi(0)]}$ we have
that $\gr_{g\hull}(\line)\in g\wc_-$ if and only if $\gr_{\bar{g}\hull}(\line)\in \inv{g}\wc_+$.
It follows that $\Wh([\inv{g}\psi(\infty),g\psi(\infty)])$ has a
component not containing
$\gr_{[\bar{g}\psi(\infty),g\psi(\infty)]}(\phi(-\infty))$ and $\gr_{[\bar{g}\psi(\infty),g\psi(\infty)]}(\phi(\infty))$.
Therefore, $\qmap(\{\inv{g}\psi(\infty),g\psi(\infty)\})$ is a cut
pair, and it crosses $\qmap(\{\psi(\infty),\phi(\infty)\})$.
\end{proof}

\begin{proposition}\label{proposition:periodic}
Uncrossed cut pairs are rational:
For every uncrossed cut pair $S$
there exists an non-trivial element $f\in\F$ such that
$\left<f\right>$ is the stabilizer of $S$.
\end{proposition}
\begin{proof}
Let $S=\{x_0,x_1\}$.
By \fullref{lemma:uncrossedaxis}, $|\qmap^{-1}(x_i)|=1$ for both
$i\in\{0,1\}$.
Let $\phi\from[-\infty,\infty]\onto [\qmap^{-1}(x_0),\qmap^{-1}(x_1)]$
be a simplicial geodesic. 

For increasing $r$, the sequence
$\#\Wh(\phi([r+1,\infty]))\wrel\phi([r,\infty])$ is a
non-decreasing sequence of integers, bounded above by the maximum
valence of $\Wh(\bp)$.
Assume $r$ is large enough so that the sequence has achieved its
maximum. 

Let $g$, $h$, and $a$ be the elements provided by
\fullref{lemma:periodic} for $\phi\from [r,\infty]\to\closure{\tree}$.
There is some $r'>r$ such that $\phi(r')=h$ and some $r''>r'+1$ such that $\phi(r'')=gha$.
Because of our choice of $r$, \fullref{lemma:periodic}
(\ref{item:preserves}) implies that for
$\line\in\lines|_{\phi([r',r'+1])}$, the edges
$\gr_{\phi([r'+1,\infty])}(\line)$ and
$\gr_{\phi([r'+1,\infty])}(g\line)$ are in the same component of $\Wh(\phi([r'+1,\infty]))\wrel\phi([r',\infty])$.

If $\phi(\infty)\neq g^\infty$ and $\phi(-\infty)\neq \inv{g}^\infty$
then, since $\qmap(\{\inv{g}^\infty,g^\infty\})$ is a cut point or cut
pair by \fullref{corollary:approximatinggiscut} and $S$ is uncrossed, $\gr_{\phi([-\infty,\infty])}(\inv{g}^\infty)$
and $\gr_{\phi([-\infty,\infty])}({g}^\infty)$ are in the same
component of $\Wh(\phi([-\infty,\infty]))$. 
Let $\line\in\lines|_{\phi((r',r'+1))}$ be a line such that
$\gr_{\phi([-\infty,\infty])}(\line)$ is in a component of
$\Wh(\phi([-\infty,\infty]))$ not containing
$\gr_{\phi([-\infty,\infty])}(\inv{g}^\infty)$ and
$\gr_{\phi([-\infty,\infty])}({g}^\infty)$.
Then $\gr_{\phi([r'+1,\infty])}(\line)$ and
$\gr_{\phi([r'+1,\infty])}(g\line)$ are in the same component of
$\Wh(\phi([r'+1,\infty]))\wrel\phi([r',\infty])$, which, in light of the
inclusion via splicing, implies
that $\gr_{\phi([-\infty,\infty])}(\line)$ and
$\gr_{\phi([-\infty,\infty])}(g\line)$ are in the same component of
$\Wh(\phi([-\infty,\infty]))$.
This implies that $\gr_{\phi([-\infty,\infty])}(g\phi(-\infty))$ is in
a component of $\Wh(\phi([-\infty,\infty]))$ not containing the vertex
$\gr_{\phi([-\infty,\infty])}({g}^\infty)$.
However, $\gr_{\phi([-\infty,\infty])}(g\phi(\infty))$ is in
the component of $\Wh(\phi([-\infty,\infty]))$ containing 
$\gr_{\phi([-\infty,\infty])}({g}^\infty)$, and $\qmap(\{g\phi(-\infty),g\phi(\infty)\})$ is a cut pair.
This contradicts the hypothesis that $S$ is uncrossed.

In the two cases that $[\inv{g}^\infty,g^\infty]$ and
$\phi([-\infty,\infty])$ share exactly one endpoint,
\fullref{lemma:crossedcutpair} gives a contradiction to the hypothesis
that $S$ is uncrossed.
Therefore $\phi(\infty)=g^\infty$ and
$\phi(-\infty)=\inv{g}^\infty$, and we take $f$ to be an indivisible root
of $g$.
\end{proof}

\begin{proposition}\label{proposition:boundedperiodicity}
$\decomp$ has finitely many orbits of uncrossed cut pairs.
\end{proposition}
\begin{proof}
The element $g$ in \fullref{proposition:periodic} is conjugate to a word of bounded length provided by
\fullref{lemma:periodic}. There are finitely many such conjugacy classes.
\end{proof}

\begin{lemma}\label{mainlemma}
If $\decomp$ is not rigid or a circle then there is an indivisible element $g\in \F$ such that
 $\qmap(\{\inv{g}^\infty,g^{\infty}\})$ is a cut set that is not crossed by
 any cut pair.
\end{lemma}
\begin{proof}
By
\fullref{corollary:uncrossedexist} there exist cut points or uncrossed
cut pairs.
For cut points apply \fullref{lemma:cutpointsareinmultiword};
for uncrossed cut pairs apply \fullref{lemma:uncrossedaxis}.
\end{proof}

\subsection{The Splitting Criterion}\label{sec:otalsplit}
\fullref{proposition:multipleotal} is a generalization of a construction of Otal
\cite{Ota92}, who proves it in the case that $\{S_i\}_{i\in I}$ is a
single orbit of cut points.
The main change is \fullref{claim:minimalelements}, which replaces Otal's Lemma 3.3.
\begin{proposition}\label{proposition:multipleotal}
Consider a non-empty, $\F$--invariant collection of disjoint cut sets $\{S_i\}_{i\in
  I}$ in $\decomp$ satisfying the following
conditions:
\begin{enumerate}
\item For each $i\in I$ there is an indivisible $g_i\in \F$ such that
  $S_i=\qmap(\{\inv{g_i}^\infty,g_i^{\infty}\})$.
\item The cut sets are pairwise non-crossing.
\item The set $\{S_i\}_{i\in I}$ is a union of finitely many $\F$--orbits.
\end{enumerate}
Then $\F$ splits as a graph of groups rel $\multimot$ with cyclic edge
stabilizers.
The vertex set is bipartite, with Type  1 vertices stabilized by
maximal cyclic subgroups generated by the $g_i$ and Type 2 vertices
stabilized by non-cyclic subgroups.
\end{proposition}
\begin{proof}
For each $i\in I$ there is a partial ordering $<_i$ on $\{S_j\}_{j\in
  I\setminus\{i\}}$ defined by $S_j<_i S_k$ if $S_j$ separates $S_i$ from $S_k$,
that is, if $S_i$ and $S_k$ are in different complementary components of $S_j$.
Since the cut sets are pairwise non-crossing, the complementary
component of $S_j$ containing $S_k$ is well defined.

\begin{claim}\label{claim:minimalelements}
  There exist $<_i$--minimal elements.
\end{claim}
\begin{claimproof}
Suppose $S_i$, $S_j$, and $S_k$ are elements of $\{S_i\}_{i\in I}$.
Let $\X$ be $[\inv{g}_i^\infty,g_i^\infty]\cap
[\inv{g}_k^\infty,g_k^\infty]$, if this intersection is non-empty.
Otherwise, 
let $\X$ be the geodesic segment in $\tree$ connecting
$[\inv{g}_i^\infty,g_i^\infty]$ to $[\inv{g}_k^\infty,g_k^\infty]$.
Suppose $[\inv{g}_j^\infty,g_j^\infty]$ does not intersect $\X$.
Then there exists an edge $e\in\tree$ incident to, but not contained
in, $[\inv{g}_j^\infty,g_j^\infty]$, such that $e$ separates
$[\inv{g}_j^\infty,g_j^\infty]$ from at least three of the points
$\inv{g}_i^\infty$, $g_i^\infty$, $\inv{g}_k^\infty$, and
$g_k^\infty$.
It then follows from \fullref{lemma:basicconnectivity} that $S_i$ and
$S_k$ contain points in a common component of $\decomp\setminus S_j$, so $S_j$ does not separate
$S_i$ from $S_k$ in $\decomp$.

We conclude that for fixed $S_i$ and $S_k$
the only $S_j$ such that $S_j<_i S_k$ belong to the finite set of
those for which $[\inv{g}_j^\infty,g_j^\infty]$ intersects $\X$.
\end{claimproof}

Define a graph on which $\F$ acts without
inversions as follows.
The graph has two types of vertices. 
There is a Type 1 vertex $v_i$ for each $S_i$.
Given a Type 1 vertex $v_i$, $S_i$ has finitely many complementary
components $C_{i,1},\dots,C_{i,m_i}$.
For each $i,j$, consider the subset $\{v_i\}\cup\{v_k\mid S_k\text{ is
}<_i\text{--minimal, and }S_k\in C_{i,j}\}$.
Define this subset to be a Type 2 vertex.
Define adjacency by inclusion. 

Since the $S_i$ are cut sets, this graph is a tree.
The quotient of this tree by the $\F$--action contains one vertex
of Type 1 for each orbit of cut set, and some
finite number of adjacent Type 2 vertices.
The stabilizers of the Type 1 vertices are the groups
$\left<g_i\right>$, so we have a cyclic splitting of $\F$.

The generators of the line pattern must be conjugate into the vertex
groups, otherwise we would have a line in the pattern crossing from
one component of some $\Wh([\inv{g_j}^\infty,g_j^\infty])$ to another, which is
absurd.
\end{proof}

Combining
\fullref{proposition:multipleotal} and \fullref{mainlemma} gives us a
splitting theorem:

\begin{theorem}[Splitting Theorem]\label{theorem:splitting}
  If $\decomp$ is not rigid then either $(\F,\multimot)$
is a three-holed sphere or $\F$ splits over $\mathbb{Z}$ relative to $\multimot$.
\end{theorem}
\begin{proof}
If $\decomp$ is neither rigid nor a circle then  \fullref{mainlemma} provides a $g$
so that the translates of $\qmap(\{\inv{g}^\infty,g^\infty\})$
satisfy the hypotheses of \fullref{proposition:multipleotal}, which gives
a relative splitting.

If $\decomp$ is a circle then, by \fullref{lemma:circles}, $(\F,\multimot)$ is a surface with boundary.
Either this is a three-holed sphere, or
there exists an essential, non-peripheral simple closed curve in the surface, which gives a relative splitting.
\end{proof}

\subsection{Refining Splittings}\label{sec:refine}
\fullref{theorem:splitting} tells us when we can split $\F$ rel
$\multimot$. 
In this subsection we determine when a splitting can be refined.

\begin{definition}
  Let $\Gamma$ be a graph of groups decompositions of $\F$ rel
  $\multimot$ with cyclic edge stabilizers.
Define the \emph{augmented multiword} $\augw$ to be a
multiword in $\F$ obtained by choosing generators of representatives
of the distinct conjugacy classes of maximal cyclic subgroups of $\F$ containing the elements of
$\multimot$ and the generators of each of the edge groups of $\Gamma$.
The choices can be, and are, made so that $\multimot\subset\augw$.
\end{definition}

\begin{definition}
  Define the \emph{augmentation map}, $\augmap$, to be the quotient
  map $\augmap\from \decompw \to \decomp_{\augw}$. Note that $\qmap_{\augw}=\augmap \circ \qmap_{\multimot}$.
\end{definition}

\begin{lemma}\label{lemma:embed}
 Let $G$ be a non-cyclic
  vertex group of $\Gamma$.
The decomposition space $\decomp_{\indg}$ of $G$ corresponding
to $\indg$ embeds naturally into the decomposition space
$\decomp_{\augw}$ of $\F$ corresponding to $\augw$. 
\end{lemma}
\begin{proof}
The inclusion $\iota\from G\hookrightarrow \F$ extends to an embedding $\bdry \iota \from \bdry G \hookrightarrow \bdry \F$.
The equivalence relation on $\bdry G$ coming from $\indg$ is the
restriction to $\bdry\iota(\bdry G)$ of the equivalence relation on
$\bdry \F$ coming from $\augw$.
\end{proof}

\begin{lemma}\label{lemma:connectedpieces}
Let $\Gamma$ be a graph of groups decomposition of $\F$ rel $\multimot$
with cyclic edge groups. 
Let $G$ be a non-cyclic vertex of $\Gamma$.
The decomposition space $\decomp_{\indg}$ of $G$ with respect to $\indg$ is connected.
\end{lemma}
\begin{proof}
By \fullref{corollary:connected}, if $\decomp_{\indg}$ is not connected there is a free splitting of
  $G$ rel $\indg$.
This gives a free splitting of $\F$ rel $\augw$, which implies that
$\decomp_{\augw}$ is not connected. 
This is not possible since $\decomp_{\augw}$ is a quotient of
$\decompw$, which, by our standing assumption, is connected.
\end{proof}

\begin{lemma}\label{lemma:rigidboundary}
Let $\Gamma$ be a graph of groups decomposition of $\F$ rel $\multimot$
with cyclic edge groups. 
Let $G$ be a non-cyclic vertex of $\Gamma$ such that $\decomp_{\indg}$
is rigid.
Let $\left<g\right>$ be the stabilizer of an edge incident to $G$.
Then $S=\qmap_{\multimot}(\{\inv{g}^\infty,g^\infty\})$ is either a
cut point or uncrossed cut pair of $\decompw$.
\end{lemma}
\begin{proof}
Suppose not. $S$ is a cut set, so by
\fullref{lemma:precisely2components} it is a cut pair with exactly two
complementary components, $B_0$ and
$B_1$. Assume $\qmap(\bdry G)\subset B_0$.

Since $S$ is crossed then there is a cut pair $R$ of $\decompw$
crossing $S$.
Let $r$ be the point of $R$ in $B_0$, and let $A_0$ and $A_1$ be the
complementary components of $R$.

Let $\{S_i\}_{i\in I}$ be the collection of cut points and cut pairs
corresponding to the edges of $\Gamma$.
$S$ is one of these, so set $S=S_{i_0}$.
Set: 
\[J=\{j\neq i_0\mid S_j \text{ crosses } R\text{ and }S_j\subset B_0\}\]
If $J=\emptyset$ then $\augmap(r)\in\decomp_{\indg}$, and:
\[\augmap(\closure{A}_0)\cap\augmap(\closure{A}_1)\cap\decomp_{\indg}=\augmap(\{S_{i_0},r\})\]
If $J\neq\emptyset$ then there is a $<_{i_0}$--minimal element $S_{i_1}$ of $\{S_j\}_{j\in
  J}$, and:
\[\augmap(\closure{A}_0)\cap\augmap(\closure{A}_1)\cap\decomp_{\indg}=\augmap(\{S_{i_0},S_{i_1}\})\]

In either case we get a cut pair of $\decomp_{\indg}$, contradicting rigidity.
\end{proof}

\begin{definition}
An \emph{uncrossed collection} in $\decomp_{\multimot}$ is a non-empty
union of orbits of cut points and uncrossed cut pairs.
\end{definition}

There are finitely many such orbits by
\fullref{lemma:cutpointsareinmultiword} and
\fullref{proposition:boundedperiodicity}, so, given an uncrossed
collection $\{S_i\}_{i\in I}$, \fullref{proposition:multipleotal} provides a
corresponding graph of groups decomposition $\Gamma$.
For the remainder of this section, fix an uncrossed collection
$\{S_i\}_{i\in I}$ and corresponding graph of groups $\Gamma$ with
Bass-Serre tree $\mathrm{BS}(\Gamma)$.

\begin{lemma}[Universality of Uncrossed
  Splittings]\label{lemma:uncrossedsplittingsareuniversal}
The stabilizer of a cut point or cut pair of $\decompw$ is elliptic in $\Gamma$.
\end{lemma}
\begin{proof}
Let $S$ be a cut point or cut pair.
By considering the $\F$--action on the convex hull of $\qmap^{-1}(S)$
it is clear that the stabilizer of $S$ is either trivial or a maximal
cyclic subgroup.
 If its stabilizer is trivial we
are done, so assume its stabilizer is $\left<g\right>$.

If $g$ is not elliptic then it has an axis in $\mathrm{BS}(\Gamma)$.
A Type 1 vertex on this axis corresponds to a cut set $S_i$ separating the two points of $S$. 
By \fullref{lemma:precisely2components},  this would mean $S$ crosses $S_i$, contradicting the hypothesis
that the $S_i$ are uncrossed.
\end{proof}

\begin{lemma}\label{lemma:cutsarecutsinaug}
If $S$ is a cut point or cut pair of
  $\decompw$ then $\augmap(S)$ is a cut point or cut pair of
  $\decomp_{\augw}$.
If $R$ is a cut point or cut pair of
  $\decomp_{\augw}$ then $\augmap^{-1}(R)$ is a cut point or cut pair of $\decompw$.
\end{lemma}
\begin{proof}
$\augmap$ identifies points of uncrossed cut pairs, so $\augmap(S)$
fails to be a cut set only if some uncrossed cut pair crosses $S$.
That is impossible, since by \fullref{lemma:precisely2components}
crossing is a symmetric relation.

Conversely, let $R$ be a cut point or cut pair of
  $\decomp_{\augw}$. 
It is clear that $\augmap^{-1}(R)$
is a cut set in $\decompw$; we just need to show that it consists of at most two points. 

Suppose a point $r\in R$ has preimage $\augmap^{-1}(r)$ consisting of
two points. 
Then $r$ is stabilized by a conjugate of an element $h\in \augw\setminus\multimot$,
so $\augmap^{-1}(r)=S_i$ for some $i$. 
By the first part of the lemma, $\augmap(S_i)=\{r\}$ is a cut point,
so $R=\{r\}$, and $\augmap^{-1}(R)=S_i$ is a cut pair.
\end{proof}

\begin{lemma}\label{lemma:prerefinement}
Let $G$ be a non-cyclic
  vertex group of $\Gamma$.
Let $g$ be an element of $G$ such that 
$\qmap_{\indg}(\{\inv{g}^\infty,g^\infty\})$ is a cut point
or uncrossed cut pair in $\decomp_{\indg}$. 
Then $\qmap_{\multimot}(\{\inv{g}^\infty,g^\infty\})$ is a cut
point or uncrossed cut pair, respectively, of $\decompw$ that is not
in the uncrossed collection.
\end{lemma}
\begin{proof}
There is a Type 2 vertex $\{v_j\}_{j\in J}$ of $\mathrm{BS}(\Gamma)$ corresponding to
$G$. 
For each $j\in J$, the corresponding cut set $S_j$ becomes a point in $\decomp_{\indg}$, and each complementary component of the image of
$\qmap_{\indg}(\{\inv{g}^\infty,g^\infty\})$ in $\decomp_{\indg}$ 
contains some of these $S_j$ points, so $\augmap^{-1}(
\qmap_{\indg}(\{\inv{g}^\infty,g^\infty\}))= \qmap_{\multimot}(\{\inv{g}^\infty,g^\infty\})$ is a cut point or cut
pair in $\decompw$ separating some of these $S_j$.
The Type 2 vertices  are subsets $\{v_j\}_{j\in J}$ such
that for $j_0,j_1\in J$ there is no $i\in I$ such that $S_i$ separates
$S_{j_0}$ from $S_{j_1}$.
Thus, 
$\qmap_{\multimot}(\{\inv{g}^\infty,g^\infty\})$ is a cut set that is
not in $\{S_i\}_{i\in I}$.

Since $\qmap_{\multimot}(\{\inv{g}^\infty,g^\infty\})$ is not one of the $S_i$, its
cardinality is the same as that of the image in $\decomp_{\indg}$.
Thus, they are either both cut points or both cut pairs. 

It remains only to show that if $\qmap_{\indg}(\{\inv{g}^\infty,g^\infty\})$ is
an uncrossed cut pair then $\qmap_{\multimot}(\{\inv{g}^\infty,g^\infty\})$ is uncrossed.
Suppose $\qmap_{\multimot}(\{\inv{g}^\infty,g^\infty\})$ is crossed
by some other cut pair $R$ of $\decomp_{\multimot}$. 
$R$ cannot cross any of the $S_i$, so $R\subset\qmap(\bdry G)$, which
implies $\augmap(R)$ and $\augmap(\qmap_{\multimot}(\{\inv{g}^\infty,g^\infty\}))$ are
crossing cut pairs in the embedded copy of $\decomp_{\indg}$ in $\decomp_{\augw}$.
Thus we get crossing cut pairs of $\decomp_{\indg}$, contrary to
hypothesis.
\end{proof}

\begin{lemma}[Refinement Lemma]\label{lemma:refinement}
  Let $G$ be a non-cyclic vertex of $\Gamma$.
If $\decomp_{\indg}$ is neither a circle nor rigid then there is a
refinement $\Gamma'$ of $\Gamma$ obtained by splitting $G$ rel $\indg$.
This is a splitting over an uncrossed collection
containing $\{S_i\}_{i\in I}$.
\end{lemma}
\begin{proof}
By \fullref{lemma:connectedpieces}, $\decomp_{\indg}$ is connected. If
it is not a circle and not rigid, then by
  \fullref{mainlemma} there is an element $g\in G$ such that
  $\qmap_{\indg}(\{\inv{g}^\infty, g^\infty\})$ is a cut point or
  uncrossed cut pair that is not in the uncrossed collection.
By \fullref{lemma:prerefinement}, $\qmap_{\multimot}(\{\inv{g}^\infty, g^\infty\})$ is a cut point or
  uncrossed cut pair of $\decompw$.
Thus, we can add the orbit of $\qmap_{\multimot}(\{\inv{g}^\infty,
g^\infty\})$ to the set $\{S_i\}_{i\in I}$ to get a larger uncrossed
collection, and hence a graph of groups decomposition $\Gamma'$
refining $\Gamma$.
\end{proof}

\subsection{The Decomposition Theorem}\label{sec:proofofdecomp}

\begin{theorem}[\decompthm]\label{theorem:JSJ}\hypertarget{decompthmtarget}
  There exists a
  canonical relative JSJ--decomposition (rJSJ), a graph of groups
  decomposition $\Gamma$ of $\F$ relative to $\multimot$ with cyclic
  edge groups,
  satisfying the following conditions:
  \begin{enumerate}[(1)]
\item If there is more than one vertex, the graph is bipartite. Cyclic vertex groups are adjacent only
  to non-cyclic vertex groups, and vice-versa. Furthermore, if $G$ is
  a non-cyclic vertex group the incident edge groups map onto
  $G$--maximal cyclic subgroups of $G$ in distinct $G$--conjugacy
  classes. Finally, the sum of the degrees of the edge inclusions at
  any cyclic vertex group is at least 2.\label{item:bipartite}
\item $\Gamma$ is universal: if $\F$ splits over a cyclic subgroup relative to $\multimot$ then the
cyclic subgroup is conjugate into one of the vertex groups.\label{item:universal}
\item $\Gamma$ is maximal: it cannot be refined and still satisfy
  these conditions.\label{item:maximal}
  \end{enumerate}

Moreover, the rJSJ is characterized by splitting $\F$ over the stabilizers of cut points
  and uncrossed cut pairs in $\decompw$.
There are three mutually exclusive possibilities:
\begin{enumerate}[(a)]
\item $(\F,\multimot)$ is rigid. $\decompw$ has no cut points or cut
  pairs. The rJSJ is
  trivial.
\item $(\F,\multimot)$ is a QH--surface. $\decompw$ is a circle. The rJSJ is
  trivial.
\item The rJSJ is nontrivial. For every non-cyclic vertex group $G$ we
  have that $(G, \indg)$ is either
  rigid or a QH--surface. Stabilizers of cut points and uncrossed cut
  pairs are conjugate to the cyclic vertex groups.\label{item:basicpieces}
\end{enumerate}
 Consequently,  if $\F$ splits over $\langle g\rangle$ relative to
 $\multimot$ then $\langle g\rangle$ is conjugate into the stabilizer of one of the cyclic vertices or
one of the QH--surface vertices of the rJSJ.
\end{theorem}
\begin{remark}
  Conditions (\ref{item:universal}) and (\ref{item:maximal}) are
  standard requirements for a JSJ decomposition.
In general there is not a canonical JSJ decomposition satisfying these
condition, but a whole deformation space of JSJ decompositions
\cite{GuiLev09}.
A particular JSJ decomposition can be chosen from this deformation
space by applying the normalizations from \fullref{sec:normalization}.
Condition (\ref{item:bipartite}) says these normalizations have been performed.
\end{remark}
\begin{proof}
If $(\F,\multimot)$ is rigid or is a three-holed sphere then there are no splittings rel $\multimot$. The rJSJ is
trivial, and we are done. 

If $(\F,\multimot)$ is a QH--surface other than a three-holed sphere then splittings rel $\multimot$ come from
essential, non-peripheral simple closed curves on the
surface. For any such curve we can find another intersecting it,
giving us an incompatible splitting. 
Therefore, no such splitting is universal, so the rJSJ is trivial, and
we are done.

If we are not in either of these cases then by
\fullref{corollary:uncrossedexist} there
exists a cut point or an uncrossed cut pair in $\decompw$.

Take the uncrossed collection $\{S_i\}_{i\in I}$ consisting of all cut points and uncrossed
cut pairs, and let $\Gamma$ be the graph of groups provided by \fullref{proposition:multipleotal}.

The tree $\mathrm{BS}(\Gamma)$ is canonically defined by the
topology of the decomposition space, and the $\F$--action is induced by
the $\F$--action on $\decomp$, so the resulting graph of groups decomposition is canonical. 

We will show $\Gamma$ satisfies conditions (\ref{item:bipartite})-(\ref{item:maximal}).
Conversely, we will show that any
 $\Gamma'$ satisfying conditions (\ref{item:bipartite})-(\ref{item:maximal}) has Bass-Serre
tree $\mathrm{BS}(\Gamma')$ equivariantly isomorphic to
$\mathrm{BS}(\Gamma)$, so $\Gamma$ and $\Gamma'$ are equivalent graph
of groups decompositions.

Condition (\ref{item:bipartite}) says that the graph of groups is
normalized as in \fullref{sec:normalization}.
Using the facts that $\F$ is free and that the cyclic vertex groups of $\Gamma$ are maximal
cyclic subgroups of $\F$ it is easy to see that $\Gamma$ satisfies
these conditions.

Uncrossed splittings are universal by
\fullref{lemma:uncrossedsplittingsareuniversal}, so $\Gamma$ satisfies
condition (\ref{item:universal}).

If for some non-cyclic vertex $G$ of $\Gamma$ the pair
$(G,\indg)$ is neither rigid nor a QH--surface
then by \fullref{lemma:refinement} there is a refinement of $\Gamma$
coming from a larger uncrossed collection.
This is absurd; we have already included all cut points and uncrossed
cut pairs in our uncrossed collection.
Together with condition (\ref{item:universal}) this implies
that any refinement of $\Gamma$ must come from splitting a QH--surface vertex
group. 
The resulting splitting would not be universal, because if
there is a way to split the surface then there
is always an incompatible way to split it.
Thus, condition (\ref{item:maximal}) is satisfied.

Now suppose $\Gamma'$ is another graph of groups decomposition of $\F$ rel
$\multimot$ satisfying conditions (\ref{item:bipartite})-(\ref{item:maximal}). 
Condition (\ref{item:basicpieces}) must be satisfied or else it
would be possible to refine $\Gamma'$ in a universal way,
contradicting maximality.

Consider a cyclic vertex group $\left< g\right>$ of $\Gamma'$.
We would like to show $S=\qmap(\{\inv{g}^\infty,g^\infty\})$ is a cut
point or uncrossed cut pair in $\decompw$.
The number of components of
$\decompw\setminus S$ is equal to
the sum of the degrees of the edge maps into the vertex group.
By condition (\ref{item:bipartite}), this is at least two, so
$S$ is a cut point if $g\in\multimot$ or cut pair otherwise. 

If the sum of the degrees of the
edge maps is greater than two, or if the sum of the degrees is equal to two and one of the adjacent non-cyclic
vertices is rigid, then $S$ is an uncrossed cut pair, by \fullref{lemma:rigidboundary}.

Otherwise, either the vertex separates two QH-surfaces glued along boundary
curves or it is adjacent to one QH--surface and the edge maps into the
cyclic vertex group with degree 2. 
In the first case, the cyclic vertex can be removed by gluing together
the two QH--surfaces to give a
larger QH--surface. 
In the second case, the cyclic vertex can be removed by gluing a
M\"obius strip to the corresponding QH--surface along their boundary curves.
In either case, the new surface contains an
essential, non-peripheral simple closed curve that intersects the curve we just glued
along. This would provide a splitting of $\F$ rel $\multimot$
incompatible with $\Gamma'$, contradicting universality. 
Thus, each
cyclic vertex group of $\Gamma'$ is the stabilizer of a cut point or
uncrossed cut pair.
Furthermore, the cyclic vertex groups account for all the cut points
and uncrossed cut pairs, since $\Gamma'$ is maximal.

$\mathrm{BS}(\Gamma)$ and $\mathrm{BS}(\Gamma')$ are both equivariantly
isomorphic to the tree constructed in
\fullref{proposition:multipleotal}, hence, to each other.
\end{proof}

\section{Virtually Geometric Multiwords}\label{sec:vg}
A \emph{handlebody} is a 3--manifold obtained by gluing 1--handles
to a 3--ball.
These are commonly imagined as thickened graphs, although we do not
assume orientability.

A multiword $\underline{w}=\{w_1,\dots, w_k\}$ in $\F=\F_n$ is \emph{geometric} if there exists a handlebody $H$ with fundamental group $\F$ such that the conjugacy classes of the $w_i$ can be represented by an embedded multicurve in the boundary of $H$.
The multiword is \emph{virtually geometric} if it becomes geometric
upon lifting to a finite index subgroup of $\F$. (Recall \fullref{definition:lift})

\begin{example}\label{ex:bs21}
  $\multimot=\{\inv{a}^2\inv{b}ab\}$ in $\F_2=\left<a,b\right>$

\begin{figure}[h!]
\begin{minipage}[b]{0.465\textwidth}
\labellist
\small
\pinlabel $a$ [c] at 30 102
\pinlabel $\inv{a}$ [c] at 30 28
\pinlabel $\inv{b}$ [c] at 147 102
\pinlabel $b$ [c] at 147 28
\tiny
\pinlabel $2$ at 34 111.5
\pinlabel $4$ at 40 106.5
\pinlabel $1$ at 40 97.5
\pinlabel $2$ at 34 39
\pinlabel $4$ at 40 33
\pinlabel $1$ at 40 24
\pinlabel $5$ at 140 110
\pinlabel $3$ at 136 101.5
\pinlabel $3$ at 136 28.5
\pinlabel $5$ at 140 20
\endlabellist
  \centering
  \includegraphics[height=1.4in]{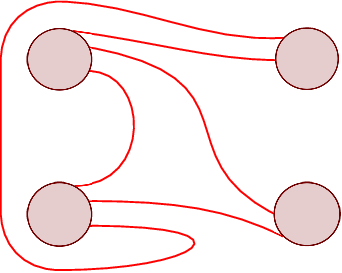}
  \caption{$\Wh(\bp)\{\inv{a}^2\inv{b}ab\}$}
  \label{fig:bs12heegaard}
\end{minipage}
\hfill
\begin{minipage}[b]{0.525\textwidth}
\labellist
\small
\pinlabel $a$ [b] at 14 122
\pinlabel $\inv{a}$ [t] at 14 1
\pinlabel $b$ [l] at 179 60
\endlabellist
  \centering
\includegraphics[height=1.2in]{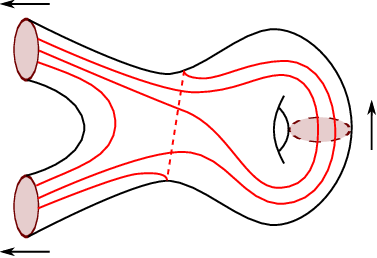}
  \caption{Corresponding handlebody for $\inv{a}^2\inv{b}ab$}
  \label{fig:bs12handlebody}
\end{minipage}
\end{figure}

In \fullref{fig:bs12heegaard} we have a Whitehead graph with vertices blown up to discs and a
numbering around each vertex.
Following the edges according to the numbering reads off the word $\inv{a}^2\inv{b}ab$.
Embed this graph on the surface of a three-ball.

The numbering around the $b$--disc can be read $3-5$ going counterclockwise,
and the numbering around the $\inv{b}$--disc can be read $3-5$ going
clockwise, so we can glue these discs together, matching the
numbering, to make an orientable $b$--handle.

The numbering around both the $a$--disc and $\inv{a}$--disc is
$1-4-2$ going counterclockwise.
We can glue these discs together and match the numbering, but we
must do it in a non-orientable way.
In drawing the corresponding handlebody in
\fullref{fig:bs12handlebody} we leave the $a$--disc and $\inv{a}$--disc
apart, but one should imagine that they have been identified to create
a non-orientable $a$--handle.
\end{example}

Similarly, any multiword with a Whitehead graph that is planar and
valence at most three is geometric. The same argument shows, more generally:

\begin{proposition}[\cite{Cis65}]\label{proposition:gooddiagramimpliesgeometric}
  If there exists a Whitehead graph of $\multimot$ with a planar embedding such
  that the cyclic orderings of edges incident to inverse vertices are
  consistent (either the same or opposite), then $\multimot$ is geometric.
\end{proposition}

In fact, more is true. 
If a multiword is geometric then every minimal Whitehead graph of
$\multimot$ has
 a planar embedding such
  that the cyclic orderings of edges incident to inverse vertices are
  consistent.

Thus, there is an algorithm to determine geometricity:
The Whitehead graph is finite, so it is possible to check by brute
force to see if
there exists a planar embedding that
respects cyclic ordering around the vertices.

These claims follow from work of  Zieschang \cite{Cis65} (see also Berge's
{\itshape Documentation for the program Heegaard} \cite{Berge}).

There is a positive algorithm for checking virtual geometricity of
$\multimot$ by
enumerating subgroups of successively larger index, computing the lift
of $\multimot$, and then checking for geometricity. 
There is not an obvious bound on how large an index is necessary for
a given multiword, so there is not a corresponding negative algorithm.
Gordon and Wilton \cite{GorWil10} even asked whether every one element
multiword is virtually geometric.
Manning \cite{Man10} answered in the negative by showing that the word
$w=bbaaccabc$ in
$\F_3=\left<a,b,c\right>$ is not virtually geometric.

Otal \cite{Ota92} notes that a free splitting of the free group corresponds to a
connected sum of the corresponding handlebodies, so, as usual, we will
confine our attention to the case that $\F$ does not split freely rel $\multimot$.

\subsection{Rigid Multiwords and Geometricity}{\label{sec:rigidgeometric}
\begin{lemma}[{\cite[Proposition 0]{Ota92}}]\label{lemma:geometricimpliesplanar}
If $\multimot$ is geometric then $\decomp$ is planar.
\end{lemma}
We include the proof for completeness:
\begin{proof}
Realize $\multimot$ by an embedded
multicurve on the surface of a handlebody, which lifts to a collection of disjoint arcs on the
boundary surface of the universal cover of the handlebody. 
The universal cover is a thickened tree, and may be compactified by
including the Cantor set boundary of the tree. 
The resulting space is a 3--ball with a collection of disjoint arcs
in the bounding 2--sphere. 
By Moore's Decomposition Theorem (\fullref{theorem:Moore}), collapsing each of these arcs to a point
sends the 2--sphere to the 2--sphere, and
the image of the Cantor
set in the quotient is the decomposition space, embedded, non-surjectively, into $S^2$.
\end{proof}

Inclusion of a finite index subgroup induces a homeomorphism of
decomposition spaces, so:
\begin{corollary}\label{corollary:vgimpliesplanar}
If $\multimot$ is virtually geometric then $\decomp$ is planar.
\end{corollary}

\begin{theorem}[{cf \cite[Theorem 1]{Ota92}}]\label{theorem:rigidvg}
Let $(\F,\underline{w})$ be rigid. The following are equivalent:
\begin{enumerate}
\item The multiword $\underline{w}$ is geometric.\label{item:rigidgeometric}
\item The decomposition space $\decomp_{\underline{w}}$ is planar.\label{item:rigidplanar}
\item Every minimal Whitehead graph for $\underline{w}$ has an
  embedding in the plane with
  consistent cyclic orderings of edges incident to inverse vertices.\label{item:planarwh}
\end{enumerate}
\end{theorem}
\begin{proof}
 \fullref{lemma:geometricimpliesplanar} shows \mbox{(\ref{item:rigidgeometric}) implies (\ref{item:rigidplanar})}.

\mbox{(\ref{item:rigidplanar}) implies (\ref{item:planarwh})} is the
content of \cite[Lemma 4.4]{Ota92}.
The hypotheses for this lemma are that every element of the multiword
is `indecomposable' and that the decomposition space embeds
into $S^2$ in such a way that closures of the complementary regions intersect
pairwise in at most one point.
The first hypothesis is unnecessarily strong. 
Indecomposability of each element of the multiword
is only used to prove that the decomposition space has no cut points.
The second hypothesis is satisfied if the decomposition space has no
cut pairs. Therefore, rigidity is a sufficient hypothesis.

\mbox{(\ref{item:planarwh}) implies
  (\ref{item:rigidgeometric})} by \fullref{proposition:gooddiagramimpliesgeometric}.
\end{proof}

\begin{corollary}\label{corollary:rigidnotgimpliesnotvg}
Virtual geometricity implies geometricity for rigid multiwords.
\end{corollary}

\subsection{Non-rigid Multiwords and Virtual Geometricity}
In this section we prove the \vgthm (\fullref{theorem:planarequalsvg}).
We first (\fullref{theorem:vgeq}) prove the theorem in the special case that the decomposition
space has no uncrossed cut pairs.
Given \fullref{theorem:vgeq}, the proof of
\fullref{theorem:planarequalsvg} amounts to showing that if the
decomposition space is planar then uncrossed cut pairs can be pinched
to cut points without making the space non-planar.

\begin{theorem}\label{theorem:vgeq}
Assume that $\decompw$ is connected with no uncrossed cut pairs. 
Let $\Gamma$ be the JSJ decomposition of $\F$ relative to $\multimot$.
The following are equivalent:
\begin{enumerate}
\item The multiword $\multimot$ is virtually geometric.\label{item:vg}
\item The decomposition space $\decompw$ is planar.\label{item:planar}
\item $\indg$ is geometric for every non-cyclic vertex group $G$ of $\Gamma$.\label{item:piecesgeometric}
\end{enumerate}
\end{theorem}
\begin{proof}
\fullref{corollary:vgimpliesplanar} shows \mbox{(\ref{item:vg})
  implies (\ref{item:planar})}.

Let $G$ be a non-cyclic vertex group of $\Gamma$.
Since there are no uncrossed cut pairs, $\multimot=\augw$, so
\fullref{lemma:embed} shows that the decomposition space
$\decomp_{\indg}$ embeds into $\decompw$.
Thus,  if
$\decomp_{\indg}$ is non-planar then $\decompw$ is non-planar
as well. 
Therefore, (\ref{item:planar}) implies that $\decomp_{\indg}$ is
planar.
If $(G,\indg)$ is a QH--surface then $\indg$ is geometric.
If $(G,\indg)$ is rigid and $\decomp_{\indg}$ is planar 
\fullref{theorem:rigidvg} says $\indg$ is geometric.
Thus, \mbox{(\ref{item:planar}) implies (\ref{item:piecesgeometric})}.

Now assume (\ref{item:piecesgeometric}).
From a graph of groups we may build a corresponding graph of spaces \cite{ScoWal79}:
For each vertex group choose a vertex space with fundamental group isomorphic
to the vertex group.
For each edge group choose a space with fundamental group isomorphic
to the edge group, and let the edge space be the product of that space
with the unit interval. 
Use the edge injections of the graph of groups to define attaching
maps of edge spaces to the corresponding vertex spaces. 
The resulting space will have fundamental group isomorphic to the
fundamental group of the graph of groups.

For each non-cyclic vertex group, the induced multiword is geometric,
so we can choose the vertex space to be a handlebody with an embedded
multicurve in the boundary representing the induced multiword.

For the edge spaces we use annuli. 
Later we will want to
thicken them to make the resulting graph of spaces a 3--manifold.

For the moment we will also make a geometricity assumption on the
cyclic vertex groups. Suppose one of the following possibilities are
true for each cyclic vertex group $\left<g\right>$:
\begin{itemize}
\item There are some number $k$ of incident edges and
each edge injection is degree one. 
In this case we choose the vertex space to be a solid torus with $k+1$
disjoint curves on the boundary, one representing the element $g$ and
$k$ to be attaching curves
to which we will glue  boundary curves of annulus edge spaces. 
\item There are some number $k$ of incident edges and the degrees of the edge injections are all
two except for possibly one of degree one. In this case we choose the
vertex space to be a solid Klein bottle, and again we have $k+1$ disjoint
curves on the boundary representing $g$ and the attaching curves.
\end{itemize}

The resulting graph of spaces has fundamental group $\F$ and has an embedded
multicurve representing $\multimot$ such that the multicurve is
disjoint from the edge spaces.
It  is not yet a 3--manifold with boundary; we need to
fatten the annuli. 
To see if this is possible, consider for each boundary component of
each annulus a small tubular neighborhood of the attaching curve in the
boundary of the corresponding handlebody.
If for each annulus the two neighborhoods are either both annuli or both
M\"{o}bius strips then the annuli may be fattened to make the graph of
spaces a 3--manifold.
Now, a fattened annulus is composed of a 1--handle and a 2--handle, so this
does not explicitly give the resulting space a handlebody structure.
However, a graph of aspherical spaces is aspherical \cite{ScoWal79},
and a compact aspherical 3--manifold with free fundamental group is a
handlebody \cite{Hem76}, so this space really is a handlebody, and  $\multimot$ is geometric. 

Thus, assuming (\ref{item:piecesgeometric}), there are two possible obstructions to geometricity:
\begin{enumerate}[(a)]
\item The degrees of the edge injections into some cyclic vertex group
  are not of one of the two forms described above.\label{ob:degree}
\item Some annulus cannot be fattened because one boundary
  neighborhood is an annulus and the other is a M\"{o}bius strip. \label{ob:mobius}
\end{enumerate}

\begin{claim}
  These obstructions vanish in a finite index subgroup of $\F$.
\end{claim}
Lift $\multimot$ to this finite index subgroup, and then apply the
graph of spaces construction to see \mbox{$(\ref{item:piecesgeometric})\implies(\ref{item:vg})$}.

\begin{claimproof}
There are finitely many elements $g_i\in\underline{w}$ such that
$\qmap(\{\inv{g}_i^\infty,g_i^\infty\})$ is a cut point in
$\decomp_{\underline{w}}$.

From the proof of \fullref{proposition:multipleotal}, an edge injection of degree greater
than one into a cyclic
vertex group $\left<g_i\right>$ occurs when the $g_i$--action permutes some
components of $\decomp_{\underline{w}}\setminus
\qmap(\{\inv{g}_i^\infty,g_i^\infty\})$.
There are only finitely many components, so there exists some minimal positive
power $a_i$ of $g_i$ such that the $g_i^{a_i}$--action fixes each
complementary component.
Additionally, if some edge incident to the $\left<g_i\right>$ vertex
attaches to a handlebody around a non-orientable handle, and if $a_i$ is odd, then consider $g_i^{2a_i}$.

 Marshall Hall's Theorem implies there exists a finite index subgroup
 $H_i$ of $\F$ in which $g_i^{a_i}$ (alternatively, $g_i^{2a_i}$) generates a free factor.
Let $H$ be the finite index subgroup $\cap_iH_i$.
If we apply the \decompthm to $H$ we get a graph of groups covering
the graph of groups decomposition for $\F$. 
By construction, the smallest power of $g_i$ in $H$ is a multiple of
$g_i^{a_i}$, so all edge inclusions are degree one.
This takes care of obstruction (\ref{ob:degree}), and we can choose
all the cyclic vertex spaces to be solid tori.

Furthermore, we can take the vertex spaces to be handlebodies finitely
covering the original handlebodies. 
If some attaching curve in the original decomposition ran along a
M\"{o}bius strip then it runs along an even covering of the M\"{o}bius
strip in the covering handlebodies.
Thus, all attaching curves have annulus neighborhoods, which
takes care of obstruction (\ref{ob:mobius}).
\end{claimproof}
\end{proof}

We would now like to show that if $\decomp_{\multimot}$ is planar then
$\decomp_{\augw}$ is planar.
To get $\decomp_{\augw}$ from $\decomp_{\multimot}$ we must pinch
uncrossed cut pairs to points.
To make sure planarity is preserved we first embed
$\decomp_{\multimot}$ in a sphere and then find an upper
semi-continuous collection of arcs so that collapsing the arcs
achieves the pinching of the uncrossed cut pairs. 
It will suffice to find such a collection of arcs for an arbitrary
vertex group of the \rjsj:
\begin{lemma}\label{planarclaim2}
If  $\decompw$ is planar
 then for each non-cyclic vertex group $G$ of the \rjsj, the
 decomposition space
 $\decomp_{\indg}$ is planar.
\end{lemma}

\begin{proof}
Let $\chi\from\decompw\into S^2$ be an embedding.
Let $\{S_i\}_{i\in I}$ be the collection of uncrossed cut pairs of
$\decomp_{\multimot}$ in $\qmap(\bdry G)$.
Let $S_i=\{x_{i,0},x_{i,1}\}$.
Since $G$ is a vertex group of the \rjsj, for each $i\in I$ all of $\qmap(\bdry G)\setminus S_i$ is
contained in a single complementary component $C_i$ of $S_i$, for
otherwise we could use
\fullref{lemma:refinement} to find a refinement of the \rjsj. 
There is at least one other complementary component $C_i'$ of $S_i$.
By \fullref{corollary:arcconnectedcomponents} and
\fullref{lemma:limitpoints}, there exists an arc
$A_i\subset\closure{C_i'}$ connecting the two points of $S_i$.
Note that $A_i\cap\qmap(\bdry G)=S_i$.
Recalling the terminology of \fullref{sec:decompositionspaces}:
\begin{claim}\label{claim:arcsusc}
  The sets $\chi(A_i)$ are the non-degenerate elements of an upper semi-continuous decomposition of $S^2$.
\end{claim}
Assuming the claim, Moore's Decomposition Theorem (\fullref{theorem:Moore}) then says that the quotient of
the sphere obtained by collapsing each of these arcs to a point is again
the sphere.
The image of $\chi(\qmap(\bdry G))$ in this quotient is
$\decomp_{\indg}$ embedded, non-surjectively, in $S^2$.
Thus, $\decomp_{\indg}$ is planar, and the lemma is proven.

\begin{claimproof}[Proof of \fullref{claim:arcsusc}]
Essentially the proof is that each $S_i$ separates $A_i\setminus S_i$
from all of the other arcs, so arcs can only be close at their
endpoints, and we know the endpoints are well behaved because the
boundary pattern of a multiword gives an upper semi-continuous
decomposition of the boundary of the free group.

Formally, fix an $i\in I$ and let $U$ be an open neighborhood of $\chi(A_i)$ in
$S^2$.
We will produce an open neighborhood $V$ of $\chi(A_i)$
such that $\chi(A_j)\cap V\neq\emptyset$ implies $\chi(A_j)\subset U$
for all $j\in I$.

By \fullref{lemma:uncrossedaxis}, for each $j$ and $k$ the set $\qmap^{-1}(x_{j,k})$ is a single point; call
it $\xi_{j,k}$.
Let $U'$ be an open neighborhood of $\chi(A_i)$ in $S^2$ such that
$\closure{U'}\subset U$.
\begin{claim}
$J=\{j\in I\mid  \chi(S_j)\subset U'\text{ and }\chi(A_j)\not\subset U\}$ is finite.
\end{claim}
\begin{claimproof}
  Suppose not. For each $j\in J$ choose a point $\chi(y_j)\in \chi(A_j)\setminus
  U$.
There is some $\sigma\from\mathbb{N}\into J$ such that
$(\chi(y_{\sigma(k)}))$ is a convergent sequence, converging to a point $\chi(y)\in \chi(\decompw)\setminus U$.
By \fullref{corollary:connected}, $y$ has a neighborhood basis in $\decompw$ of
connected neighborhoods $\nbhd(y,r)$.
Thus, given $r$ there is a $K$ so that for all $k>K$ we have
$y_{\sigma(k)}\in\nbhd(y,r)$.
But $S_{\sigma(k)}$ is a cut set in $\decompw$ separating $y_{\sigma(k)}$ from all
of the other $y_j$, so if $y_{\sigma(k)}$ and $y_j$ are both contained
in the connected set $\nbhd(y,r)$ then so is at least one of the
points of $S_{\sigma(k)}$.
Thus, $y$ is a limit point of $\cup_{j\in J}S_j$, so $\chi(y)$ is a
limit point of $\chi(\cup_{j\in J}S_j)$.
This is impossible, since $\cup_{j\in J}\chi(S_j)\subset U'\subset \closure{U'}$ and
$\{\chi(y_j)\}_{j\in J}\subset S^2\setminus U$ are contained in disjoint closed sets.
\end{claimproof}

Let $U''=U'\setminus (\cup_{j\in J}\chi(A_j))$. 
This is an open neighborhood
of $\chi(A_i)$ contained in $U$.
By \fullref{proposition:usc}, the decomposition of $\bdry\tree$ whose
non-degenerate elements are elements of the boundary pattern for
$\augw$ is upper semi-continuous.
Thus, for the neighborhood $\mathcal{U}''=\qmap^{-1}(\chi^{-1}(\chi(\decompw)\cap
U''))\subset\bdry\tree$ of $\{\xi_{i,0},\xi_{i,1}\}$ there exists an open neighborhood
$\mathcal{V}\subset\mathcal{U}''$ of $\{\xi_{i,0},\xi_{i,1}\}$ so that if for some $j$ and $k$ we have $\xi_{j,k}\in\mathcal{V}$ then $\xi_{j,1-k}\in\mathcal{U}''$.
As in
\fullref{sec:nbhdbasis}, we may assume $\mathcal{V}$ is saturated and
$\qmap(\mathcal{V})$ 
has two components, one containing $x_{i,0}$ and the other containing $x_{i,1}$.
Since $\chi$ is an embedding, there exists an open set $V'\subset S^2$
such that $V'\cap\chi(\decompw)=\chi(\qmap(\mathcal{V}))$.

  Let $W=S^2\setminus \chi(C_i)$, and let $V=U''\cap (V'\cup W)$.

Suppose, for some $j\neq i$, that $\chi(A_j)\cap V\neq\emptyset$.
Since $C_i$ is the complementary component of $S_i$ containing $A_j$, we have $\chi(A_j)\cap W=\emptyset$, so
$A_j\cap\qmap(\mathcal{V})\neq\emptyset$.
On the other hand, $S_j$ separates $S_i$ from $A_j\setminus S_j$, and
$\qmap(\mathcal{V})$ consists of connected neighborhoods of the two
points of $S_i$, so if $A_j\cap\qmap(\mathcal{V})\neq\emptyset$ then one of the 
$x_{j,k}$ must be in $\qmap(\mathcal{V})$. 
By definition of $\mathcal{V}$, this means that
$\{\xi_{j,0},\xi_{j,1}\}\subset\mathcal{U}''$.
By definition of $\mathcal{U}''$, this means that $\chi(A_j)\subset U$.
\end{claimproof}\hfill\end{proof}

\begin{theorem}[\vgthm]\label{theorem:planarequalsvg}
Assume that $\decompw$ is connected. 
Let $\Gamma$ be the JSJ decomposition of $\F$ relative to $\multimot$.
The following are equivalent:
\begin{enumerate}
\item The multiword $\multimot$ is virtually geometric.\label{item:vg2}
\item The decomposition space $\decompw$  is planar.\label{item:planar2}
\item For every non-cyclic vertex group $G$ of $\Gamma$, the
  induced multiword $\indg$ is geometric.\label{item:piecesgeometric2}
\end{enumerate}
Thus, virtually geometric multiwords are those that are built
from geometric pieces.
\end{theorem}

\begin{proof}

\fullref{corollary:vgimpliesplanar} shows \mbox{(\ref{item:vg})
  implies (\ref{item:planar})}.

If the decomposition space is planar then 
\fullref{planarclaim2} shows that the induced decomposition spaces of
each of the vertex groups in the \rjsj are planar.
\fullref{theorem:rigidvg} and the fact that
QH--surface multiwords are always geometric show that all of the
induced multiwords are geometric. Thus, \mbox{(\ref{item:planar})
  implies (\ref{item:piecesgeometric2})}.

By \fullref{theorem:vgeq}, if the induced multiword in each non-cyclic vertex group of
the \rjsj is geometric then $\augw$ is virtually
geometric.
This implies $\multimot$ is virtually geometric, since
$\multimot$ is a subset of $\augw$, so  \mbox{(\ref{item:piecesgeometric2})
  implies (\ref{item:vg})}.
\end{proof}


\subsection{Examples}

\subsubsection{Baumslag's Word}
$w=\inv{a}^2\inv{b}\inv{a}ba\inv{b}ab$ in $\F_2=\left<a,b\right>$ is
known as \emph{Baumslag's word}.
In response to a question of Gordon and Wilton, Manning showed, by
brute force,  that
this word becomes geometric in an orientable handlebody with
fundamental group an index four subgroup of $\F_2$.

The \rjsj for $\F=\left<a,b\right>\cong\left<a,b,c\mid c=\inv{b}ab\right>$ 
is shown in \fullref{fig:longword}. 

\begin{figure}[h]
\labellist
\small
\pinlabel $\left<a,c\right>$  at -4 23
\pinlabel $\left<a\right>$ at 84 23
\pinlabel $a$ [r] at 2 11
\pinlabel $c$ [r] at 2 36
\pinlabel $a$ [l] at 80 11
\pinlabel $a$ [l] at 80 36
\endlabellist  
\centering
  \includegraphics[scale=.8]{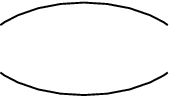}
  \caption{\rjsj-Decomposition of $\left<a,b\right>$ for $\inv{a}^2\inv{b}\inv{a}ba\inv{b}ab$ ($c=\inv{b}ab$)}
  \label{fig:longword}
\end{figure}

The word $w$ becomes $\inv{a}^2\inv{c}ac$ when rewritten in the rank two vertex
group, so the induced multiword is $\{\inv{a}^2\inv{c}ac,a,c\}$. 
One can check that this multiword is rigid, so this is the \rjsj. 
(Checking rigidity takes some work, using techniques of \cite{CasMac11}.)

In \fullref{fig:heegaard} we have a reduced Whitehead graph/Heegaard
diagram for the induced multiword that shows it is geometric.
\fullref{fig:nonorientablehandlebody} shows a (non-orientable) handlebody with embedded
multicurve representing $\{\inv{a}^2\inv{c}ac,a,c\}$.

\begin{figure}[h]
\begin{minipage}[b]{0.465\textwidth}
\labellist
\small
\pinlabel $a$ [c] at 30 102
\pinlabel $\inv{a}$ [c] at 30 28
\pinlabel $\inv{c}$ [c] at 147 102
\pinlabel $c$ [c] at 147 28
\tiny
\pinlabel $2$ at 34 111.5
\pinlabel $4$ at 40 106.5
\pinlabel $1$ at 40 97.5
\pinlabel $6$ at 34 91.5
\pinlabel $2$ at 34 39
\pinlabel $4$ at 40 33
\pinlabel $1$ at 40 24
\pinlabel $6$ at 34 18
\pinlabel $5$ at 140 110
\pinlabel $3$ at 136 101.5
\pinlabel $7$ at 140 94
\pinlabel $7$ at 140 37
\pinlabel $3$ at 136 28.5
\pinlabel $5$ at 140 20
\endlabellist
  \centering
  \includegraphics[height=1.4in]{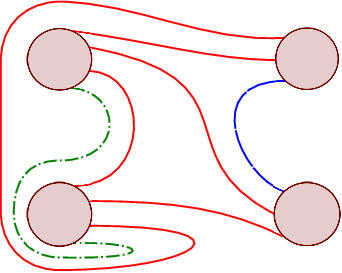}
  \caption{Whitehead graph/ Heegaard diagram}
  \label{fig:heegaard}
\end{minipage}
\hfill
\begin{minipage}[b]{0.525\textwidth}
\labellist
\small
\pinlabel $a$ [b] at 14 122
\pinlabel $\inv{a}$ [t] at 14 1
\pinlabel $c$ [l] at 179 60
\endlabellist
  \centering
\includegraphics[height=1.4in]{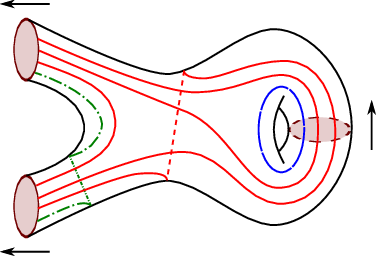}
  \caption{Corresponding non-orientable handlebody for $\{\inv{a}^2\inv{c}ac,a,c\}$}
  \label{fig:nonorientablehandlebody}
\end{minipage}
\end{figure}

The obstruction to geometricity of $w$ is that the curve representing
$c$ runs around a orientable handle, while the curve representing
$a$ does not.
We cannot achieve the conjugation of $a$ to $c$ by a fattened annulus. 

To correct this problem, pass to the index two subgroup
$G=\left<A=a^2, b, B=ab\inv{a}\right>$.

After applying the automorphism that sends $B$ to $bB$ and fixes
$b$ and $A$,
the image of $w^2$ is $\inv{A}(\inv{b}\inv{A}b)BAB\inv{A}\inv{B}^2(\inv{b}Ab)$.

The splitting over $\left<A\right>$ is an HNN extension with 
 $b$ conjugating $A$ to $C$, and
the induced multiword in the vertex group $\left<A,B,C\right>$ is
$\{A,C, \inv{A}\inv{C}BAB\inv{A}\inv{B}^2C\}$, which is geometric in a non-orientable handlebody,
as seen in \fullref{fig:nonorientableindex2baumslag}.

\begin{figure}[h]
\labellist
\small
\pinlabel $A$ [b] at 78 190
\pinlabel $C$ [t] at 78 3
\pinlabel $\inv{B}$ [b] at 241 170
\pinlabel $B$ [t] at 241 20
\endlabellist
  \centering
 \includegraphics[scale=.6]{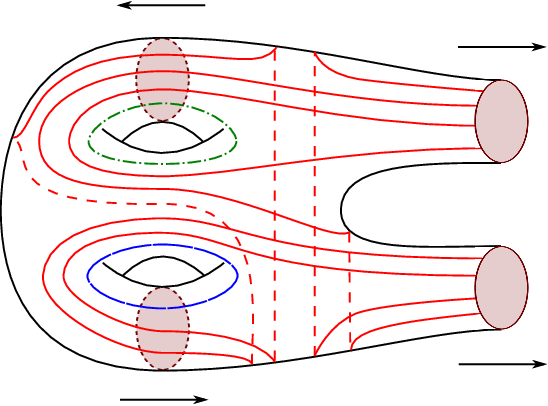} 
  \caption{A non-orientable handlebody for $\{A,C, \inv{A}\inv{C}BAB\inv{A}\inv{B}^2C\}$}
  \label{fig:nonorientableindex2baumslag}
\end{figure}

This time we can build a 3--manifold
graph of spaces because we only need to conjugate words that run
around orientable handles.
Gluing on a fattened annulus conjugating $A$ to $C$ gives a
non-orientable handlebody with fundamental group isomorphic to $G$ for
which the image of
$w^2$ is geometric.

\subsubsection{Baumslag-Solitar Words}\label{sec:bswords}
Another interesting family of examples is given by the
Baumslag-Solitar words $w_{p,q}=\inv{a}^q \inv{b}a^pb$ in
$\F_2=\left<a,b\right>$.
We will assume that $0< p\leq q$.
Gordon and Wilton \cite{GorWil10} have shown that $w_{p,q}$ is
virtually geometric when $p$ and $q$ are relatively prime.

The decomposition space associated to $w_{p,q}$ is connected without cut points.
The pair $\qmap(\{\inv{a}^{\infty},a^\infty\})$ is a cut pair.
$\Wh([\inv{a}^\infty,a^\infty])$ has components
$\wc^+_j=\{\gr(a^ib)\mid i\equiv j \mod p\}$ for $0\leq j< p$ that are cyclically
permuted by the $a$--action, and components $\wc^-_j=\{\gr(a^i\inv{b})\mid i\equiv j \mod q\}$ for $0\leq j< q$ that are cyclically
permuted by the $a$--action.

The  case $p=q=1$ is special; in this case the Whitehead graph is
a circle, which implies the decomposition space is a circle and the
word is geometric.

Otherwise, the number of complementary components is $p+q>2$, so that
$\qmap(\{\inv{a}^\infty, a^\infty\})$ is an uncrossed cut pair.
The \rjsj for this case is shown in \fullref{fig:baBA2}.

\begin{figure}[h]
\labellist
\small
\pinlabel $\left<\inv{b}a^pb,a^q\right>$ [c] at -4 23
\pinlabel $\left<a\right>$ [c] at 84 23
\pinlabel $a^q$ [r] at 2 11
\pinlabel $\inv{b}a^pb$ [r] at 2 36
\pinlabel $a^q$ [l] at 80 11
\pinlabel $a^p$ [l] at 80 37
\endlabellist  
\centering
  \includegraphics[scale=.8]{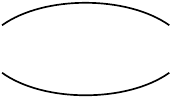}
  \caption{\rjsj-Decomposition of $\left<a,b\right>$ for $\inv{a}^q\inv{b}a^pb$}
  \label{fig:baBA2}
\end{figure}

The rank two vertex group is $\left<A=a^q, C=\inv{b}a^pb\right>$, and
the induced multiword in this vertex group is $\{A,C,\inv{A}C\}$. 
The Whitehead graph for this multiword is a circle, which implies the vertex
decomposition space is a circle and the induced multiword is geometric.
Thus, \fullref{theorem:planarequalsvg} says $w_{p,q}$ is at least virtually geometric.

The cyclic vertex group has edge inclusions of degrees $p$ and $q$. 

If $p=1$ and $q=2$ we can make this geometric by using a solid Klein
bottle for the cyclic vertex space.
(We saw the non-cyclic vertex space for this example back in \fullref{ex:bs21}.)

If $p=q$ the word is also geometric, because two disjoint degree $p$
curves fit into the boundary of a solid torus.

In all other cases, the word $w_{p,q}$ is not geometric.
Virtual geometricity can be verified by passing to the index $m=\mathrm{lcm}(p,q)$ subgroup:
\[G=\left<A, B_0, B_1,\dots, B_{m-1}\mid A=a^{m},
  B_i=a^ib\inv{a}^i\right>\]


\bibliographystyle{hypershort}
\bibliography{split}
\end{document}